\documentclass[onecolumn,final]{elsart3p}
\usepackage{comment}
\usepackage{amsmath}       
\usepackage{amssymb}       
\usepackage{chapterbib}    
\usepackage{color}
\usepackage{comment}
\usepackage{bm,bbm}
\usepackage{overpic}
\usepackage{times}
\usepackage{epsfig}
\usepackage{graphicx,graphics,rotating}
\usepackage{subfigure}
\usepackage{multicol}
\usepackage{multirow}

\theoremstyle{plain}
\newtheorem{theorem}{Theorem}[section]
\newtheorem{lemma}[theorem]{Lemma}

\theoremstyle{definition}

\newtheorem{remark}[theorem]{Remark}

\theoremstyle{plain}

\definecolor{MyDarkGreen}{rgb}{0,0.45,0}

\def\trait #1 #2 #3 {\vrule width #1pt height #2pt depth #3pt}
\def\fin{\hfill
        \trait .3 5 0
        \trait 5 .3 0
        \kern-5pt
        \trait 5 5 -4.7
        \trait 0.3 5 0
\medskip}

\newenvironment{proof}{\textit{Proof.}}{\fin}

\newcommand{\VEM}{\text{VEM}}

\newcommand{\KER} {\textrm{ker}}
\newcommand{\INTP}{\footnotesize{\texttt{I}}}
\newcommand{\SPAN}[1]{\mbox{\textrm{span}}\big\{#1\big\}}
\newcommand{\REAL}{\mathbbm{R}}

\newcommand{\ASSUM}[1]{\textbf{(#1)}}

\renewcommand{\cv}{\mathbf{c}}

\newcommand{\nv}{\mathbf{n}}

\newcommand{\vv}{\mathbf{v}}
\newcommand{\wv}{\mathbf{w}}
\newcommand{\xv}{\mathbf{x}}
\newcommand{\yv}{\mathbf{y}}

\newcommand{\as}{a}
\newcommand{\bs}{b}

\newcommand{\fs}{f}

\newcommand{\ms}{m}

\newcommand{\qs}{q}

\newcommand{\us}{u}
\newcommand{\vs}{v}
\newcommand{\ws}{w}
\newcommand{\xs}{x}

\newcommand{\HONE}  {H^1}
\newcommand{\HONEzr}{H^1_0}
\newcommand{\HONEnc}{H^{1,nc}}

\newcommand{\LTWO}  {L^2}

\newcommand{\HS}[1] {H^{#1}}

\newcommand{\VS}[1] {V^{#1}}
\newcommand{\Vspace}{\VS{}}

\newcommand{\PS}[1] {\mathbbm{P}_{#1}}

\renewcommand{\P} {\textsf{P}}            
\newcommand  {\F} {f}
\newcommand  {\E} {e}
\newcommand  {\T} {T}
\renewcommand{\S} {\sigma} 
\newcommand  {\Psgm}{\P_{\S}}

\newcommand{\hh}{h}
\newcommand{\Th}{\Omega_{\hh}}

\newcommand{\xvP}{\xv_{\P}}        
\newcommand{\xvF}{\xv_{\F}}        
\newcommand{\xvE}{\xv_{\E}}        
\newcommand{\xvS}{\xv_{\S}}        

\newcommand{\DIM} {d}              
\newcommand{\INT} {0}              
\newcommand{\BND} {\partial}       

\newcommand{\hP}{\hh_{\P}}
\newcommand{\hF}{\hh_{\F}}
\newcommand{\hE}{\hh_{\E}}

\newcommand{\hS}{\hh_{\S}}

\newcommand{\mP}{\ABS{\P}}
\newcommand{\mF}{\ABS{\F}}

\newcommand{\mS}{\ABS{\S}}

\newcommand{\Pset}{\mathcal{P}}    
\newcommand{\Fset}{\mathcal{F}}    
\newcommand{\Eset}{\mathcal{E}}    
\newcommand{\Sset}{\mathcal{S}}    

\newcommand{\NMB}{N}
\newcommand{\NP}{\NMB^{\Pset}}   
\newcommand{\NS}{\NMB^{\Sset}}   

\newcommand{\dx}{\,d\xv}

\newcommand{\norE} {\mathbf{n}_{\E}}
\newcommand{\norF} {\mathbf{n}_{\F}}
\newcommand{\norP} {\mathbf{n}_{\P}}
\newcommand{\norS} {\mathbf{n}_{\S}}

\newcommand{\uss} {\us^{s}}
\newcommand{\ussh}{\ush^{s}}
\newcommand{\ussht}{\widetilde{\us}_{\hh}^{s}}
\newcommand{\ush}{\us_{\hh}}
\newcommand{\usht}{\widetilde{\us}_{\hh}}

\newcommand{\vsh}{\vs_{\hh}}

\newcommand{\vsS}{\vs_{\S}}
\newcommand{\vsI}{\vs^{\INTP}}

\newcommand{\wsh}{\ws_{\hh}}
\newcommand{\wsht}{\widetilde{\ws}_{\hh}}

\newcommand{\vvh}{\vv_{\hh}}

\newcommand{\wvh}{\wv_{\hh}}

\newcommand{\asP}{\as^{\P}}
\newcommand{\bsP}{\bs^{\P}}

\newcommand{\ash}{\as_{\hh}}
\newcommand{\bsh}{\bs_{\hh}}

\newcommand{\ashP}{\as^P_{\hh}}
\newcommand{\bshP}{\bs^P_{\hh}}

\newcommand{\bsht} {\widetilde{\bs}_{\hh}}
\newcommand{\bshtP}{\widetilde{\bs}^{\P}_{\hh}}

\newcommand{\SP} {S^{\P}}
\newcommand{\SPt}{\widetilde{S}^{\P}}

\newcommand{\nlen}{\hspace{-0.2mm}}

\newcommand{\snorm}  [2]{|#1|_{#2}}

\newcommand{\norm}   [2]{|\nlen|#1|\nlen|_{#2}}

\newcommand{\NORM}   [2]{\left|\nlen\left|#1\right|\nlen\right|_{#2}}

\newcommand{\abs}    [1]{|#1|}

\newcommand{\ABS}    [1]{\left|#1\right|}

\newcommand{\jump}[1]{\lbrack\!\lbrack\,#1\,\rbrack\!\rbrack}

\newcommand{\Vhk} {\VS{\hh}_{k}}
\newcommand{\Vhkt}{{\widetilde{V}}^{\hh}_{k}}

\newcommand{\calM} [1]{\mathcal{M}_{#1}}
\newcommand{\calMs}[1]{\mathcal{M}^{*}_{#1}}

\newcommand{\Pin}[1]{\Pi^{\nabla, \P}_{#1}}
\newcommand{\Piz}[1]{\Pi^{0, \P}_{#1}}
\newcommand{\PiSz}[1]{\Pi^{0,\S}_{#1}}

\newcommand{\calNh}{\mathcal{N}_{\hh}}

\newcommand{\cbot}{c_*}
\newcommand{\ctop}{c^*}
\newcommand{\cttop}{\widetilde{c}^*}
\newcommand{\ctbot}{\widetilde{c}_*}

\newcommand{\sumP}{\sum_{\P\in\Th}}

\newcommand{\restrict}[2]{{#1}_{|{#2}}}

\newcommand{\EOD}{\end{document}}

\begin{document}

\begin{frontmatter}

  \title{The nonconforming virtual element method \\for eigenvalue
    problems}
  
  \author[UNIPV] {F. Gardini}
  \author[LANL]  {, G. Manzini}
  \author[UNIMIB]{, and G. Vacca}

  \address[UNIPV]{
    Dipartimento di Matematica \emph{F. Casorati},
    Universit\`a di Pavia,
    Via Ferrata, 5 - 27100 Pavia, Italy;
    \emph{e-mail: francesca.gardini@unipv.it}
  }
  \address[LANL]{
    Group T-5,
    Theoretical Division,
    Los Alamos National Laboratory,
    Los Alamos, New Mexico - 87545, USA;
    \emph{e-mail: gmanzini@lanl.gov}
  }
  \address[UNIMIB]{
    Dipartimento di Matematica e Applicazioni, 
    Universit\`a di Milano Bicocca, 
    Via R. Cozzi, 55 – 20125 Milano, Italy;
    \emph{e-mail: giuseppe.vacca@unimib.it}
  }

  \begin{abstract}
    We analyse the nonconforming Virtual Element Method (VEM)
    for the approximation of elliptic eigenvalue problems.
    The nonconforming VEM allow to treat in the same formulation the two- and three-dimensional case. 
    We present two possible formulations of the discrete problem, derived  respectively by the nonstabilized and stabilized approximation of the $L^2$-inner product, and we study the convergence
properties of the corresponding discrete eigenvalue problem.
    The proposed schemes provide a correct approximation of the
    spectrum, in particular we prove optimal-order error estimates for the
    eigenfunctions and the usual double order of convergence of the eigenvalues.
    Finally we show a large set of numerical tests supporting the
    theoretical results, including a comparison with the conforming
    Virtual Element choice.

  \end{abstract}

  \begin{keyword}
    nonconforming virtual element,
    eigenvalue problem,
    polygonal mesh
  \end{keyword}

\end{frontmatter}

\maketitle
\raggedbottom

\section{Introduction}\label{sec:Introduction}
The Virtual Element Method (in short, VEM) was introduced in
\cite{BeiraodaVeiga-Brezzi-Cangiani-Manzini-Marini-Russo:2013} as a
generalization of the finite element method to arbitrary polygonal and
polyhedral meshes and as a variational reformulation of the Mimetic Finite
Difference (MFD)
method~\cite{BdV-Lipnikov-Manzini:book,Lipnikov-Manzini-Shashkov:2014}.
The main idea behind VEM is that the approximation spaces consist of
the usual polynomials and additional nonpolynomial functions that
locally solve suitable differential problems.
Consequently, the virtual functions are not explicitly known
pointwisely (hence the name virtual), but only a limited set of
information about them are at disposal.
Nevertheless, the available information is sufficient to construct
the discrete operators and the right-hand side. 
Indeed, the VEM does not require the evaluation of test and trial
functions at the integration points, but uses suitable projections
onto the space of piecewise polynomials that are exactly computable from the
degrees of freedom.
Therefore, the approximated discrete bilinear forms require only
integration of polynomials on each polytopal element in order to be
computed, without the need to integrate complex non-polynomial
functions on the elements and without any loss of accuracy. 
Morover the VEM 
can be easily applied to three dimensional problems and can handle non-convex (even
non simply connected) elements
\cite{Ahmad-Alsaedi-Brezzi-Marini-Russo:2013,Dassi:2017}. 
The Virtual Element Method has been developed
successfully for a large range of mathematical and engineering problems
\cite{
Stokes:divfree,%
BeiraodaVeiga-Manzini:2015,
Cangiani-Manzini-Russo-Sukumar:2015,%
Benedetto-Berrone:2016,%
Wriggers:2016,%
Chi-BdV-Paulino:2017,%
Vacca:2017,%
ADLP-HR,%
Vacca:2018}. Finally, high-order and higher-order continuity schemes 
have been presented in~\cite{Chernov-BdV-Mascotto-Russo:2016} and 
~\cite{Brezzi-Marini:2013,BdV-Manzini:2013,Antonietti-BdV-Scacchi-Verani:2016}, 
respectively.

The present paper focuses on the nonconforming VEM for the
approximation of the second-order elliptic eigenvalue problem.
The main advantage of nonconforming VEM introduced in
\cite{Ayuso-Lipnikov-Manzini:2016} is to cover ``\textit{in one
shot}'', i.e., using the same formulation, the two- and
three-dimensional case.
We recall that for the nonconforming methods, the approximating space
is not a subspace of the solution space.
In particular, for second-order elliptic problems we do not require
the $H^1$ regularity of the global discrete space as for the
conforming schemes, but we just impose that the moments, up to a
certain order, of the jumps of the discrete space functions across all
mesh interfaces are zero.

The nonconforming VEM has been applied successfully to the general
second-order elliptic problem \cite{Cangiani-Manzini-Sutton:2017}, the
Stokes equation \cite{Cangiani-Gyrya-Manzini:2016}, the biharmonic
problem \cite{Antonietti-Manzini-Verani:2018}, and the nonconforming
approach has been recently extended to the $h$- and $p$-version of the
harmonic VEM~\cite{Mascotto-Perugia-Pichler:2018}.
As first observed in \cite{Ayuso-Lipnikov-Manzini:2016} and recently
investigated in~\cite{DiPietro-Droniou-Manzini:2018}, the
nonconforming VEM coincides with the high-order MFD method proposed
in~\cite{Lipnikov-Manzini:2014}.

In the present paper we study the nonconforming VEM for the
approximation of the Laplace eigenvalue problem.
Using the nonconforming virtual space introduced in
\cite{Ayuso-Lipnikov-Manzini:2016}, we introduce two approximated
bilinear forms, one stands for the discrete grad-grad form and the
other one stands for the discrete version of the $L^2$-inner product.
In particular, for the $L^2$-inner product, we consider both a
nonstabilized form and a stabilized one, and we study the convergence
properties of the corresponding discrete formulations.
It is shown that the Virtual Element Method provides optimal
convergence rates both in the approximations of eigenfunctions and the eigenvalues.

We remark that the conforming VEM formulation has been proposed for
the approximation of the Steklov eigenvalue problem
\cite{Mora-Rivera:2015, Mora-Rivera:2017}, the Laplace eigenvalue
problem \cite{Gardini-Vacca:2017},  the acoustic vibration problem~\cite{Beirao-Mora-Rivera-Rodriguez:2017}, 
and the vibration problem of Kirchhoff plates~\cite{Mora-Rivera-Velasquez:2017},
whereas~\cite{Cangiani-Gardini-Manzini:2011} deals with the Mimetic Finite
Difference approximation of the eigenvalue problem in mixed form. 

The outline of the paper is as follows. 
In Section~\ref{sec:Continuous:problem} we formulate the model Laplace
eigenvalue problem.
In Section~\ref{sc:spaces} we introduce the broken Sobolev spaces
(with respect to the polygonal decompositions) and we define the
conformity error. Moreover we recall the definition of the nonconforming
Virtual Element Spaces and their degrees of freedom.
In Section~\ref{sc:approx} we construct the approximated bilinear
forms and we state the nonconforming virtual problem.
In Section~\ref{sc:compact} we recall some fundamental results for the
spectral approximation of compact operators.
In Section~\ref{sc:analysis} we show the optimal rate of convergence
of the method by proving the {\itshape priori} error estimates for the
eigenvalues and eigenfunctions.
Section~\ref{sc:tests} presents several numerical tests.
Finally, in Section~\ref{sc:conclusions} we offer our final remarks
and conclusions.



\section{The continuous eigenvalue problem}
\label{sec:Continuous:problem}

In this section we describe the continuous eigenvalue problem and its
associated source problem.
Throughout the paper, we use the notation of Sobolev spaces, norms and
seminorms detailed in~\cite{Adams:1975}.
In particular, the symbols $\snorm{\cdot}{s,\omega}$ and
$\norm{\cdot}{s,\omega}$ are the seminorm and the norm of the Sobolev
space $\HS{s}(\omega)$ defined on the open bounded subset $\omega$ of
$\REAL^{\DIM}$, and $(\cdot,\cdot)_{\omega}$ is the $\LTWO$-inner
product.
If $\omega$ is the whole computational domain $\Omega$, the subscript may
be omitted and we may denote the Sobolev seminorm and norm by
$\snorm{\cdot}{s}$ and $\norm{\cdot}{s}$, and the $L^2$-inner product
by $(\cdot,\cdot)$.

\medskip
Let $\Omega\subset\REAL^{\DIM}$ for $\DIM=2,3$ be an open
polytopal domain with Lipschitz boundary $\Gamma$.
Consider the eigenvalue problem: \emph{Find
  $\lambda \in\REAL$, such that there exists $\us \neq 0$}:
\begin{align}
  -\Delta\us &= \lambda\us\phantom{0} \quad\textrm{in~}\Omega,\label{eq:mod00:A}\\
  \us        &= 0\phantom{\lambda\us} \quad\textrm{on~}\Gamma.\label{eq:mod00:B}
\end{align}
Its variational formulation 
reads as:
\emph{Find $(\lambda,\us)\in\REAL\times\Vspace$, $\NORM{u}{0} = 1$, such
  that}
\begin{equation}\label{eq:eigPbm}
  \as(\us,\vs)=\lambda\bs(\us,\vs)\qquad\forall\vs\in\Vspace,
\end{equation}
where $\Vspace=\HONEzr(\Omega)$, the bilinear form
$\as:\Vspace\times\Vspace\to\REAL$ is given by
\begin{align}
  \as(\us,\vs)=\int_{\Omega}\nabla\us\cdot\nabla\vs\, {\rm d}\Omega\qquad\forall\us,\,\vs\in\Vspace,
  \label{eq:exact:a}
\end{align}
and the bilinear form $\bs:\Vspace\times\Vspace\to\REAL$ is the
$\LTWO$-inner product on $\Omega$, i.e., 
\begin{align}
	\bs(\us,\vs)=(\us,\vs) \qquad\forall\us,\,\vs\in\Vspace \,.
	\label{eq:exact:b}
\end{align}
The eigenvalues of problem~\eqref{eq:eigPbm} form a positive
increasing divergent sequence and the corresponding eigenfunctions are
an orthonormal basis of $\Vspace$ with respect both to the $\LTWO$-inner
product and the scalar product associated with the bilinear form
$\as(\cdot,\cdot)$. Moreover, each eigenspace has finite dimension ~\cite{Boffi:2010,Babuska-Osborn:1991}.

\medskip
In the spectral convergence analysis, we will need the approximation
results of the source problem associated with~\eqref{eq:eigPbm}, which
we state as follows:
\emph{Find $\uss\in\Vspace$ such that}
\begin{equation}\label{eq:sourcePbm}
  \as(\uss,\vs)=\bs(\fs,\vs)\quad\forall\vs\in\Vspace,
\end{equation}
where we assume the forcing term $\fs$ (at least) in $\LTWO(\Omega)$.
%
%
Due to regularity results~\cite{Agmon:1965,Grisvars:1992}, there exists a constant
$r>1/2$ that depends only on $\Omega$ such that the solution $\uss$
belongs to the space $\HS{1+r}(\Omega)$.
If $\Omega$ is a convex polytopal domain, then $r\geq
1$.
Instead, $r\geq\pi/\omega-\varepsilon$ for any $\varepsilon >0$ if
$\Omega$ is a two-dimensional non-convex polygonal domain with maximum
interior angle $\omega<2\pi$.
A similar result holds in the three-dimensional case,  $\omega$ being  the maximum
reentrant wedge angle.
Eventually, there exists a positive constant $C$ such that
\begin{equation}
  \snorm{ \uss }{1+r}\leq C\norm{\fs}{0} \, .
  \label{eq:regularity}
\end{equation}

\section{The nonconforming virtual element method} 
\label{sc:spaces}

In this section, we introduce the family of mesh decompositions of the
computational domain and the mesh regularity assumptions and, then, we
define the non-conforming virtual element space (and related
approximation properties) that we need for the proper formulation of
the virtual element method, cf.~\cite{Ayuso-Lipnikov-Manzini:2016}.

\subsection{Mesh definition and regularity assumptions}
Let $\mathcal{T}=\{\Th\}_{\hh}$ be a family of decompositions of
$\Omega$ into nonoverlapping polytopal elements $\P$ with
nonintersecting boundary $\partial\P$, center of gravity $\xvP$,
$\DIM$-dimensional measure $\mP$, and diameter
$\hP=\sup_{\xv,\yv\in\P}\abs{\xv-\yv}$.
The subindex $\hh$ that labels each mesh $\Th$ is the maximum of the
diameters $\hP$ of the elements of that mesh.
The boundary of $\P$ is formed by straight edges when $\DIM=2$ and
flat faces when $\DIM=3$.
The midpoint and length of each edge $\E$ are denoted by $\xvE$ and
$\hE$, respectively.
The center of gravity, diameter and area of each face $\F$ are denoted
by $\xvF$, $\hF$, and $\mF$, respectively.
Sometimes we may refer to the geometric objects forming the elemental
boundary $\partial\P$ by the term \emph{side} instead of
\emph{edge}/\emph{face}, and adopt a unified notation by using the
symbol $\S$ instead of $\E$ or $\F$ regardless of the number of
spatial dimensions.
Accordingly, $\xvS$, $\hS$, and $\mS$ denote the center of gravity,
diameter, and measure of side $\S$.

We denote the unit normal vector to the elemental boundary
$\partial\P$ by $\norP$, and the unit normal vector to edge $\E$, face
$\F$ and side $\S$ by $\norE$, $\norF$, $\norS$, respectively.
Each vector $\norP$ points out of $\P$ and the orientation of $\norE$,
$\norF$, and $\norS$ is fixed \emph{once and for all} in every mesh
$\Th$.
Finally, $\Eset_{\hh}$, $\Fset_{\hh}$, and $\Sset_{\hh}$ denote the
set of edges, faces, and sides of the skeleton of $\Th$.
We may distinguish between \emph{internal} and \emph{boundary} sides
by using the superscript $\INT$ and $\BND$.
Therefore, $\Sset_{\hh}^{\INT}$ is the set of the internal sides,
$\Sset_{\hh}^{\BND}$ the set of the boundary sides, and, obviously,
$\Sset_{\hh}^{\INT}\cap\Sset_{\hh}^{\BND}=\emptyset$ and
$\Sset_{\hh}=\Sset_{\hh}^{\INT}\cup\Sset_{\hh}^{\BND}$.

\medskip 
Now, we state the mesh regularity assumptions that are required for
the convergence analysis.
Since the method cannot be used simultaneously for $\DIM=2$ and
$\DIM=3$ and, hence, no ambiguity is possible in such sense, we may
refer to the two- and three-dimensional case using the same label
\ASSUM{A0} and the same symbol $\rho$ to denote the \emph{mesh
  regularity constant}.
Note that the assumptions for $\DIM=2$ can be derived from those for
$\DIM=3$ by reducing the spatial dimension.

\medskip
\noindent
\ASSUM{A0}~\textbf{Mesh regularity assumptions}.

\begin{itemize}
  \medskip
\item $\mathbf{\DIM=3}$.
  There exists a positive constant $\varrho$ independent of $\hh$ (and, hence, of
  $\Th$)  such that for every polyhedral element $\P\in\Th$ it
  holds that
  \begin{description}
  \item[$(i)$] $\P$ is star-shaped with respect to a ball with radius $\ge\varrho\hP$;
  \item[$(ii)$] every face $\F\in\P$ is star-shaped with respect to a
    disk  with radius $\ge\varrho\hF$;
  \item[$(iii)$] for every edge $\E\in\partial\F$ of every face
    $\F\in\partial\P$ it holds that $\hE\geq\varrho\hF\geq\varrho^2\hP$.
  \end{description}

  \medskip
\item $\mathbf{\DIM=2}$.
  There exists a positive constant $\varrho$ independent of $\hh$ (and, hence, of
  $\Th$) such that for every polygonal element $\P\in\Th$ it
  holds that
  \begin{description}
  \item[$(i)$] $\P$ is star-shaped with respect to a disk  with radius $\ge\varrho\hP$;
  \item[$(ii)$] for every edge $\E\in\partial\P$ it holds that
    $\hE\geq\varrho\hP$.
  \end{description}
\end{itemize}

\medskip
\begin{remark}
  The star-shapedness property implies that elements and faces are
  \emph{simply connected} subsets of $\REAL^{\DIM}$ and
  $\REAL^{\DIM-1}$, respectively.
  The scaling assumption implies that the number of edges and faces in
  each elemental boundary is uniformly bounded over the whole mesh
  family $\mathcal{T}$.
\end{remark}

\subsection{Basic setting}
We introduce the broken Sobolev space for any $s>0$
\begin{align*}
  \HS{s}(\Th) 
  = \prod_{\P\in\Th}\HS{s}(\P) 
  = \big\{\,\vs\in\LTWO(\Omega)\,:\,\restrict{\vs}{\P}\in\HS{s}(\P)\,\big\}, 
\end{align*}
and define the broken $\HS{s}$-norm
\begin{align}
  \label{eq:Hs:norm-broken}
  \norm{\vs}{s,\hh}^2 = \sum_{\P\in\Th}\norm{\vs}{s,\P}^{2}
  \qquad\forall\,\vs\in\HS{s}(\Th),
\end{align}
and for $s=1$ the broken $H^{1}$-seminorm
\begin{align}
  \label{eq:norm-broken}
  \snorm{\vs}{1,h}^2 = \sum_{\P\in\Th}\norm{\nabla\vs}{0,\P}^{2} 
  \qquad\forall\,\vs\in\HONE(\Th).
\end{align}

Let $\sigma\subset\partial\Psgm^{+}\cap\partial\Psgm^{-}$ be the internal
side (edge/face) shared by elements $\Psgm^{+}$ and $\Psgm^{-}$, and $\vs$
a function that belongs to $\HONE(\Th)$.
We denote 
the traces of $\vs$ on $\S$ from the interior of elements
$\Psgm^{\pm}$ by $\vsS^{\pm}$,
and the unit normal vectors to $\S$ pointing from $\Psgm^{\pm}$ to
$\Psgm^{\mp}$ by $\norS^{\pm}$.
Then, we introduce the jump operator
$\jump{\vs}=\vsS^{+}\norS^{+}+\vsS^{-}\norS^{-}$ at each internal side
$\S\in\Sset_{\hh}^{\INT}$, and $\jump{\vs}=\vsS\norS$ at each boundary
side $\S\in\Sset_{\hh}^{\BND}$.
The nonconforming space $\HONEnc(\Th;k)$ for any integer $k\geq1$ is
the subspace of the broken Sobolev space $\HONE(\Th)$ defined as
\begin{equation}
  \label{eq:H1-nc:def}
  \HONEnc(\Th;k) = \left\{\,
  \vs\in\HONE(\Th)\,:\,\int_{\S}\jump{\vs}\cdot\norS\,\qs\,{\rm d}\S=0\,
  \,\,\forall\,\qs\in\PS{k-1}(\S),
  \,\,\forall\S\in\Sset_{\hh}
  \,\right\}.
\end{equation}
Since $\jump{\vs}=0$ on any internal mesh side whenever $\vs$ belongs to
$\HONE(\Omega)$, it is trivial to show that
$\HONEzr(\Omega)\subset\HONEnc(\Th;k)$.

\medskip
Hereafter, we consider the extension of the bilinear form
$\as(\cdot,\cdot)$ to the broken Sobolev space $\HONE(\Th)$, which is
given by splitting it as sum of local terms:
\begin{align}
  &\as:\HS{1}(\Th)\times\HS{1}(\Th)\to\REAL\textrm{~~with~~}\nonumber\\[0.5em]
  &\as(\us,\vs) = \sum_{\P\in\Th}\asP(\us,\vs)
  \textrm{~~where~~}\asP(\us,\vs) = \int_{\P}\nabla\us\cdot\nabla\vs\, {\rm d}\P,
  \quad\forall\us,\vs\in\HS{1}(\Th).
  \label{eq:extension}
\end{align}
Clearly, the same definition applies when at least one entry of
$\as(\cdot,\cdot)$ belongs to the nonconforming space
$\HONEnc(\Th;k)$, which is a subspace of $\HONE(\Th)$, and the
nonconforming virtual element space, which will be defined in the next
section as a subspace of $\HONEnc(\Th;k)$.

\medskip
The nonconforming space with $k=1$ has the minimal regularity required
for the VEM formulation and the convergence analysis.
It is straightforward to show that $\snorm{\cdot}{1,\hh}$ is a norm on
$\HONEnc(\Th;k)$, although it is only a seminorm for the discontinuous
functions of $\HONE(\Th)$.
Moreover, using the Poincar\'e-Friedrichs inequality
$\norm{\vs}{0}^{2}\leq C_{PF}\snorm{\vs}{1,\hh}^{2}$, which
holds for every $\vs\in\HONEnc(\Th;k)$ and some positive constant
$C_{PF}$ independent of $\hh$, cf.~\cite{Brenner:2003}, we can show
that $\snorm{\cdot}{1,\hh}$ is equivalent to $\norm{\cdot}{1,\hh}$.
Therefore, we may refer to the seminorm $\snorm{\cdot}{1,h}$ as a norm
in $\HONEnc(\Th;k)$.

According to~\cite{Ayuso-Lipnikov-Manzini:2016}, for 
$\us\in\HS{s}(\Omega)$ with $s\geq3\slash{2}$ solution to 
\eqref{eq:sourcePbm}
and $\vs\in\HONEnc(\Th;k)$
we find that
\begin{align}
  \label{eq:def:Nh}
  \calNh(\us,\vs) 
  := \as(\us,\vs) - \bs(\fs,\vs)
  = \sum_{\S\in\Sset_{\hh}}\int_{\S}\nabla\us\cdot\jump{\vs}\, {\rm d}\S.
\end{align}
The quantity $\calNh(\us,\vs)$ is called the \emph{conformity error}.

\medskip
We now recall an estimate for the term measuring the nonconformity.
\begin{lemma}\label{le:Nh}
  Assume \ASSUM{A0} is satisfied. 
  Let $\us\in\HS{1+r}(\Omega)$ with $r\geq 1$ be the solution to
  \eqref{eq:sourcePbm}.
  Let $\vs\in\HONEnc(\Th;k)$, $k\geq1$, as defined in
  \eqref{eq:H1-nc:def}.
  Then, there exists a constant $C>0$ depending only on the polynomial
  degree and the mesh regularity such that
  \begin{align}
    \abs{\calNh(\us,\vs)} \leq C\hh^{t}\snorm{\us}{1+r}\snorm{\vs}{1,h}
  \end{align}
  where $t=\min\{k,r\}$ and 
  $\calNh(\us,\vs)$ is defined in \eqref{eq:def:Nh}.
\end{lemma}

\medskip
Throughout the paper, $\PS{\ell}(D)$ denotes the space of polynomials
of degree up to $\ell$ for any integer number $\ell\geq0$ on the
bounded connected subset $D$ of $\REAL^{\nu}$ with $\nu=1,2,3$.
The polynomial space $\PS{\ell}(D)$ is finite dimensional and we
denote its dimension by $\pi_{\ell,\nu}$.
It holds that $\pi_{\ell,1}=\ell+1$ for $\nu=1$;
$\pi_{\ell,2}=(\ell+1)(\ell+2)\slash{2}$ for $\nu=2$;
$\pi_{\ell,3}=(\ell+1)(\ell+2)(\ell+3)\slash{6}$ for $\nu=3$.
We also conventionally take $\PS{-1}(D)=\{0\}$ and $\pi_{-1,\nu}=0$.
Let $\xv_{D}$ denote the center of gravity of $D$ and $\hh_{D}$ its
characteristic length, as, for instance, the edge length for $\nu=1$,
the face diameter for $\nu=2$ and the cell diameter for $\nu=3$.
A basis for $\PS{\ell}(D)$ is provided by $\calM{\ell}(D) =
\left\{\,\left(({\xv-\xv_{D})}\slash{\hh_{D}}\right)^{\alpha}\textrm{~with~}\abs{\alpha}\leq\ell\,\right\}$,
the set of the \emph{scaled monomials of degree up to $\ell$}, where
$\alpha=(\alpha_1,\ldots\alpha_{\nu})$ is a $\nu$-dimensional
multi-index of nonnegative integers $\alpha_i$ with degree
$\abs{\alpha}=\alpha_1+\ldots+\alpha_{\nu}$ and, with obvious notation,
$\xv^{\alpha}=\xs_1^{\alpha_1}\cdots\xs_{\nu}^{\alpha_{\nu}}$ for any
$\xv\in\REAL^{\nu}$.
We will also use the set of \emph{scaled monomials of degree exactly
  equal to $\ell$}, denoted by $\calM{\ell}^{*}(D)$ and obtained by
setting $\abs{\alpha}=\ell$ in the definition of $\calM{\ell}(D)$.
Finally, we denote by $\Piz{\ell}\,:\,\LTWO(\P)\to\PS{\ell}(\P)$ for
$\ell\geq0$ the $\LTWO$-orthogonal projection onto the polynomial
space $\PS{\ell}(\P)$, and by
$\PiSz{\ell}\,:\,\LTWO(\S)\to\PS{\ell}(\S)$ for $\ell\geq0$ the
$\LTWO$-orthogonal projection onto the polynomial space
$\PS{\ell}(\S)$.

\subsection{Local and global nonconforming virtual element space}
We construct the local nonconforming virtual element space by
resorting to the so-called \emph{enhancement strategy} originally
devised in~\cite{Ahmad-Alsaedi-Brezzi-Marini-Russo:2013} for the
conforming \VEM{} and later extended to the nonconforming \VEM{}
in~\cite{Cangiani-Manzini-Sutton:2017}.
To this end, on every polytopal cell $\P\in\Th$ and for any integer
number $k\geq 1$ we first define the finite dimensional functional
space
\begin{equation}\label{eq:def:Vhkt0}
  \Vhkt(\P) = \left\{\,
  \vsh\in\HONE(\P)\,:\,
  \frac{\partial\vsh}{\partial\nv}\in\PS{k-1}(\S)\, \, \forall\S\subset\partial\P,\,
  \Delta\vsh\in\PS{k}(\P)\,
  \right\}.
\end{equation}
We notice that the space $\Vhkt(\P)$ clearly contains the polynomials of degree $k$.

Then, we introduce the set of continuous linear functionals from
$\Vhkt(\P)$ to $\REAL$ that for every virtual function $\vsh$ of
$\Vhkt(\P)$ provide:
\begin{description}
\item[]\textbf{(D1)} the moments of $\vsh$ of order up to $k-1$ on
  each $(\DIM-1)$-dimensional side $\S\in\partial\P$:
  \begin{equation}\label{eq:dofs:01}
    \frac{1}{\mS}\int_{\S}\vsh\,\ms\, {\rm d}\S,
    \,\,\forall\ms\in\calM{k-1}(\S),\,
    \forall\S\in\partial\P;
  \end{equation}
\item[]\textbf{(D2)} the moments of $\vsh$ of order up to $k-2$ on
  $\P$:
  \begin{equation}\label{eq:dofs:02}
    \frac{1}{\mP}\int_{\P}\vsh\,\ms\, {\rm d}\P,
    \quad\forall\ms\in\calM{k-2}(\P).
  \end{equation}
\end{description}
Finally, we introduce the elliptic projection operator
$\Pin{k}:\Vhkt(\P)\to\PS{k}(\P)$ that for any $\vsh\in\Vhkt(\P)$ is
defined by:
\begin{align}
  \int_{\P}\nabla\Pin{k}\vsh\cdot\nabla\qs\, {\rm d}\P =
  \int_{\P}\nabla\vsh\cdot\nabla\qs\, {\rm d}\P
  \qquad \text{$\forall q \in \PS{k}(\P)$}
\end{align}
together with the additional conditions:
\begin{align}
  \int_{\partial\P}(\Pin{k}\vsh-\vsh)\, {\rm d} \S&=0 \qquad\textrm{if~}k=1,    \label{eq:def:Pib:1}\\[0.5em]
  \int_{\P}       (\Pin{k}\vsh-\vsh)\dx&=0 \qquad\textrm{if~}k\geq 2.\label{eq:def:Pib:k}
\end{align}
As proved in~\cite{Cangiani-Manzini-Sutton:2017}, the polynomial
projection $\Pin{k}\vsh$ is exactly computable using only the values
from the linear functionals \textbf{(D1)}-\textbf{(D2)}.
Furthermore, $\Pin{k}$ is a polynomial-preserving operator, i.e.,
$\Pin{k}\qs=\qs$ for every $\qs\in\PS{k}(\P)$.

\medskip
We are now ready to introduce the \emph{local nonconforming virtual
  element space of order $k$} on the polytopal element $\P$, which is
the subspace of $\Vhkt(\P)$ defined as follow:
\begin{equation}\label{eq:def:Vhk0}
  \Vhk(\P) = \Big\{\,
  \vs\in\Vhkt(\P)\quad \textrm{~such~that~} \quad
  (\vsh-\Pin{k}\vsh,\ms)_{\P}=0
  \, \,\,\forall\ms\in\calMs{k-1}(\P)\cup\calMs{k}(\P)
  \,\Big\}.
\end{equation}
The space $\Vhk(\P)$ has the two important properties that we outline
below:
\begin{description}
\item[$(i)$] it still contains the space of polynomials of degree at
  most $k$;
\item[$(ii)$] the values provided by the set of continuous linear
  functionals \textbf{(D1)}-\textbf{(D2)} uniquely determine every
  function $\vsh$ of $\Vhk(\P)$ and can be taken as the \emph{degrees
    of freedom} of $\Vhk(\P)$.
\end{description}

\medskip
Property $(i)$ above is a direct consequence of space definitions 
\eqref{eq:def:Vhkt0} and \eqref{eq:def:Vhk0}, and guarantees the 
optimal order of approximation, 
while property $(ii)$ follows from the unisolvency of the degrees of
freedom \textbf{(D1)}-\textbf{(D2)} that was proved
in~\cite{Ayuso-Lipnikov-Manzini:2016,Cangiani-Manzini-Sutton:2017}).
Additionally, the $\LTWO$-orthogonal projection $\Piz{k}\vsh$ is
\emph{exactly} computable using only the degrees of freedom of $\vsh$,
cf.~\cite{Cangiani-Manzini-Sutton:2017}, and
$\Piz{k}\vsh=\Pin{k}\vsh$ for $k=1,2$ as for the conforming \VEM{},
cf.~\cite{Ahmad-Alsaedi-Brezzi-Marini-Russo:2013}.

\medskip
Finally, the \emph{global nonconforming virtual element space $\Vhk$
  of order $k\geq1$ subordinate to the mesh $\Th$} is obtained by
gluing together the elemental spaces $\Vhk(\P)$ to form a subspace of
the nonconforming space $\HONEnc(\Th;k)$.
The formal definition reads as:
\begin{align}
  \Vhk:=\Big\{\,\vsh\in\HONEnc(\Th;k)\,:\,\restrict{\vsh}{\P}\in\Vhk(\P)
  \,\,\,\forall\P\in\Th\,\Big\}.
  \label{eq:def:nvem0}
\end{align}
A set of degrees of freedom for $\Vhk$ is given by collecting the
values from the linear functionals \textbf{(D1)} for all the mesh
sides and \textbf{(D2)} for all the mesh elements.
The unisolvence of such degrees of freedom
\eqref{eq:dofs:01}-\eqref{eq:dofs:02} for $\Vhk$ is an immediate
consequence of their unisolvence on each local space $\Vhk(\P)$.
Thus, the dimension of $\Vhk$ is equal to
$\NS\times\pi_{k-1,\DIM-1}+\NP\times\pi_{k-2,\DIM}$, where $\NS$ is
the total number of sides, and $\NP$ the total number of elemnts, 
and we recall that $\pi_{\ell,\nu}$ is the dimension of
the space of polynomials of degree up to $\ell$ in $\REAL^{\nu}$.

\medskip
\begin{remark}
  The set of degrees of freedom can be properly redefined by excluding
  the moments on the $\NMB^{\Sset,\BND}$ boundary sides, i.e., for
  $\S\in\Sset_{\hh}^{\BND}$, which are set to zero to impose the
  homogeneous Dirichlet boundary condition~\eqref{eq:mod00:B}.
  This reduces the dimension of $\Vhk$ to
  $\NMB^{\Sset,\INT}\times\pi_{k-1,\DIM-1}+\NP\times\pi_{k-2,\DIM}$.
\end{remark}

\subsection{Approximation properties}
Both for completeness of exposition and future reference in the paper,
we briefly summarize the local approximation properties by polynomial
functions and functions in the virtual nonconforming space.
We omit here any details about the derivation of these estimate and
refer the interested readers to
References~\cite{%
Ciarlet:1991,%
Brenner-Scott:2008,%
BeiraodaVeiga-Brezzi-Cangiani-Manzini-Marini-Russo:2013,%
Ahmad-Alsaedi-Brezzi-Marini-Russo:2013,%
Ayuso-Lipnikov-Manzini:2016,%
2016stability,%
Brenner:2017}.

\medskip
\noindent
\emph{Local polynomial approximations.}
On a given element $\P\in\Th$, let
$\vs\in\HS{s}(\P)$ with $1 \leq s\leq k+1$.
Under mesh assumptions~\ASSUM{A0}{}, there exists a piecewise polynomial approximation $\vs_{\pi}$
that is of degree $k$ on each element,  such that
\begin{equation}\label{eq:approx:pi}
  \norm{\vs-\vs_{\pi}}{0,\P} + \hP\snorm{\vs-\vs_{\pi}}{1,\P}
  \leq C\hP^{s}\snorm{\vs}{s,\P},
\end{equation} 
for some constant $C>0$ that depends only on the polynomial degree
$k$ and the mesh regularity constant $\varrho$.

\medskip
\noindent
An instance of such a local polynomial approximation is provided by
the $\LTWO$-projection $\Piz{k}\vs$ onto the local polynomial space
$\PS{k}(\P)$, which  satisfies the (optimal) error bound \eqref{eq:approx:pi}.

\medskip
\noindent
Furthermore, consider the internal side $\S\in\Sset_{\hh}^{\INT}$
and let $\P^{\pm}_{\S}$ be the two elements sharing $\S$, so that
$\S \subset \partial\P^{+}_{\S}\cap\partial\P^{-}_{\S}$.
Let $\Omega_{\S}=\P^{+}_{\S}\cup\P^{-}_{\S}$.
Then, the following trace inequality~\cite{Brenner-Scott:2008} holds 
for every $\vs\in\HS{1}(\Omega_\S)$
\begin{equation}\label{eq:trace:inequality}
  \norm{\vs}{0,\S} 
  \leq C\hS^{-\frac{1}{2}}\norm{\vs}{0,\Omega_{\S}} 
  +  C\hS^{\frac{1}{2}}\snorm{\vs}{1,\Omega_{\S}}.
\end{equation}
Moreover, for every $\vs\in\HS{s}(\Th)$ with $1\leq s\leq k+1$, 
combining the $\LTWO$-projection approximation property with \eqref{eq:trace:inequality}
the following useful error
estimate holds \cite{DiPietro-Ern:2012}
\begin{equation}\label{eq:approx:pitrace}
  \norm{\vs-\PiSz{k}\vs}{0,\S} 
  \leq C\hS^{s- \frac{1}{2}}\snorm{\vs}{s,\Omega_{\S}} \,.
\end{equation}
The same estimate holds also for the boundary sides
$\S\in\Sset_{\hh}^{\BND}$ by taking $\Omega_{\S}=\P$, the element to
which $\S$ belongs.

\medskip
\noindent
\emph{Interpolation error.}
Similarly, under mesh regularity assumptions~\ASSUM{A0}{}, we can
define an interpolation operator in $\Vhk$ having optimal
approximation properties.
Therefore, for every $\vs\in\HS{s}(\P)$ with $1\leq s\leq k+1$ we can
find the local interpolate $\vsI\in\Vhk(\P)$ such that
\begin{align}\label{eq:interp:00}
  \norm{\vs-\vsI}{0,\P} + \hP\snorm{\vs-\vsI}{1,\P}\leq C\hP^{s}\snorm{\vs}{s,\P},
\end{align}
where $C>0$ is a positive constant independent of $\hh$.

\section{Virtual element discretization}
\label{sc:approx}

This section briefly reviews the nonconforming virtual element
discretization of the source problem and its extension to the
eigenvalue problem.

\subsection{The virtual element discretization of the source problem}
The goal of the present section is to introduce the virtual element
discretization of the source problem \eqref{eq:sourcePbm}. 
According to~\cite{Ayuso-Lipnikov-Manzini:2016,Gardini-Vacca:2017}, we
define a suitable discrete bilinear form $\ash(\cdot,\cdot)$
approximating the continuous gradient-gradient form $\as(\cdot,
\cdot)$, whereas for what concerns the bilinear form $\bs(\cdot,
\cdot)$ we propose two possible discretizations, hereafter denoted by
$\bsh(\cdot,\cdot)$ and $\bsht(\cdot,\cdot)$.

The discrete bilinear form $\ash(\cdot,\cdot)$ is the sum of elemental
contributions
\begin{align}
  \label{eq:a:global}
  \ash(\ush,\vsh)=\sum_{\P\in\Th}\ashP(\ush,\vsh),
\end{align}
where
\begin{align}
  \label{eq:discreteforms}
  \ashP(\ush,\vsh) = \asP (\Pin{k}\ush,\Pin{k}\vsh) 
  + \SP\Big( (I-\Pin{k})\ush, (I-\Pin{k})\vsh \Big),
\end{align}
and $\SP(\cdot,\cdot)$ denotes any symmetric positive definite
bilinear form on the element $\P$ for which there exist two
positive uniform constants $\cbot$ and $\ctop$ such that
\begin{align}
  \cbot\asP(\vsh,\vsh)
  \leq\SP(\vsh,\vsh)
  \leq\ctop\asP(\vsh,\vsh)
  \quad\forall\vsh\in\Vhk(\P)\cap
  \ker(\Pin{k}) \,.
\end{align}
This requirement implies that $\SP(\cdot,\cdot)$ scales like
$\asP(\cdot,\cdot)$, namely $\SP(\cdot,\cdot)\simeq\hP^{d-2}$.

\medskip
The choice of the discrete form $\ash(\cdot,\cdot)$ is driven by the
need to satisfy the following properties:
\begin{description}
\item[-] {\emph{$k$-consistency}}: for all $\vsh\in\Vhk(\P)$ and for all
  $\qs\in\PS{k}(\P)$ it holds
  \begin{align}
    \label{eq:k-consistency}
    \ashP(\vsh,\qs) = \asP(\vsh,\qs);
  \end{align}
\item[-] {\emph{stability}}: there exists two positive constants
  $\alpha_*,\,\alpha^*$, independent of $\hh$ and of $\P$, such that
  \begin{align}
    \label{eq:stability}
    \alpha_*\asP(\vsh,\vsh)
    \leq\ashP(\vsh,\vsh)
    \leq\alpha^*\asP(\vsh,\vsh)\quad\forall\vsh\in\Vhk(\P) \,.
  \end{align}
\end{description} 
In particular, the first term in~\eqref{eq:discreteforms} ensures the
$k$-consistency of the method and the second one its stability,
cf.~\cite{Ayuso-Lipnikov-Manzini:2016}.

For what concerns the right hand side $\bs(\cdot,\cdot)$,
following the setting in \cite{Gardini-Vacca:2017}, we introduce two
possible approximated bilinear forms.

\medskip
\noindent
\emph{Non--stabilized bilinear form.}  In the first choice, we
consider the bilinear form $\bsh(\cdot,\cdot)$, which satisfies the
$k$-consistency property but not the stability property 
(extending to $\bsh(\cdot,\cdot)$ the definitions in 
\eqref{eq:k-consistency} and \eqref{eq:stability}).
Let us split the right-hand side $\bsh(\cdot,\cdot)$ with the sum of
local contributions:
\begin{align}
  \label{eq:bscal:globdef}
  \bsh(\fs,\vsh) = \sum_{P\in\Th}\bshP(\fs,\vsh).
\end{align}
Then, we define 
\begin{align}
  \label{eq:bscal:locdef}
  \bshP(\fs,\vsh) = \bsP(\Piz{k} \fs, \vsh) \qquad \forall \vsh \in \Vhk(\P).
\end{align}
We observe that each local term is fully computable for any functions
$\fs\in\LTWO(\Omega)$ and $\vsh$ in $\Vhk$ since
\begin{equation*}
  \bsh(\fs,\vsh) 
  = \sum_{\P\in\Th}\int_{\P}(\Piz{k}\fs)\vsh\, {\rm d}\P
  = \sum_{\P\in\Th}\int_{\P}\fs (\Piz{k}\vsh)\, {\rm d}\P,
\end{equation*}
and 
$\Piz{k}\vsh$ is computable (exactly) from the degrees of
freedom of $\vsh$ (cf. Property $(i)$ ).
Moreover, by definition of $L^2$-projection $\Piz{k}$, it is straightforward to check that
\begin{align}
  \label{eq:bscal:locdef:equivalence}
  \bshP(\fs,\vsh)  
   = \bsP( \Piz{k} \fs, \Piz{k}\vsh)\qquad \forall \vsh \in \Vhk(\P).
\end{align}
We estimate the approximation error of the right-hand side by using
the Cauchy-Schwarz inequality twice and estimate~\eqref{eq:approx:pi},
after noting that $\restrict{\vsh}{\P}\in\HONE(\P)$ and $\Piz{0}\vsh$
is orthogonal to $\fs-\Piz{k}\fs$:
\begin{equation}
\label{eq:bs-bsh}
\begin{split}
  \abs{\bs(\fs,\vsh) - \bsh(\fs, \vsh)} 
  &\leq\sum_{\P\in\Th}\big|\bsP(\fs-\Piz{k}\fs,\vsh)\big| 
  =\sum_{\P\in\Th}\big|\bsP((I-\Piz{k})\fs,(I-\Piz{0})\vsh)\big| \\
  & \leq
  \sum_{\P\in\Th}\norm{(I-\Piz{k})\fs}{0,P}\norm{(I-\Piz{0})\vsh}{0,P}
  \\[0.5em]
  &\leq\sum_{\P\in\Th}C\hP\norm{(I-\Piz{k})\fs}{0,P}\snorm{\vsh}{1,P}
  \leq C\hh\norm{(I-\Piz{k})\fs}{0}\snorm{\vsh}{1,\hh}.
\end{split}
\end{equation}
In the light of \eqref{eq:a:global} and \eqref{eq:bscal:globdef}, the virtual element discretization of source
problem~\eqref{eq:sourcePbm} reads
as
: \emph{Find $\ussh\in\Vhk$ such
that}
\begin{equation}
  \label{eq:discreteSource}
  \ash(\ussh,\vsh) = \bsh(\fs,\vsh) \quad\forall\vsh\in\Vhk \,.
\end{equation}

\medskip
\noindent
\emph{Stabilized bilinear form.}  The second approximation of the
bilinear form in \eqref{eq:exact:b} is inspired by the definition of
the virtual bilinear form $\ash$.  
In order to distinguish the two formulations, we denote the stabilized
bilinear form by $\bsht(\cdot,\cdot)$.
As usual, we first decompose $\bsht(\cdot,\cdot)$ into the sum of
local contributions:
\begin{align}
\label{eq:btilde:globdef}
  \bsht(\fs,\vsh) = \sum_{P\in\Th}\bshtP(\fs,\vsh) \, ,
\end{align}
and, then, we define
\begin{equation}
  \label{eq:discretebstab}
  \bshtP(\fs,\vsh) 
  = \bsP(\Piz{k} \fs, \Piz{k}\vsh) 
  + \SPt\Big( (I-\Piz{k})\fs, (I-\Piz{k})\vsh \Big),
\end{equation}
where $\SPt$ is any positive definite bilinear form on the element
$\P$ such that there exist two uniform positive constants $\ctbot$ and
$\cttop$ such that
\begin{align}
  \label{eq:ctbot}
  \ctbot\bsP(\vs,\vs)\leq\SPt(\vs,\vs)\leq\cttop\bs(\vs,\vs)
  \quad\forall\vs\in\Vhk(\P)\cap\KER(\Piz{k}).
\end{align}
We notice that the bilinear form $\bshtP$ defined above satisfies both
the $k$-consistency and the stability property.
\medskip
\begin{remark}
  \label{eq:scale2}
  In analogy with the condition on the form $\SP(\cdot,\cdot)$, we
  require that the form $\SPt(\cdot,\cdot)$ scales like
  $\bsP(\cdot,\cdot)$, that is $\SPt(\cdot,\cdot)\simeq\hh^{\DIM}$.
\end{remark}

By definition \eqref{eq:discretebstab} and from inequalities in
\eqref{eq:ctbot}, using similar computations as in \eqref{eq:bs-bsh},
for all $\fs \in \LTWO(\P)$ and for all $\vs \in \Vhk(\P)$ it holds
that
\begin{equation}
\label{eq:bs-bsht}
\begin{split}
  \abs{\bs(\fs,\vsh) - \bsht(\fs, \vsh)} 
  &\leq\sum_{\P\in\Th} \left( \big|\bsP(\fs-\Piz{k}\fs,\vsh)\big| + 
  \SPt\Big( (I-\Piz{k})\fs, (I-\Piz{k})\vsh \Big)\right) \\
  &\leq\sum_{\P\in\Th} \left( \big|\bsP((I-\Piz{k})\fs,(I-\Piz{0})\vsh)\big| +
  \cttop \norm{(I-\Piz{k})\fs}{0} \norm{(I-\Piz{k})\vsh}{0} 
   \right)\\
  & \leq
  \sum_{\P\in\Th}(1 + \cttop)\norm{(I-\Piz{k})\fs}{0,P}\norm{(I-\Piz{0})\vsh}{0,P}
  \\[0.5em]
  &\leq\sum_{\P\in\Th}C\hP\norm{(I-\Piz{k})\fs}{0,P}\snorm{\vsh}{1,P}
  \leq C\hh \, \sum_{\P\in\Th}\norm{(I-\Piz{k})\fs}{0} \, \snorm{\vsh}{1,\hh} \,.
\end{split}
\end{equation}
From \eqref{eq:a:global} and \eqref{eq:btilde:globdef}, the second formulation reads as:
\emph{Find $\ussht\in\Vhk$ such that}
\begin{align}
  \label{eq:discreteSource2}
  \ash(\ussht,\vsh) = \bsht(\fs,\vsh)\quad\forall\vsh\in\Vhk.
\end{align}
The well-posedness of the discrete source problems \eqref{eq:discreteSource} and ~\eqref{eq:discreteSource2} stem from the the coercivity and the continuity of 
the bilinear form $a_h(\cdot,\cdot)$ (cf. \eqref{eq:stability}), and from the continuity of 
the discrete forms $\bsh(\cdot,\cdot)$ and $\bsht(\cdot,\cdot)$.

\medskip
\medskip
\noindent 

Following the same arguments as in~\cite{Ayuso-Lipnikov-Manzini:2016} but with $\Piz{k}\fs$ for $k\geq1$ instead of $\Piz{k-2}\fs$, we
derive the error estimates in the $\HONE$-norm and $\LTWO$-norm that
we summarize in the following theorem for next reference in the paper.

We can state the following optimal error estimates between the solution of the continuous and discrete source problems \eqref{eq:discreteSource} and \eqref{eq:discreteSource2}.

\begin{theorem}\label{thm:aprioriestimate}
  Under the  assumptions \ASSUM{A0},
  let $\uss\in\HS{1+r}(\Omega)$ with regularity index $r\geq 1$ be the
  solution to \eqref{eq:sourcePbm} with $\fs\in\LTWO(\Omega)$. \\
  Let $\ussh$ and $\ussht\in\Vhk$ be respectively the solutions to the
  nonconforming virtual element method~\eqref{eq:discreteSource} and
  ~\eqref{eq:discreteSource2}.
  %
  Let $\vsh \in \{\ussh, \, \ussht\}$, then we have the following error estimates\medskip	    
  
  \begin{itemize}
  \item $\HONE$-error estimate:
    \begin{align*}
      \snorm{ \uss-\vsh }{1,\hh} &\leq C \left( 
        \hh^{t}\snorm{\uss}{1+r} 
        + \hh\norm{(I-\Piz{k})\fs}{0}
      \right) 
    \end{align*}

    \medskip
  \item $\LTWO$-error estimate (for a convex $\Omega$):
    \begin{align*}
      \norm{ \uss-\vsh }{0} &\leq 
      C \left( \hh^{t+1}
        \snorm{\uss}{1+r}
        + \hh^2\norm{(I-\Piz{k})\fs}{0}
      \right) 
    \end{align*}
  \end{itemize}
  with $t=min(k,r)$.
\end{theorem}
\medskip The proofs of these estimates are omitted as they are almost
identical to those
of~\cite[Theorem~4.3,Theorem~4.5]{Ayuso-Lipnikov-Manzini:2016}.  
The only difference is that here the forcing term $\fs$ is
approximated in each element by the orthogonal projection onto
polynomials of degree $k$ instead of $\overline{k}=max(k-2,0)$ (see
estimates \eqref{eq:bs-bsh} and \eqref{eq:bs-bsht}).
Note, indeed, that in the $\LTWO$-estimate
of~\cite[Theorem~4.5]{Ayuso-Lipnikov-Manzini:2016}, the term that
depends on the approximation of $\fs$ is given by
\begin{align*}
  (\hh^2+\hh^{\overline{k}+1})\norm{(I-\Piz{k})\fs}{0}
  =
  \begin{cases}
    (\hh^2 + \hh)\norm{(I-\Piz{k})\fs}{0} & \textrm{for~}k=1,2,\\[0.5em]
    \hh^2\norm{(I-\Piz{k})\fs}{0} & \textrm{for~}k\geq3.
  \end{cases}
\end{align*}
The difference in the coefficients is a consequence of the
orthogonality of $(I-\Piz{k})\fs$ to the polynomials of degree $k$ instead
of $\overline{k}$.

\medskip


\begin{remark}
\label{rm:estimate-regular-load}
   We observe that if the load term $\fs$ is an eigenfunction of problem~\eqref{eq:eigPbm}, 
    then $\fs$ solves the continuous source problem~\eqref{eq:sourcePbm} with
    datum $\lambda \fs$ and thus, thanks to the regularity
    result~\eqref{eq:regularity}, it belongs to $\HS{1+r}(\Omega)$ and
    $\snorm{f}{1+r} \le C\norm{f}{0}$. Then, the a priori
    error estimates in Theorem ~\ref{thm:aprioriestimate} reduce to\medskip
    \begin{itemize}
    \item $\HONE$-error estimate:
    \begin{align*}
    \snorm{ \uss-\vsh }{1,\hh} &\leq C \left( 
        \hh^{t}\snorm{\uss}{1+r} 
        + \hh\norm{(I-\Piz{k})\fs}{0}
      \right) 
      \le  \hh^{t}\snorm{\uss}{1+r}  \leq C \hh^{t}\norm{f}{0} \leq C \hh^{t},
    \end{align*}
    \item $\LTWO$-error estimate:
    \begin{align*}
      \norm{ \uss - \vsh }{0} &\leq 
      C \left( \hh^{t+1}
        \snorm{\uss}{1+r}
        + \hh^2\norm{(I-\Piz{k})\fs}{0} \right) 
        \leq C \hh^{t+1} \snorm{u^s}{1+r}
        \leq C \hh^{t+1}\norm{f}{0} \leq C \hh^{t+1},
    \end{align*}
    \end{itemize}
    since
    \begin{equation*}
      \norm{(I-\Piz{k})\fs}{0}\leq C \hh^{\min\{k+1,1+r\}}\snorm{f}{1+r} \leq C \hh^{t+1}\norm{f}{0}.
    \end{equation*}
\end{remark}

\begin{remark}
  In~\cite{Ayuso-Lipnikov-Manzini:2016} it is proved that the discrete
  problem \eqref{eq:discreteSource} is well-posed by taking, 
  in the definition of the discrete bilinear form $\bsh(\cdot,\cdot)$, 
  instead of $\Piz{k}\fs$,
  $\Piz{k-2}\fs$  for $k\geq 2$ and a first-order
  approximation of $\Piz{0}\fs$ for $k=1$.
  However, in the definition of the discrete bilinear forms $\bsh(\cdot,\cdot)$
  and $\bsht(\cdot,\cdot)$, we project onto the space $\PS{k}(\P)$. 
  This choice does not provide a better convergence rate, due
  to the $k$-consistency property, but it has been numerically observed that it gives 
  more accurate  results. 
\end{remark}

\subsection{The virtual element discretization of the eigenvalue
  problem.}
In the light of the nonconforming virtual element
method~\eqref{eq:discreteSource} and ~\eqref{eq:discreteSource2}
introduced in the previous section,
following~\cite{Gardini-Vacca:2017}, we consider two different
discretizations of the eigenvalue problem~\eqref{eq:eigPbm}.
The first formulation in inspired to the source problem \eqref{eq:discreteSource},
and uses definition~\eqref{eq:bscal:globdef}. We formulate the virtual
element approximation of~\eqref{eq:eigPbm} as:
\emph{Find $(\lambda_h,\ush)\in\REAL\times\Vhk$, $\NORM{\ush}{0}=1$, such that}
\begin{align}
  \label{eq:discreteEigPbm}
  \ash(\ush,\vsh) = \lambda_{\hh}\bsh(\ush,\vsh)\quad\forall\vsh\in\Vhk.
\end{align}

\medskip
\noindent
The second formulation is inspired to the virtual source problem \eqref{eq:discreteSource2} and
uses definition~\eqref{eq:btilde:globdef}. We formulate the second
approximation of problem~\eqref{eq:eigPbm} as:
\emph{Find $(\widetilde{\lambda}_{\hh},\usht)\in\REAL\times\Vhk$,
  $\NORM{\usht}{0} = 1$, such that}
\begin{align}\label{eq:discreteEigPbm2}
  \ash(\usht,\vsh) = \widetilde{\lambda}_{\hh}\bsht(\usht,\vsh)\quad\forall\vsh\in\Vhk.
\end{align}


\section{Spectral approximation for compact operators}
\label{sc:compact}

In this section, we briefly recall some spectral approximation results
that can be deduced from~\cite{Babuska-Osborn:1991,Boffi:2010,Kato:1976}.  
For more general results, we refer to the original papers.
Before stating the spectral approximation results, we introduce a
natural compact operator associated with problem~\eqref{eq:eigPbm} and
its discrete counterpart, and we recall their connection with the
eigenmode convergence.

\medskip
We associate problem~\eqref{eq:eigPbm} with its
solution operator $T\in\mathcal{L}(\LTWO(\Omega))$, which is the 
bounded linear operator $T:\LTWO(\Omega)\rightarrow\LTWO(\Omega)$
mapping the forcing term $\fs$ to $\uss=:T\fs$:
\begin{align*}
  \label{eq:T}
  \left\{
    \begin{array}{l}
      T\fs\in\Vspace\textrm{~such~that}\\
      \as(T\fs,\vs) = \bs(\fs,\vs)\quad\forall\vs\in\Vspace.
    \end{array}
  \right.
\end{align*}
Operator $T$ is self-adjoint and positive definite with respect to the
inner products $\as(\cdot, \cdot)$ and $\bs(\cdot, \cdot)$ on $\Vspace$, and compact due to the
compact embedding of $\HONE(\Omega)$ in $\LTWO(\Omega)$.

\medskip
Similarly, we associate problem~\eqref{eq:discreteSource} with its
solution operator $T_{\hh}\in\mathcal{L}(\LTWO(\Omega))$ and
problem~\eqref{eq:discreteSource2} with its solution operator
$\widetilde{T}_{\hh}\in\mathcal{L}(\LTWO(\Omega))$.
The former is the bounded linear operator mapping the forcing term
$\fs$ to $\ussh=:T_{\hh}\fs$ and satisfies:
\begin{equation*}
  \left\{
    \begin{array}{l}
      \label{eq:Th}
      T_{\hh}\fs\in\Vhk\textrm{~such~that}\\
      \ash(T_{\hh}\fs,\vsh) = \bsh(\fs,\vsh)\quad\forall\vsh\in\Vhk.
    \end{array}
  \right.
\end{equation*}
The latter is the bounded linear operator mapping the forcing term
$\fs$ to $\ussht=:\widetilde{T}_{\hh}\fs$ and satisfies:
\begin{align*}
  \left\{
    \begin{array}{l}
      \label{eq:Thtilde}
      \widetilde{T}_{\hh}\fs\in\Vhk\textrm{~such~that}\\
      \ash(\widetilde{T}_{\hh}\fs,\vsh) = \bsht(\fs,\vsh)\quad\forall\vsh\in\Vhk.
    \end{array}
  \right.
\end{align*}

Both operators $T_{\hh}$ and $\widetilde{T}_{\hh}$ are self-adjoint
and positive definite with respect to the inner products
$\ash(\cdot, \cdot)$, $\bsh(\cdot, \cdot)$ and $\ash(\cdot, \cdot)$, $\bsht(\cdot, \cdot)$.
They are also compact since their ranges are finite dimensional.

\medskip
\noindent
The eigensolutions of the continuous problem~\eqref{eq:eigPbm} and the
discrete problems ~\eqref{eq:discreteEigPbm}
and~\eqref{eq:discreteEigPbm2} are respectively related to the
eigenmodes of the operators $T$, $T_{\hh}$, and $\widetilde{T}_{\hh}$.
In particular, $(\lambda,\us)$ is an eigenpair of
problem~\eqref{eq:eigPbm} if and only if $T\us=(1\slash{\lambda})\us$,
i.e. $(\frac{1}{\lambda}, \us)$ is an eigenpair for the operator $T$,
and analogously for problems ~\eqref{eq:discreteEigPbm}
and~\eqref{eq:discreteEigPbm2} and operators $T_{\hh}$ and
$\widetilde{T}_{\hh}$.
By virtue of this correspondence, the convergence analysis can be
derived from the spectral approximation theory for compact operators.
In the rest of this section we refer only to operators $T$ and
$\widetilde{T}_{\hh}$.
Identical considerations hold for operators $T$ and $T_{\hh}$
and we omit them for brevity.

\medskip
A sufficient condition for the correct spectral approximation of a
compact operator $T$ is the uniform convergence to $T$ of the family
of discrete operators $\{\widetilde{T}_{\hh}\}_{\hh}$ (see~\cite[Proposition
7.4]{Boffi:2010}, cf. also~\cite{Babuska-Osborn:1991}):
\begin{equation}
  \label{eq:unifconv}
  \norm{T-\widetilde{T}_{\hh}}{\mathcal{L}(\LTWO(\Omega))}\to 0, \quad\textrm{as~}\hh\to 0,
\end{equation}
or, equivalently, 
\begin{equation}
  \label{eq:unifconv2}
  \norm{(T-\widetilde{T}_{\hh})\fs}{0}\leq C\rho(\hh)\norm{\fs}{0}
  \quad\forall\fs\in\LTWO(\Omega),
\end{equation}
with $\rho(\hh)$ tending to zero as $\hh$ goes to zero. 
Condition\eqref{eq:unifconv2} usually follows by a-priori estimates
with no additional regularity assumption on $\fs$. 
Besides the convergence of the eigenmodes,
condition~\eqref{eq:unifconv}, or the equivalent
condition~\eqref{eq:unifconv2}, implies that no spurious eigenvalues
may pollute the spectrum.
In fact, 
\renewcommand{\theenumi}{\roman{enumi}}
\renewcommand{\labelenumi}{(\theenumi)}
\begin{enumerate}
\item each continuous eigenvalue is approximated by a number of
  discrete eigenvalues (counted with their multiplicity) that
  corresponds exactly to its multiplicity;
\item each discrete eigenvalue approximates a continuous eigenvalue.
\end{enumerate}

Condition~\eqref{eq:unifconv} does not provide any indication on the
approximation rate. 
It is common to split the convergence analysis for
eigenvalue problems into two steps: first, the convergence and the
absence of spurious modes is studied; then, suitable convergence
rates are proved.
We now report the main results about the spectral approximation for compact operators.
(cf.~\cite[Theorems 7.1--7.4]{Babuska-Osborn:1991}; see
also~\cite[Theorem 9.3--9.7]{Boffi:2010}),
which deal with the order of
convergence of eigenvalues and eigenfunctions.
\begin{theorem}
  \label{thm:BOeigfun}
  Let the uniform convergence~\eqref{eq:unifconv} holds true. 
  Let $\mu$ be an eigenvalue of $T$, with multiplicity $m$, and denote 
  the corresponding eigenspace by $E_{\mu}$. 
  Then, exactly $m$ discrete eigenvalues $\widetilde{\mu}_{1, \hh}, \dots, \widetilde{\mu}_{m, \hh}$
  (repeated according to their respective multiplicities) converges to $\mu$.  
  Moreover, let $\widetilde{E}_{\mu,\hh}$ be the direct sum of the eigenspaces
  corresponding to the discrete eigenvalues
  $\widetilde{\mu}_{1,\hh},\cdots,\widetilde{\mu}_{m,\hh}$ converging to $\mu$.
  Then
  \begin{equation}
    \delta(E_{\mu}, \widetilde{E}_{\mu,\hh})\leq
    C \norm{(T-\widetilde{T}_{\hh})_{|E_{\mu}}}{\mathcal{L}(\LTWO(\Omega))},
    \label{eq:eigenmodeRate}
  \end{equation}
  with 
  \begin{align*}
    \delta(E_{\mu},\widetilde{E}_{\mu,\hh}) = \max( \hat\delta(E_{\mu},\widetilde{E}_{\mu,\hh}) , \hat\delta(\widetilde{E}_{\mu,\hh},E_{\mu}) ) 
  \end{align*}
   where, in general,
   \[
     \hat\delta(U, \, W) = \sup_{\us\in U, \NORM{\us}{0} = 1}\inf_{w\in W}\norm{\us-w}{0}
  \]
  denotes the gap between $U$, $W \subseteq L^2(\Omega)$.
\end{theorem}

Concerning the eigenvalue approximation error, we recall the following result.

\begin{theorem}
  \label{th:BOeig}
  Let the uniform convergence~\eqref{eq:unifconv} holds true. 
  Let $\phi_1,\ldots,\phi_m$ be a basis of the eigenspace $E_{\mu}$ of
  $T$ corresponding to the eigenvalue $\mu$.
  Then, for $i=1,\ldots,m$
  \begin{equation}
    \label{eq:eigRate}
    \abs{\mu - \widetilde{\mu}_{i,\hh}} \le C 
    \displaystyle
    \Big(\,\sum_{j,k=1}^m \abs{ b((T-\widetilde{T}_{\hh})\phi_k,\phi_j) } + 
    \norm{ (T-\widetilde{T}_{\hh})_{|E_{\mu}} }{\mathcal{L}(\LTWO(\Omega))}^2\,\Big),
  \end{equation}
  where $\widetilde{\mu}_{1,\hh},\ldots,\widetilde{\mu}_{m,\hh}$ are the $m$ discrete
  eigenvalues converging to $\mu$ repeated according to their
  multiplicities.
\end{theorem}

\section{Convergence analysis and error estimates}
\label{sc:analysis}

In this section we study the convergence of the discrete eigenmodes
provided by the VEM approximation to the continuous ones.
We will consider the stabilized discrete
formulation~\eqref{eq:discreteEigPbm2}. The analysis can be easily applied 
to the non--stabilized one~\eqref{eq:discreteEigPbm}.

\subsection{Convergence analysis for the stabilized formulation}
\label{subsec:convnonstab}
In the case of the first VEM approximation of
problem~\eqref{eq:eigPbm}, which uses the stabilized form
$\bsht(\cdot,\cdot)$, the uniform convergence of the sequence of
operators ${\widetilde{T}_{\hh}}$ to $T$ directly stems from the $\LTWO$-\emph{a
  priori} error estimate of Theorem~\ref{thm:aprioriestimate}.

\begin{theorem}
  \label{thm:thm1}
  The family of operators $\widetilde{T}_{\hh}$ associated with
  problem~\eqref{eq:discreteSource2} converges uniformly to the operator
  $T$ associated with problem~\eqref{eq:sourcePbm}, that is,
  \begin{equation}
    \label{eq:uniformconv}
    \norm{T-\widetilde{T}_{\hh}}{\mathcal{L}(\LTWO(\Omega))}\to 0
    \quad\textrm{for}\quad\hh\to 0.
  \end{equation}
\end{theorem}
\begin{proof}
  Let $\uss$ and $\ussht$ be the solutions to the continuous and the
  discrete source problems~\eqref{eq:sourcePbm}
  and~\eqref{eq:discreteSource2} respectively.
  The $\LTWO$-estimate of Theorem~\ref{thm:aprioriestimate} with
  $\fs\in\LTWO(\Omega)$ and the stability
  condition~\eqref{eq:regularity} imply that
  \begin{align*} 
    \norm{ \uss-\ussht }{0} \leq C\hh^{\min(t+1,2)}\norm{\fs}{0}
  \end{align*}
  with $t=\min(k,r)$, $k\geq 1$ being the order of the method and $r$
  at least in $(1\slash{2},1]$ being the regularity index of the
  solution $\uss\in\HS{1+r}(\Omega)$ to the continuous source problem
  in equation~\eqref{eq:regularity}.
  From this inequality it follows that
  \begin{align*}
    \norm{T-\widetilde{T}_{\hh}}{\mathcal{L}(\LTWO(\Omega))} 
    = \sup_{\fs\in\LTWO(\Omega)}\dfrac{\norm{T\fs-\widetilde{T}_{\hh}\fs}{0}}{\norm{\fs}{0}} 
    = \sup_{\fs\in\LTWO(\Omega)}\dfrac{\norm{\uss-\ussht}{0}}{\norm{\fs}{0}}
    \leq C\hh^{\min(t+1,2)}.
  \end{align*}
\end{proof}
\begin{remark}
   \label{rm:regularity}
   We observe that if $f\in\mathcal{E}_{\mu}$ then, thanks to the $L^2$ a priori error estimate 
   in Remark~\eqref{rm:estimate-regular-load}, it holds 
   \begin{align*}
   \norm{(T-\widetilde{T}_{\hh})_{|E_{\mu}}}{\mathcal{L}(\LTWO(\Omega))}
   = \sup_{\fs\in\mathcal{E}_{\mu}}\dfrac{\norm{T\fs-\widetilde{T}_{\hh}\fs}{0}}{\norm{\fs}{0}} 
    = \sup_{\fs\in\mathcal{E}_{\mu}}\dfrac{\norm{\uss-\ussht}{0}}{\norm{\fs}{0}}
    \leq C\hh^{t+1}.
   \end{align*}
\end{remark}
Putting together Theorem~\ref{thm:BOeigfun}~ ,Theorem~\ref{thm:thm1}, and Remark~\ref{rm:regularity}, we can state the following result. 
 \begin{theorem}
   Let $\mu$ be an eigenvalue of $T$, with multiplicity $m$, and
   denote the corresponding eigenspace by $E_{\mu}$.  
   Then, exactly $m$ discrete eigenvalues $\widetilde{\mu}_{1, \hh},
   \dots, \widetilde{\mu}_{m, \hh}$ (repeated according to their
   respective multiplicities) converges to $\mu$.
  Moreover, let $\widetilde{E}_{\mu,\hh}$ be the direct sum of the eigenspaces
  corresponding to the discrete eigenvalues
  $\widetilde{\mu}_{1,\hh},\cdots,\widetilde{\mu}_{m,\hh}$ converging to $\mu$.
  Then
  \begin{equation}
    \delta(E_{\mu}, \widetilde{E}_{\mu,\hh})\leq
    C \hh^{t+1}.
  \end{equation}
 \end{theorem}

A direct consequence of the previous result
(cf.~\cite{Babuska-Osborn:1991,Boffi:2010}) is the following one.
\begin{theorem}
  \label{th:convl2}
  Let $\us$ be a unit eigenfunction associated with the eigenvalue
  $\lambda$ of multiplicity $m$ and let $\wsht^{(1)},\ldots,\wsht^{(m)}$
  denote linearly independent eigenfunctions associated with the $m$
  discrete eigenvalues converging to $\lambda$. 
  Then there exists $\usht\in\SPAN{\wsht^{(1)},\ldots,\wsht^{(m)}}$
  such that
  \begin{align*}
    \norm{\us-\usht}{0}\leq C\hh^{t+1},
  \end{align*}
  where $t=\min\{k,r\}$, being $k$ the order of the method and $r$ the
  regularity index of $\us$.
\end{theorem}

\medskip
Using Theorem~\ref{thm:aprioriestimate} we can obtain an estimate of
the conformity error~\eqref{eq:def:Nh} better than the one in
Lemma~\eqref{le:Nh} when its argument are the solution to the
continuous problem~\eqref{eq:sourcePbm} and the discrete
problem~\eqref{eq:discreteSource2}.
It is worth noting that the solution to the discrete problem does not
need to be an approximation of the solution to the continuous problem, cf.~\cite{Boffi:2010}. 
\begin{lemma}\label{lemma:Nh:def}
  Consider $\us\in\HS{1+r}(\Omega)$, $\vs\in\LTWO(\Omega)$ and let
  $T\us,T\vs\in\HS{1+r}(\Omega)$ with $r >1/2$ be, 
  respectively, the solutions to problem~\eqref{eq:sourcePbm} 
  with load term $\us$ and $\vs$.
  Assume that \ASSUM{A0} is satisfied and let
  $\widetilde{T}_{\hh}\us\in\Vhk\subset\HONEnc(\Th;k)$ for some integer $k\geq1$
  be the virtual element approximation of $T\us$ that solves
  problem~\eqref{eq:discreteSource2}.
  Then, there exists a constant $C>0$, independent of $\hh$, 
  such that
  \begin{align}
    \abs{\calNh(T\vs, \widetilde{T}_{\hh}\us)} 
    \leq C\hh^{2t}\,\snorm{T\vs}{1+r}\,\snorm{T\us}{1+r}
  \end{align}
  where $t=\min\{k,r\}$ and $\calNh(\us,\vs)$ is the conformity error
  defined in \eqref{eq:def:Nh}.
\end{lemma}
\begin{proof}
  We start the following chain of developments from the definition of
  the conformity error given in~\eqref{eq:def:Nh}, note that
  $\jump{T\us}=0$ on every mesh side and that the moments  
  up to order $k-1$ of $\jump{\widetilde{T}_h\us}$ across all
  mesh interfaces are zero, and apply the Cauchy-Schwarz
  inequality in the last two steps:
  \begin{align*}
    \calNh(T\vs,\widetilde{T}_{\hh}\us) 
    & = \sum_{\S\in\Sset_{\hh}}\int_{\S}\nabla T\vs\cdot\jump{\widetilde{T}_{\hh}\us}\, {\rm d}\S \\
    & = \sum_{\S\in\Sset_{\hh}}\int_{\S}\big(I-\PiSz{k-1}\big)\nabla T\vs\cdot\big(I-\PiSz{0}\big)\jump{\widetilde{T}_{\hh}\us}\, {\rm d}\S \\
    & = \sum_{\S\in\Sset_{\hh}}\int_{\S}\big(I-\PiSz{k-1}\big)\nabla T\vs\cdot\Big(
    \big(I-\PiSz{0}\big)\jump{\widetilde{T}_{\hh}\us})-\big(I-\PiSz{0}\big)\jump{T\us}\Big)\, {\rm d}\S \\
    & = \sum_{\S\in\Sset_{\hh}}\int_{\S}\big(I-\PiSz{k-1}\big)\nabla T\vs\cdot\big(I-\PiSz{0}\big)\jump{(\widetilde{T}_{\hh}-T)\us}\, {\rm d}\S\\[0.35em]
    & \leq\sum_{\S\in\Sset_{\hh}}\norm{\big(I-\PiSz{k-1}\big)\nabla T\vs\cdot\norS}{0,\S}\,\norm{(I-\PiSz{0})\jump{(\widetilde{T}_{\hh}-T)\us}\cdot\norS}{0,\S}\\
    & \leq
    \left[\,\sum_{\S\in\Sset_{\hh}}\norm{\big(I-\PiSz{k-1}\big)\nabla T\vs\cdot\norS}{0,\S}^2\right]^{\frac{1}{2}}\times
    \left[\,\sum_{\S\in\Sset_{\hh}}\norm{(I-\PiSz{0})\jump{(\widetilde{T}_{\hh}-T)\us}\cdot\norS}{0,\S}^2\right]^{\frac{1}{2}}\\[0.5em]
    &= \mathcal{N}_1\times\mathcal{N}_2.
  \end{align*}
  Trace inequality~\eqref{eq:approx:pitrace} yields 
  \begin{align*}
    \norm{(I-\Pi_{k-1}^{0,\S})\nabla T\vs\cdot\norS}{0,\S}
    \leq C\hS^{t-\frac{1}{2}}\snorm{T\vs}{1+r,\Omega_{\S}},
  \end{align*}
  and summing over all the mesh sides, noting that $\hS\leq\hh$, the
  number of sides per element is uniformly bounded due to \ASSUM{A0}
  and using definition~\eqref{eq:norm-broken} yield
  \begin{align*}
    \abs{\mathcal{N}}_1^2
    = \sum_{\S\in\Sset_{\hh}}\norm{\big(I-\PiSz{k-1}\big)\nabla T\vs\cdot\norS}{0,\S}^2
    \leq C\big(\hh^{t-\frac{1}{2}}\big)^2\sum_{\S\in\Sset_{\hh}}\snorm{T\vs}{1+r,\Omega_{\S}}^2
    \leq C\big(\hh^{t-\frac{1}{2}}\big)^2\sum_{\P\in\Th}\snorm{T\vs}{1+r,\P}^2,
  \end{align*}
  and finally
  \begin{align*}
    \abs{\mathcal{N}_1} \leq C\hh^{t-\frac{1}{2}}\snorm{T\vs}{1+r}.
  \end{align*}
  Similarly, trace inequality~\eqref{eq:approx:pitrace} and the jump definition yield 
  \begin{align*}
    \norm{(I-\PiSz{0})\jump{(\widetilde{T}_{\hh}-T)\us}\cdot\norS}{0,\S}
    \leq C\hS^{\frac{1}{2}}\snorm{(\widetilde{T}_{\hh}-T)\us}{1,\Omega_{\S}}
  \end{align*}
  and using the same arguments as above we have that
  \begin{align*}
    \abs{\mathcal{N}_2}^2
    &= \sum_{\S\in\Sset_{\hh}}\norm{(I-\Pi^{0,\S}_{0})\jump{(\widetilde{T}_{\hh}-T)\us}\cdot\norS}{0,\S}^2 
    \leq C\hh\sum_{\S\in\Sset_{\hh}}\snorm{(\widetilde{T}_{\hh}-T)\us}{1,\Omega_{\S}}^2 
    \leq C\hh\sum_{\P\in\Th}\snorm{(\widetilde{T}_{\hh}-T)\us}{1,\P}^2\\[0.5em]
    &= C\hh\snorm{(\widetilde{T}_{\hh}-T)\us}{1,\hh}^2.
  \end{align*}
  Using this relation and the \emph{a priori} error estimate in the
  energy norm of Remark~\ref{rm:estimate-regular-load} finally yield:
  \begin{align*}
    \abs{\mathcal{N}}_2\leq\hh^{t+\frac{1}{2}}\snorm{T\us}{1+r}.
  \end{align*}
  The assertion of the lemma follows by collecting the above estimates
  together.
\end{proof}

We now prove the usual double order convergence of the eigenvalues.
\begin{theorem}
  \label{theorem:double:convergence:rate}
  Let $\lambda$ be an eigenvalue of problem~\eqref{eq:eigPbm} with multiplicity $m$, 
  and denote by $\widetilde{\lambda}_{1,h},\cdots,\widetilde{\lambda}_{m,h}$ the $m$ discrete eigenvalues 
  converging towards $\lambda$. 
  Then the following optimal double order convergence holds:
  \begin{align}
    \label{eq:double:convergence:rate}
    \abs{\lambda - \widetilde{\lambda}_{i,\hh}}
    \leq C \hh^{2t}\quad\forall i=1,\ldots,m, 
  \end{align}
  with $t=\min\{k,r\}$, being $k$ the order of the method and $r$ the
  regularity index of the eigenfunction corresponding to $\lambda$.
\end{theorem}
\begin{proof}
  We use the result stated in Theorem~\ref{th:BOeig}. 
  It is clear that the second term in the estimate of
  Theorem~\ref{th:BOeig} is of double order compared to the
  $\HONE$-rate of convergence. 
  Hence, we analyse in detail the term
  $b \left((T-\widetilde{T}_{\hh})\phi_j,\phi_k\right)$.
  Let $\us$ and $\vs$ be two eigenfunctions associated with the
  eigenvalue $\lambda$. 
  Then, we note that $(T-\widetilde{T}_{\hh})\us\in\HONEnc(\Th;k)$ and begin the
  chain of developments that follows from the definition of the
  conformity error in~\eqref{eq:def:Nh}:
  \begin{align}
    \label{eq:eigproof}
    \begin{array}{rll}
      &b \left( (T-\widetilde{T}_{\hh})\us,\vs \right) 
      = \as(T\vs,(T-\widetilde{T}_{\hh})\us) - \mathcal{N}_{\hh}(T\vs,(T-\widetilde{T}_{\hh})\us)                                              & \mbox{\big[\textrm{note~that~}$\mathcal{N}_{\hh}(T\vs,T\us)=0$\big]}     \nonumber\\[0.5em]
      &\qquad = \as(T\vs,(T-\widetilde{T}_{\hh})\us) + \mathcal{N}_{\hh}(T\vs,\widetilde{T}_{\hh}\us)                                          & \mbox{\big[\textrm{add and subtract~}$\widetilde{T}_{\hh}\vs$ in the first term\big]}\nonumber\\[0.5em]
      &\qquad = \as((T-\widetilde{T}_{\hh})\vs,(T-\widetilde{T}_{\hh})\us)  + \as(\widetilde{T}_{\hh}\vs,(T-\widetilde{T}_{\hh})\us) + \mathcal{N}_{\hh}(T\vs,\widetilde{T}_{\hh}\us) & \mbox{\big[split the middle term and use~\eqref{eq:def:Nh}\big]}          \nonumber\\[0.5em]
      &\qquad = \as((T-\widetilde{T}_{\hh})\vs,(T-\widetilde{T}_{\hh})\us) + \bs(\widetilde{T}_{\hh}\vs,\us) + \mathcal{N}_{\hh}(T\us,\widetilde{T}_{\hh}\vs)           & \nonumber\\
      &\qquad\phantom{=}- \as(\widetilde{T}_{\hh}\vs,\widetilde{T}_{\hh}\us) + \mathcal{N}_{\hh}(T\vs,\widetilde{T}_{\hh}\us)                              &  \mbox{\big[add both sides of~\eqref{eq:discreteSource2} \big]}           \nonumber\\[0.5em]
      &\qquad = \as((T-\widetilde{T}_{\hh})\vs,(T-\widetilde{T}_{\hh})\us) + \big[\bs(\us, \widetilde{T}_{\hh}\vs) - \bsht(\us, \widetilde{T}_{\hh}\vs)\big] +         & \nonumber\\
      &\qquad\phantom{=}
      \big[ \ash(\widetilde{T}_{\hh}\vs,\widetilde{T}_{\hh}\us) - \as(\widetilde{T}_{\hh}\vs,\widetilde{T}_{\hh}\us) \big] +
      \big[ \mathcal{N}_{\hh}(T\us,\widetilde{T}_{\hh}\vs) + \mathcal{N}_{\hh}(T\vs,\widetilde{T}_{\hh}\us) \big]
      \\[0.5em] &\qquad 
      = \sum_{i=1}^{4}\mathsf{R}_i.
    \end{array}
  \end{align}
  Term $\mathsf{R}_1$ is clearly of order  $h^{2t}$, being $\us$ and $\vs$ eigenfunctions (see Remark~\eqref{rm:estimate-regular-load}) :
  \begin{align*}
    \abs{\mathsf{R}_1} 
    = \abs{ \as((T-\widetilde{T}_{\hh})\vs,(T-\widetilde{T}_{\hh})\us)}
    \leq \snorm{(T-\widetilde{T}_{\hh})\vs}{1,\hh}\,\snorm{(T-\widetilde{T}_{\hh})\us}{1,\hh}
    \leq C\hh^{2t}.
  \end{align*}
  To bound term $\mathsf{R}_2$ using \eqref{eq:discretebstab} and \eqref{eq:ctbot}, the definition of $L^2$-orthogonal projection and triangular inequality, we get
  \begin{equation}
\label{eq:R2}
\begin{split}
  \abs{\mathsf{R}_2} & \leq \abs{\bs(\us,\widetilde{T}_{\hh}\vs) - \bsht(\us, \widetilde{T}_{\hh}\vs)} 
  \leq\sum_{\P\in\Th} \abs{\bsP(\us,\widetilde{T}_{\hh}\vs) - \bshtP(\us, \widetilde{T}_{\hh}\vs)}  \\	
  &  \leq\sum_{\P\in\Th} \left( \left|\bsP(\us-\Piz{k}\us,\widetilde{T}_{\hh}\vs)\right| + 
  \left| \SPt\left( (I-\Piz{k})\us, (I-\Piz{k})\widetilde{T}_{\hh}\vs \right)\right| \right) \\
  & \leq\sum_{\P\in\Th} \left( \big|\bsP((I-\Piz{k})\us,(I-\Piz{k})\widetilde{T}_{\hh}\vs)\big| +
  \cttop \norm{(I-\Piz{k})\us}{0} \norm{(I-\Piz{k})\widetilde{T}_{\hh}\vs}{0} 
   \right)\\
  & \leq
  \sum_{\P\in\Th}(1 + \cttop)\norm{(I-\Piz{k})\us}{0,P}\norm{(I-\Piz{k})\widetilde{T}_{\hh}\vs}{0,P}
  \\
  &\leq C \sum_{\P\in\Th}
  \left(
  \norm{(I-\Piz{k})\us}{0,P} \left( 
  \norm{(I-\Piz{k})(T -\widetilde{T}_{\hh})\vs}{0,P} +    \norm{(I-\Piz{k})T\vs}{0,P}  
  \right) 
  \right)
  \end{split}
\end{equation}
  Now, note that
  \begin{align}
  \label{eq:r21}
    \norm{\us-\Piz{k}\us}{0,\P}
    \leq C\hh^{\min(k+1,r+1)}\snorm{\us}{r+1,\P}
    \leq C\hh^{t+1}\snorm{\us}{r+1,\P}.
  \end{align}
  By the continuity of $L^2$-projection with respect the $L^2$-norm we get
  \begin{align}
  \label{eq:r22}
  \norm{(I-\Piz{k})(T -\widetilde{T}_{\hh})\vs}{0,P}
    \leq \norm{(T -\widetilde{T}_{\hh})\vs}{0,P}.
  \end{align}
  Moreover polynomial approximation estimate \eqref{eq:approx:pi} yields
  \begin{align}
  \label{eq:r23}
  \norm{(I-\Piz{k})T\vs}{0,P}
    \leq C\hh^{t+1}\snorm{T\vs}{1+r,\P}.
  \end{align}
  Hence collecting \eqref{eq:r21}, \eqref{eq:r22}, \eqref{eq:r23} in \eqref{eq:R2}, 
  and using the $L^2$ \emph{a priori}  error estimate in Remark~\eqref{rm:estimate-regular-load} 
  and the stability estimate~\eqref{eq:regularity},  
  we obtain
  \begin{align*}
  \abs{\bs(\us,\widetilde{T}_{\hh}\vs) - \bsht(\us, \widetilde{T}_{\hh}\vs)} 
  \leq C\hh^{2t+2}\snorm{\us}{r+1} \, \snorm{T\vs}{r+1}
  \leq C\hh^{2t+2} \norm{\us}{0} \, \norm{\vs}{0}=C \hh^{2t+2}.
  \end{align*}
  To estimate term $\mathsf{R}_3$, we first consider the developments:
  \begin{align*}
    \asP(\widetilde{T}_{\hh}\us,\widetilde{T}_{\hh}\vs) - 
    \ashP(\widetilde{T}_{\hh}\us,\widetilde{T}_{\hh}\vs) 
    & = \asP(\widetilde{T}_{\hh}\us - (T\us)_{\pi},\widetilde{T}_{\hh}\vs) + 
    \ashP\big( (T\us)_{\pi},\widetilde{T}_{\hh}\vs \big) - \ashP(\widetilde{T}_{\hh}\us,\widetilde{T}_{\hh}\vs) \\
    & = \asP(\widetilde{T}_{\hh}\us - (T\us)_{\pi},\widetilde{T}_{\hh}\vs - (T\vs)_{\pi}) + \asP_h(\widetilde{T}_{\hh}\us - (T\us)_{\pi}, (T\vs)_{\pi})          \\ 
    & \phantom{=} + \ashP\big( (T\us)_{\pi} - \widetilde{T}_{\hh}\us, \widetilde{T}_{\hh}\vs \big)                                                \\
    & = \asP(\widetilde{T}_{\hh}\us - (T\us)_{\pi}, \widetilde{T}_{\hh}\vs - (T\vs)_{\pi} ) + \ashP( \widetilde{T}_{\hh}\us - (T\us)_{\pi}, (T\vs)_{\pi} - \widetilde{T}_{\hh}\vs ), 
  \end{align*}
  where we make use of the consistency
  condition~\eqref{eq:k-consistency} to introduce $(T\us)_{\pi}$ and
  $(T\vs)_{\pi}$, the elemental polynomial approximations of $T\us$
  and $T\vs$ that exist in accordance with~\eqref{eq:approx:pi}.
  The terms on the right-hand side of the previous equation are
  similar and we can estimate both as follows:
  \begin{equation*}
    \abs{ \asP( \widetilde{T}_{\hh}\us - (T\us)_{\pi}, \widetilde{T}_{\hh}\vs - (T\vs)_{\pi} ) } 
    \le \Big( \snorm{ (\widetilde{T}_{\hh}-T)u }{1,\P} + \snorm{ T\us - (T\us)_{\pi} }{1,\P} \Big) 
    \Big( \snorm{ (\widetilde{T}_{\hh} - T)\vs }{1,\P} + \snorm{ T\vs - (T\vs)_{\pi} }{1,\P} \Big),
  \end{equation*}
  and, using the \emph{a priori} error estimate in the broken
  $\HONE$-norm in Remark~\eqref{rm:estimate-regular-load}, 
  the local approximation
  properties of the VEM space by polynomials~\eqref{eq:approx:pi}, 
  and the stability estimate~\eqref{eq:regularity}, it holds
  \begin{align*}
    \abs{\mathsf{R}_3} 
    &\leq\sumP\Big| \asP(\widetilde{T}_{\hh}\us, \widetilde{T}_{\hh}\vs) - \ashP( \widetilde{T}_{\hh}\us,\widetilde{T}_{\hh}\vs ) \Big| \\
    &\leq C
    \Big( \snorm{(\widetilde{T}_{\hh}-T)\us}{1,\hh} + \snorm{T\us-(T\us)_{\pi}}{1,\hh} \Big)
    \Big( \snorm{(\widetilde{T}_{\hh}-T)\vs}{1,\hh} + \snorm{T\vs-(T\vs)_{\pi}}{1,\hh} \Big)  \\
    &  \leq C\hh^{2t}  \snorm{\T\us}{1+r} \snorm{\T\vs}{1+r}  \le C \hh^{2t}.
  \end{align*}  
  Finally, for term $\mathsf{R}_4$ we apply the triangle inequality, 
  Lemma~\ref{lemma:Nh:def}, and the stability estimate~\eqref{eq:regularity} to obtain:
  \begin{align*}
    \abs{\mathsf{R}_4} 
    \leq \abs{\calNh(T\us,\widetilde{T}_{\hh}\vs)} + \abs{\calNh(T\vs,\widetilde{T}_{\hh}\us)} 
    \leq C\hh^{2t}\snorm{T\vs}{1+r}\snorm{T\us}{1+r} \leq C \hh^{2t}.
  \end{align*}
  The assertion of the theorem follows from the above estimates.
\end{proof}

The proof of the optimal error estimate for the eigenfunctions in the 
discrete energy norm follows along the same line as the one for the 
nonconforming finite element method.  
We briefly report it here for the sake of completeness.
\begin{theorem}
  With the same notation as in Theorem~\ref{th:convl2}, we have
  \begin{align*}
    \snorm{\us-\usht}{1,\hh}\le C\hh^{t}, 
  \end{align*}
  where $t=\min(k,r)$, $k$ being the order of the method and $r$ the
  regularity index of $\us$.
\end{theorem}
\begin{proof}
  \begin{equation*}
    \us - \usht = \lambda T\us - \widetilde{\lambda}_{\hh} \widetilde{T}_{\hh}\usht 
    = (\lambda - \widetilde{\lambda}_{\hh})T\us + \widetilde{\lambda}_{\hh}(T - \widetilde{T}_{\hh})\us + \widetilde{\lambda}_{\hh} \widetilde{T}_{\hh}(\us - \usht),
  \end{equation*}
  then 
  \begin{align*}
    \snorm{\us -\usht}{1,h}
    \leq \abs{\lambda - \lambda_{\hh}} \, \snorm{T\us}{1,\hh} + \lambda_{\hh}\snorm{(T-\widetilde{T}_{\hh})\us}{1,\hh} 
    + \widetilde{\lambda}_{\hh} \snorm{\widetilde{T}_{\hh}(\us - \usht)}{1,\hh}.
  \end{align*}
  The first term at the right-hand side of the previous equation is of
  order $\hh^{2t}$, while the second one is of order $\hh^{t}$.
  Finally, for the last term, using \eqref{eq:stability}, the continuity of the operator $\widetilde{T}_{\hh}$, and Theorem \ref{th:convl2}, we obtain
  \begin{align*}
    \snorm{ \widetilde{T}_{\hh}(\us-\usht)}{1,\hh}^2 
    &\leq \dfrac{1}{\alpha_*}\ash(\widetilde{T}_{\hh}(\us - \usht),\widetilde{T}_{\hh}(\us - \usht)) \\
    &= \dfrac{1}{\alpha_*}\bsht(\us-\usht,\widetilde{T}_{\hh}(\us - \usht)) 
    \le C\norm{\us-\usht}{0}^2 
    \le C \hh^{2t+2}.
  \end{align*}
\end{proof}

\section{Numerical experiments}
\label{sc:tests}

\begin{figure}
  \centering
  \begin{tabular}{cccc}
    \begin{overpic}[scale=0.2]{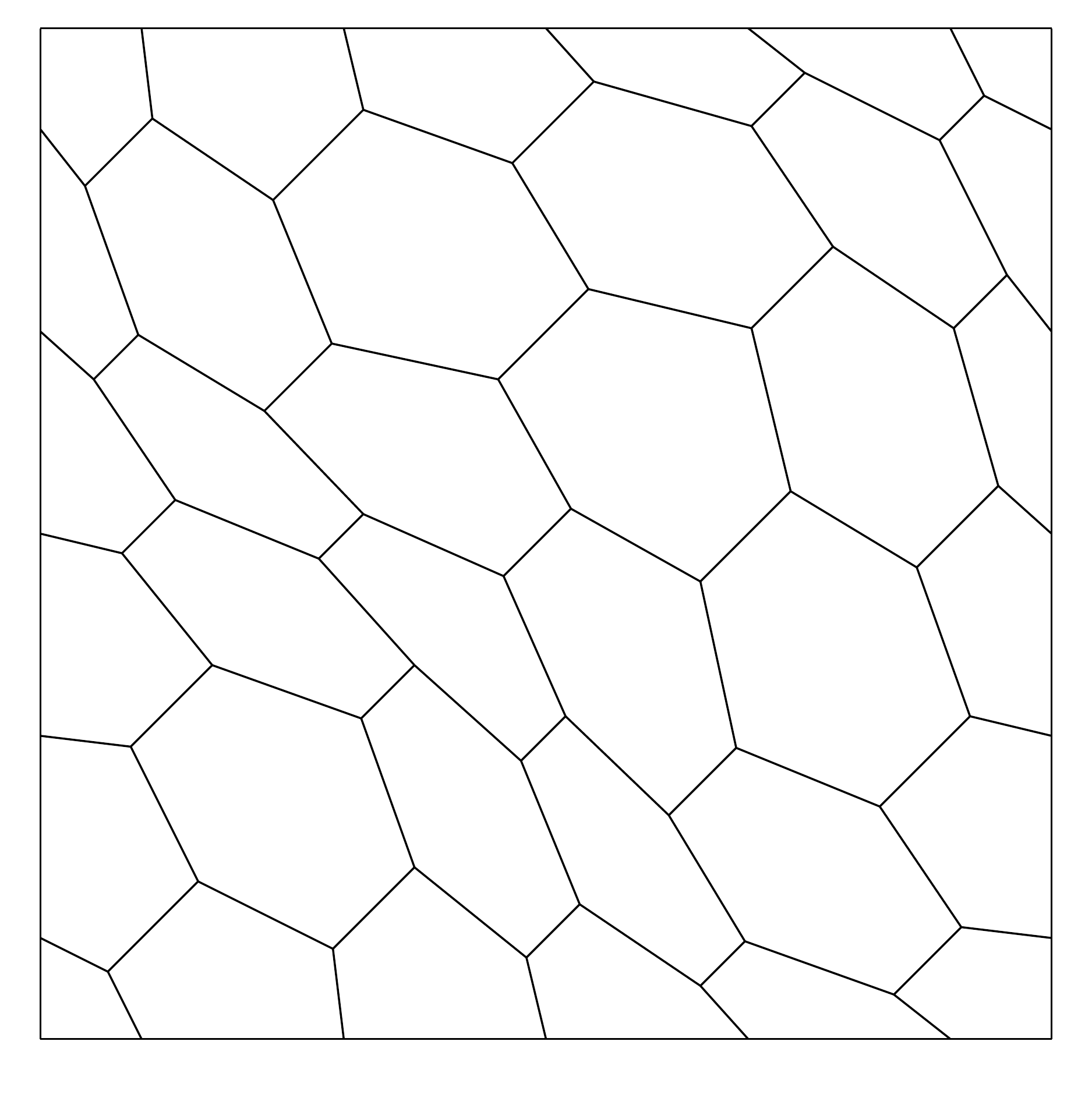}
    \end{overpic} 
    &
    \begin{overpic}[scale=0.2]{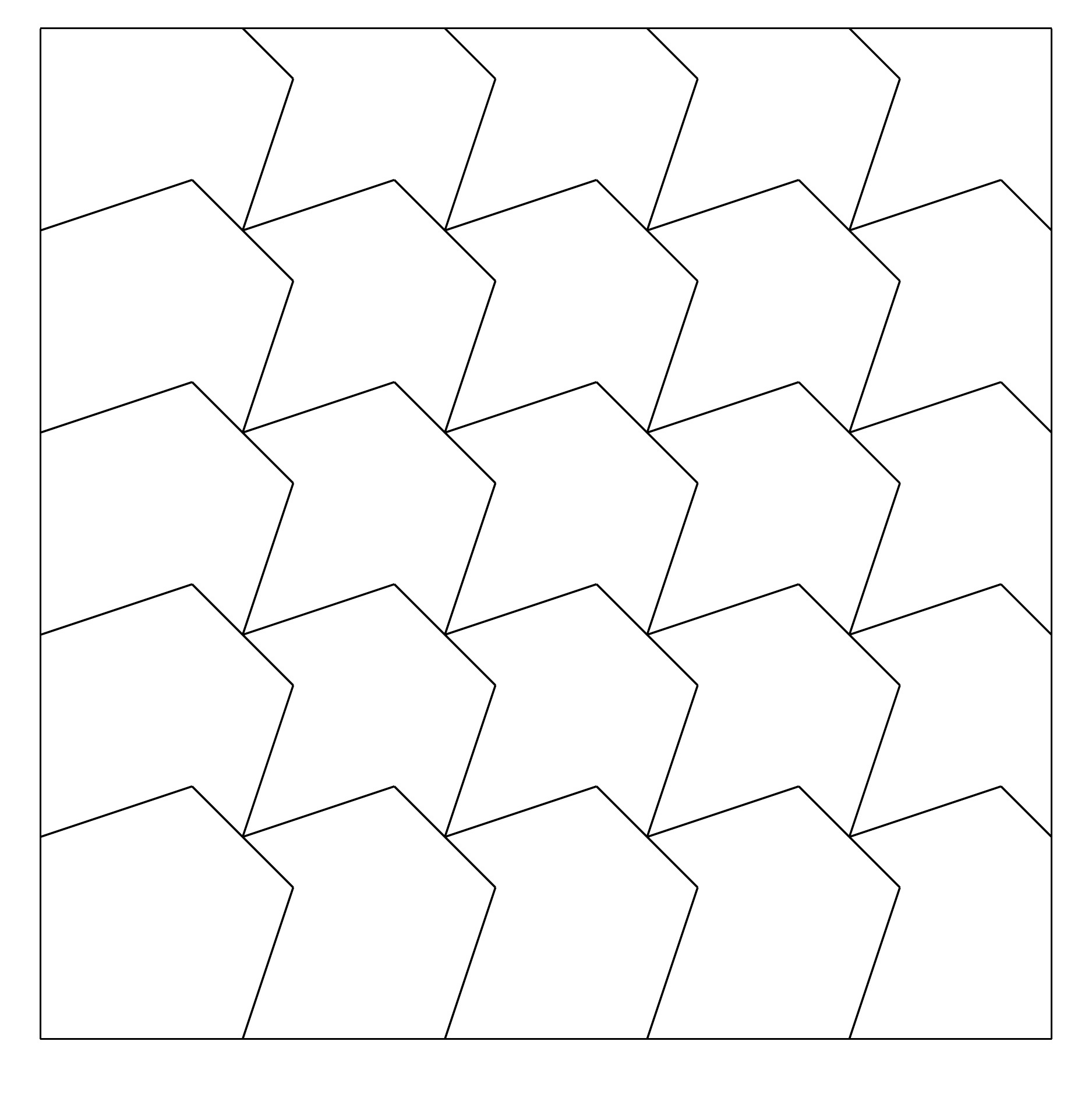}
    \end{overpic}
    &
    \begin{overpic}[scale=0.2]{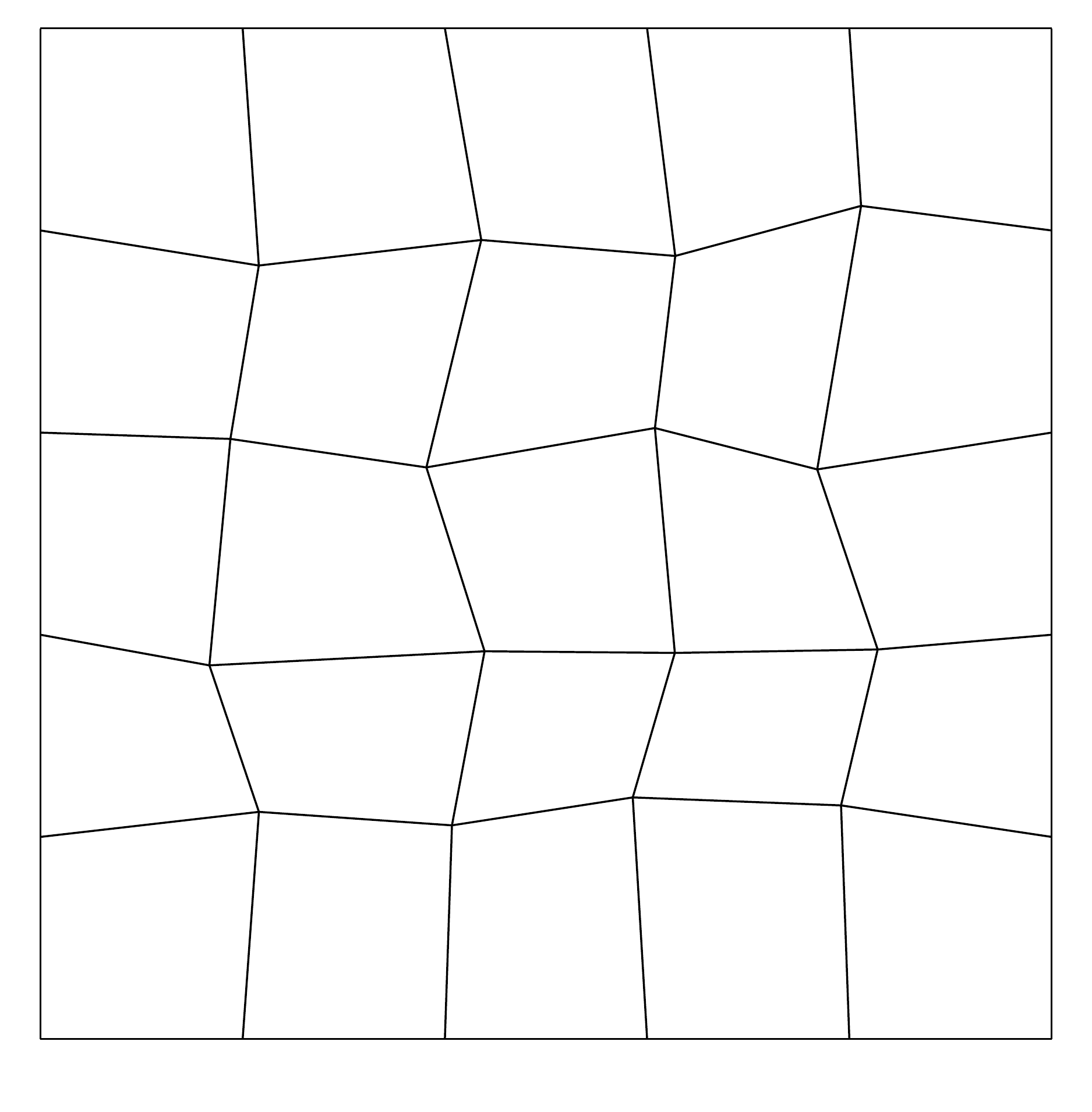}
    \end{overpic} 
    &
    \begin{overpic}[scale=0.2]{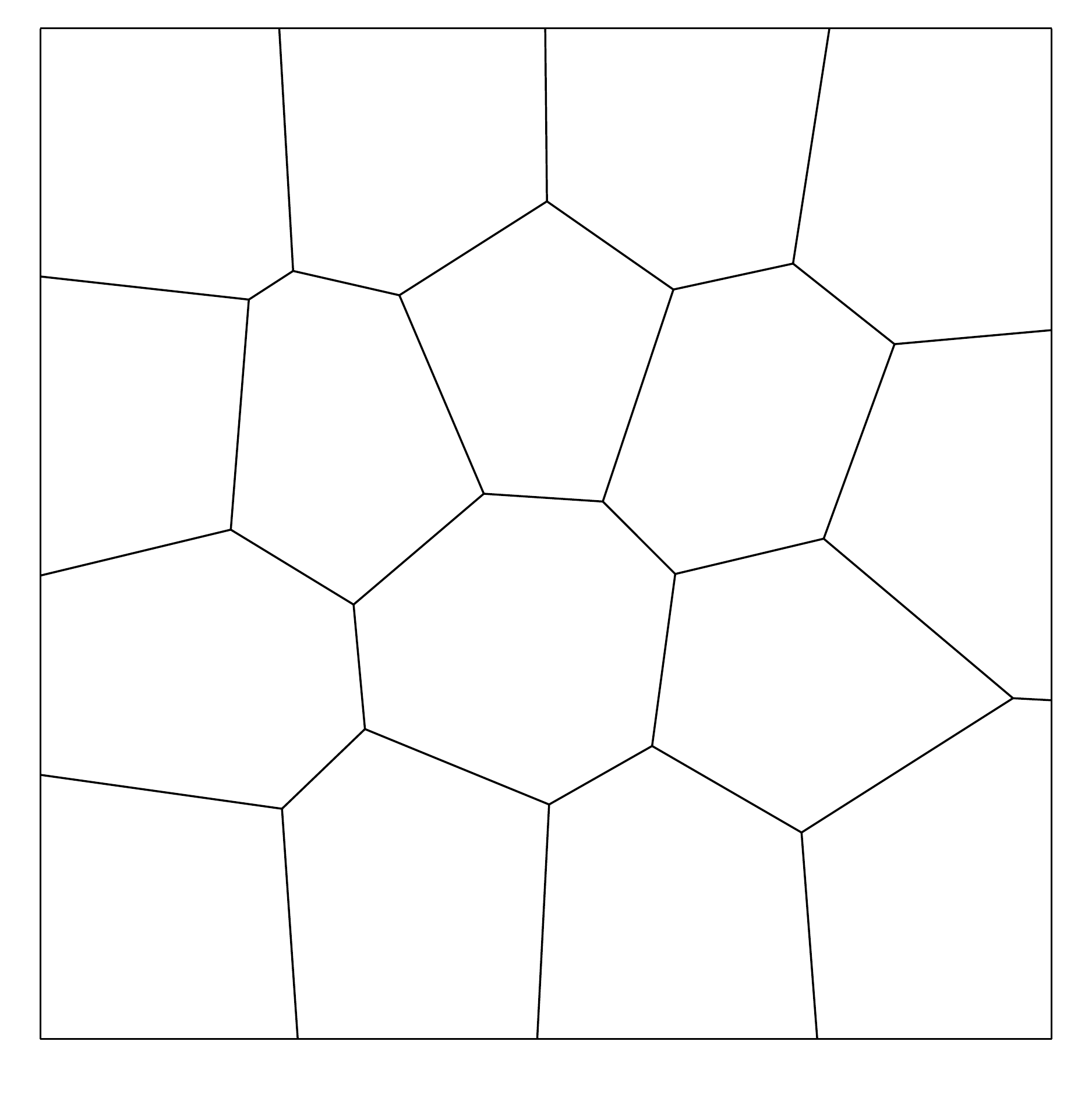}
    \end{overpic}
    \\[0.25em]
    \begin{overpic}[scale=0.2]{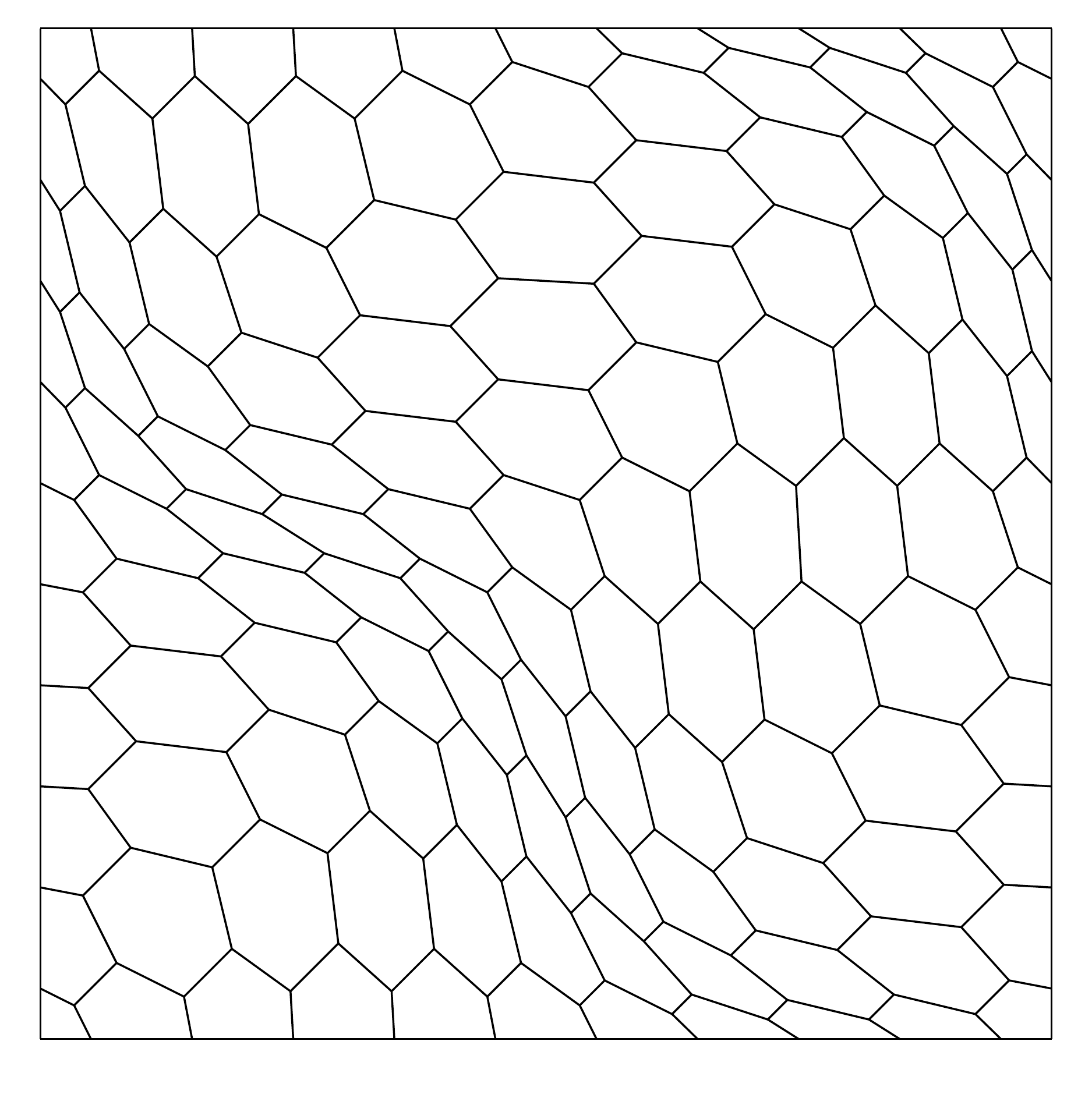}
    \end{overpic} 
    &
    \begin{overpic}[scale=0.2]{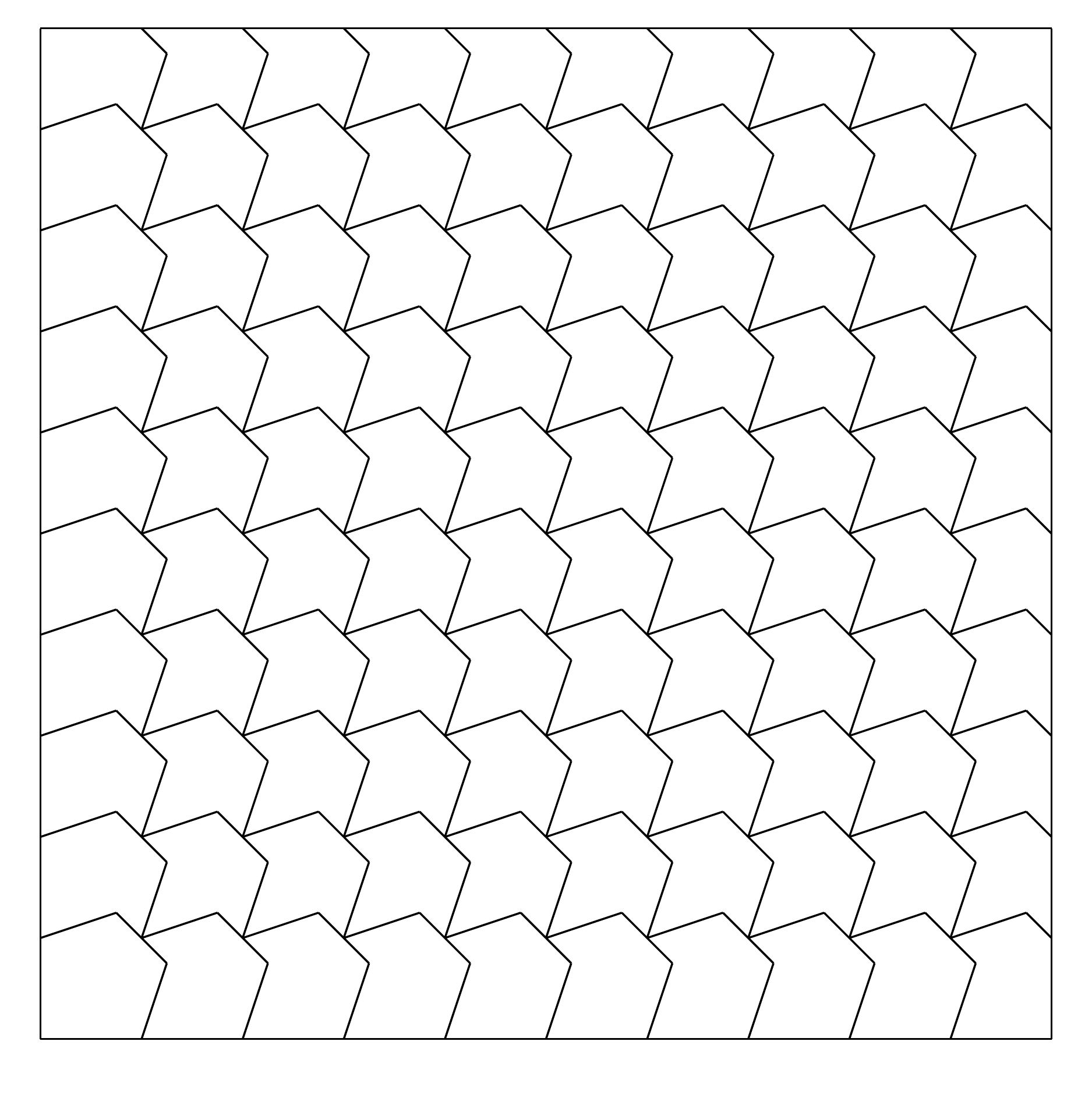}
    \end{overpic}
    &
    \begin{overpic}[scale=0.2]{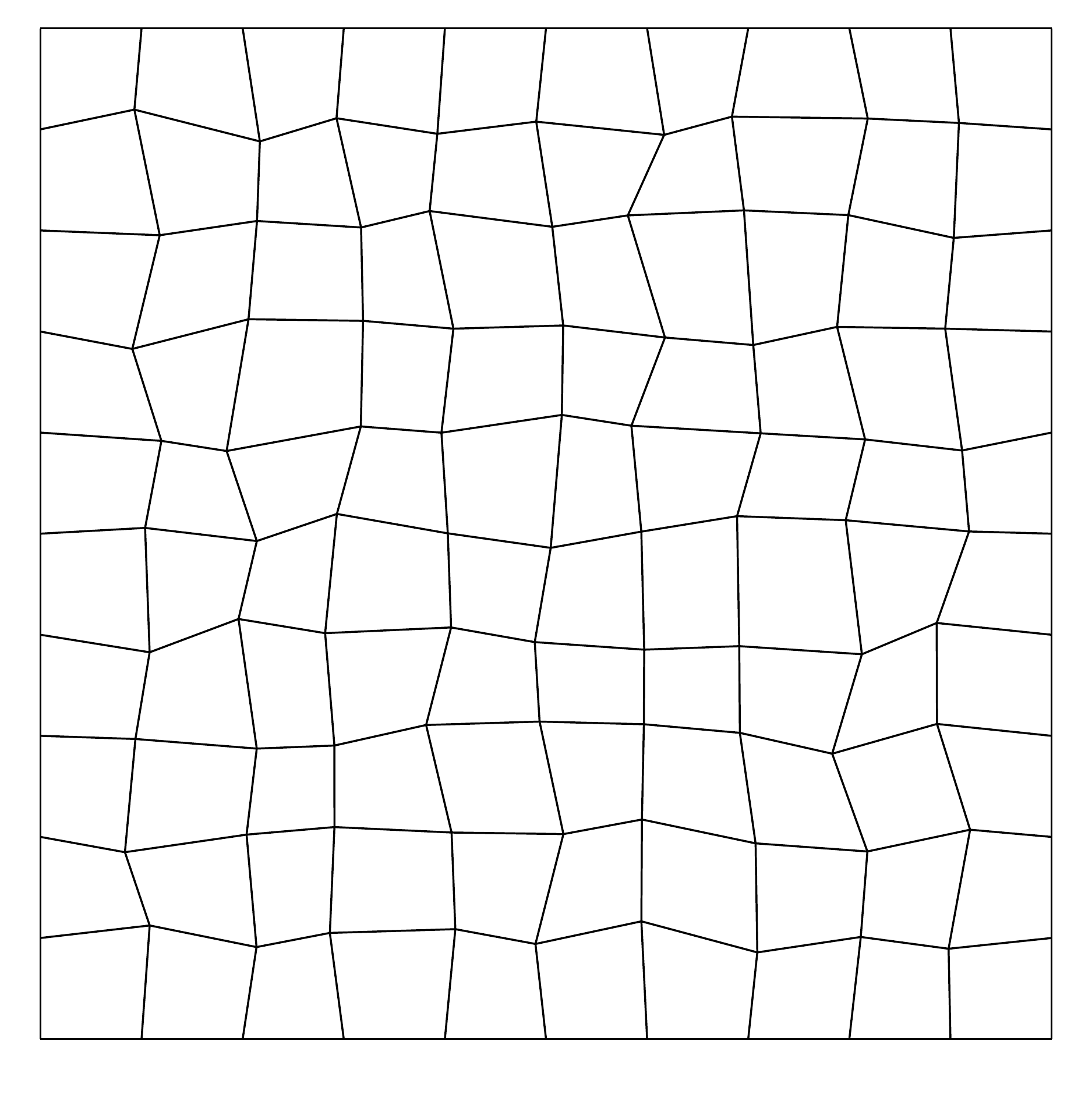}
    \end{overpic} 
    &
    \begin{overpic}[scale=0.2]{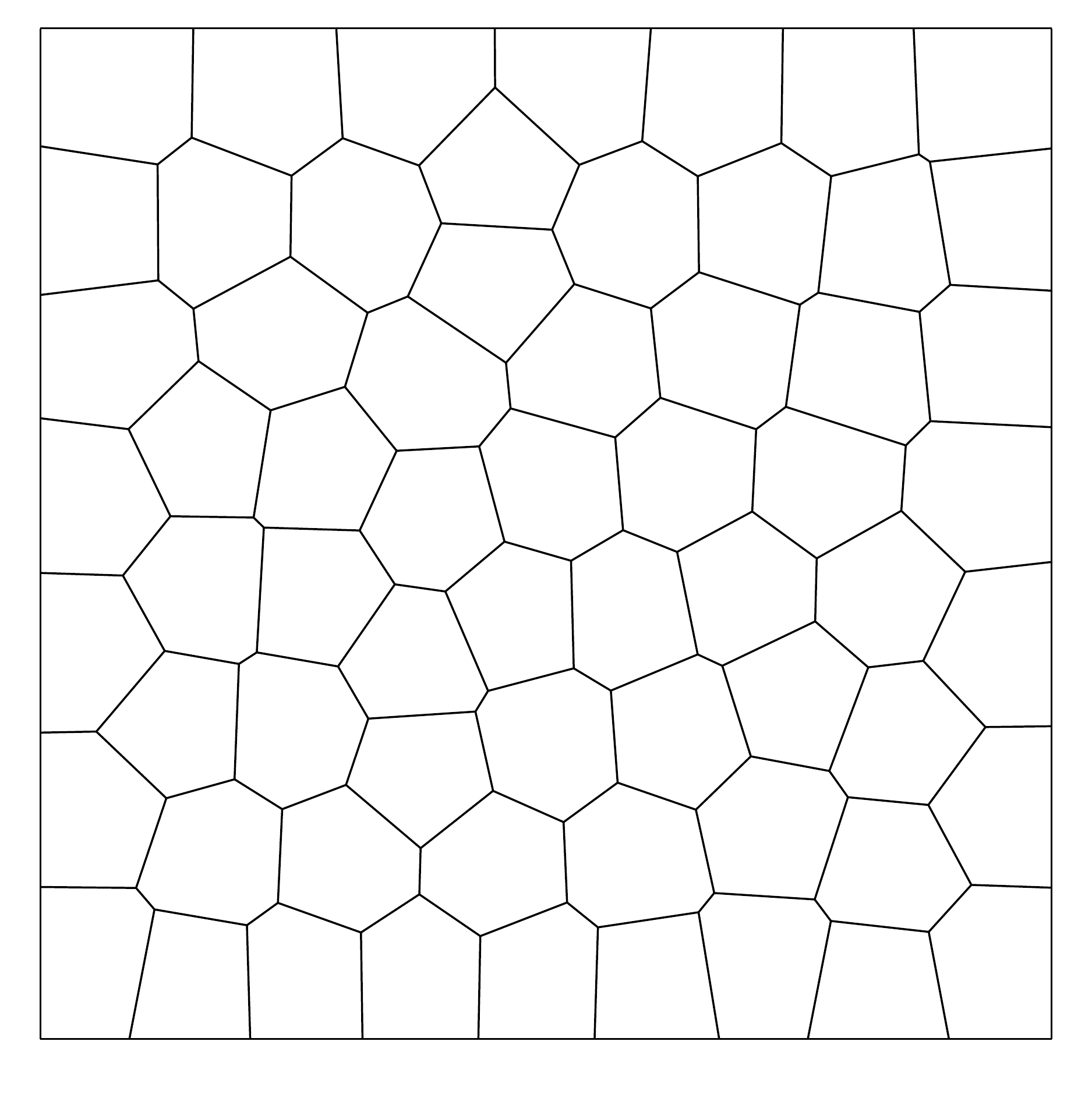}
    \end{overpic}
    \\[-0.25em]
    \textit{Mesh}~1 & \textit{Mesh}~2 & \textit{Mesh}~3 & \textit{Mesh}~4
  \end{tabular}
  \caption{Base meshes (top row) and first refined meshes (bottom row) of the 
    following mesh families from left to right:  mainly hexagonal mesh; nonconvex
    octagonal mesh; randomized quadrilateral mesh; voronoi mesh.}
  \label{fig:Meshes}
\end{figure}
In this section, we aim to confirm the optimal convergence rate of the
numerical approximation of the eigenvalue problem~\eqref{eq:eigPbm}
predicted by Theorem~\ref{theorem:double:convergence:rate} for the
nonconforming virtual element method.
In particular, we present the performance of the nonconforming VEM 
applied to
the eigenvalue problem on a two-dimensional square domain (Test
Case~1) and on the L-shaped domain (Test Case~2).
The convergence of the numerical approximation is shown through the
relative error quantity
\begin{align}
  \epsilon_{\hh,\lambda} := \frac{\abs{\lambda-\lambda_{\hh}}}{\lambda},
\end{align}
where $\lambda$ denotes an eigenvalue of the continuous problem and $\lambda_{\hh}$ its virtual
element approximation.
For Test Case~1, we also compare the error curves for the
nonconforming and the conforming VEM of
Reference~\cite{Gardini-Vacca:2017}.
For both test cases, we use the scalar stabilization for the bilinear
form $\asP(\cdot,\cdot)$ and $\bsP(\cdot,\cdot)$, which reads as follows:
\begin{align*}
  S^{\P}(\vsh,\wsh) &= \sigma_{\P}\vvh^T\wvh,\\
  \widetilde{S}^{\P}(\vsh,\wsh) &= \tau_{\P}\hh^{2}\vvh^T\wvh.
\end{align*}
where $\vvh$, $\wvh$ denote the vector containing the values of the local 
DoFs associated to $\vsh$, $\wsh\in\Vhk(\P)$ and the stability parameters $\sigma_{\P}$ and $\tau_{\P}$ are two
positive constants independent of $\hh$.
In the numerical tests, when $k=1$, constant $\sigma_{\P}$ is the mean
value of the eigenvalues of the matrix stemming from the consistency
part of the local bilinear form $\as^{\P}$, i.e.,
$\as^{\P}(\Pi^{\nabla,\P}_1\cdot,\Pi^{\nabla,\P}_1\cdot)$.
For $k=2,3$, we set $\sigma_{\P}$ to the maximum eigenvalue of
$\as^{\P}(\Pi^{\nabla,\P}_k\cdot,\Pi^{\nabla,\P}_k\cdot)$.
Likewise, when $k=1$, constant $\tau_{\P}$ is set to the mean
value of the eigenvalues of the matrix stemming from the consistency
part of the local bilinear form $\frac{1}{h^2}\bsP(\cdot,\cdot)$, i.e.,
$\frac{1}{h^2}\bsP(\Pi^{0,\P}_1\cdot,\Pi^{0,\P}_1\cdot)$. 
For $k=2,3$, we set $\tau_{\P}$ to the maximum eigenvalue of
$\frac{1}{h^2}\bsP(\Pi^{0,\P}_k\cdot,\Pi^{0,\P}_k\cdot)$.

\subsection{Test~1.}

\begin{figure}
  \centering
  \begin{tabular}{ccc}
    \begin{overpic}[scale=0.325]{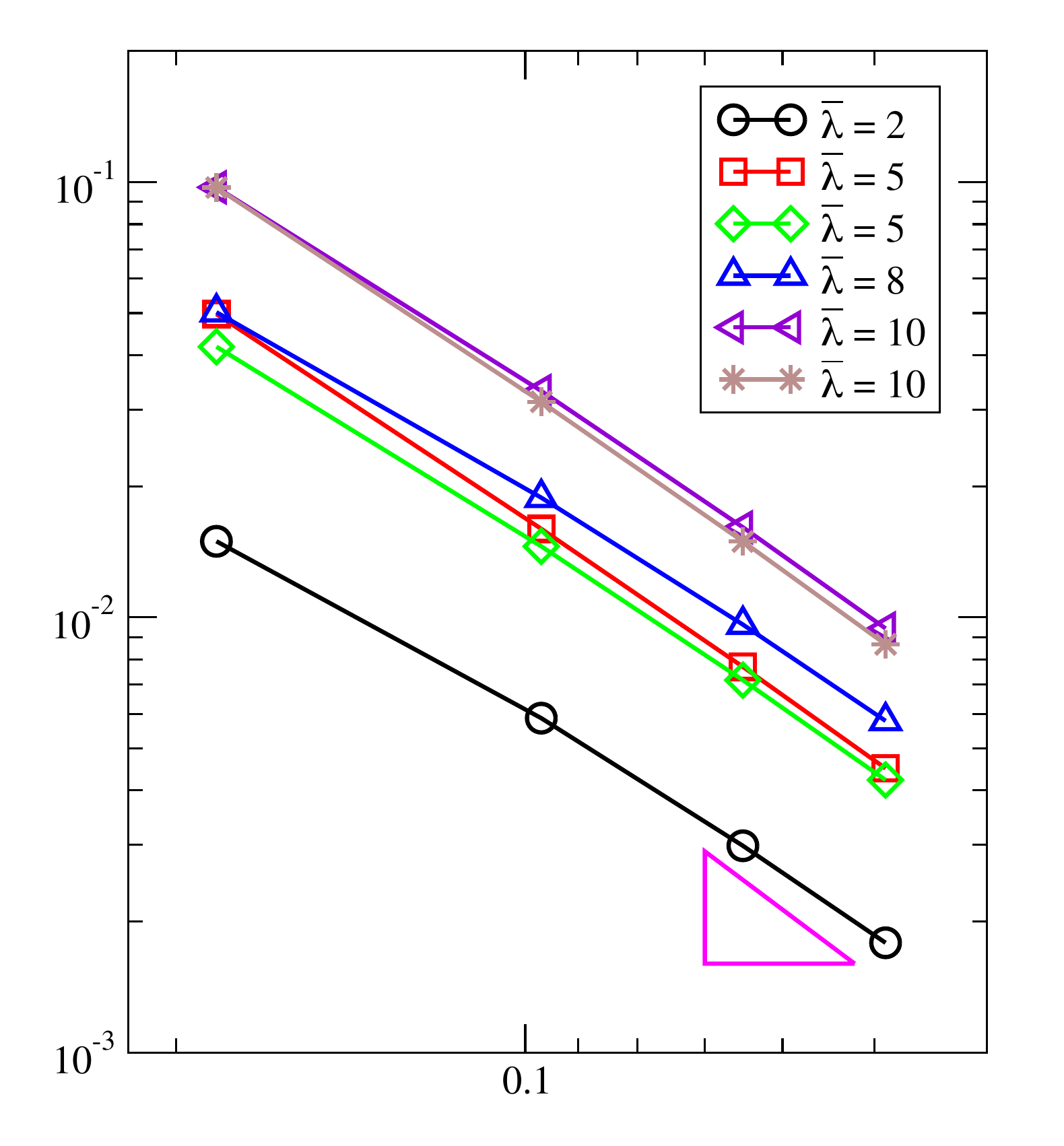}
      \put(-5,14){\begin{sideways}\textbf{Relative approximation error}\end{sideways}}
      \put(32,-5) {\textbf{Mesh size $\mathbf{\hh}$}}
      \put(56,18){\textbf{2}}
    \end{overpic} 
    &
    \begin{overpic}[scale=0.325]{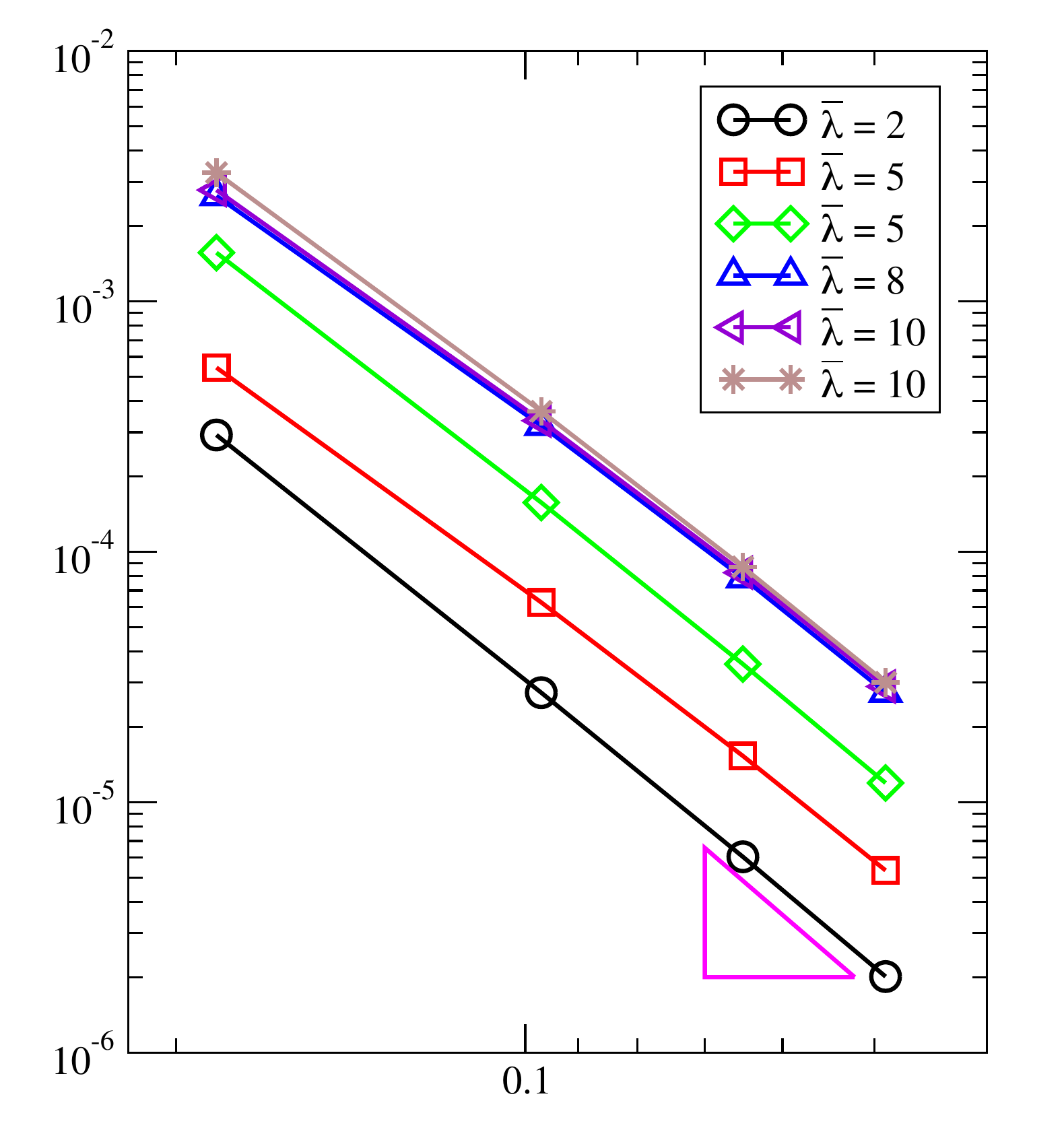}
      \put(32,-5) {\textbf{Mesh size $\mathbf{h}$}}
      \put(56,18){\textbf{4}}
    \end{overpic}
    &
    \begin{overpic}[scale=0.325]{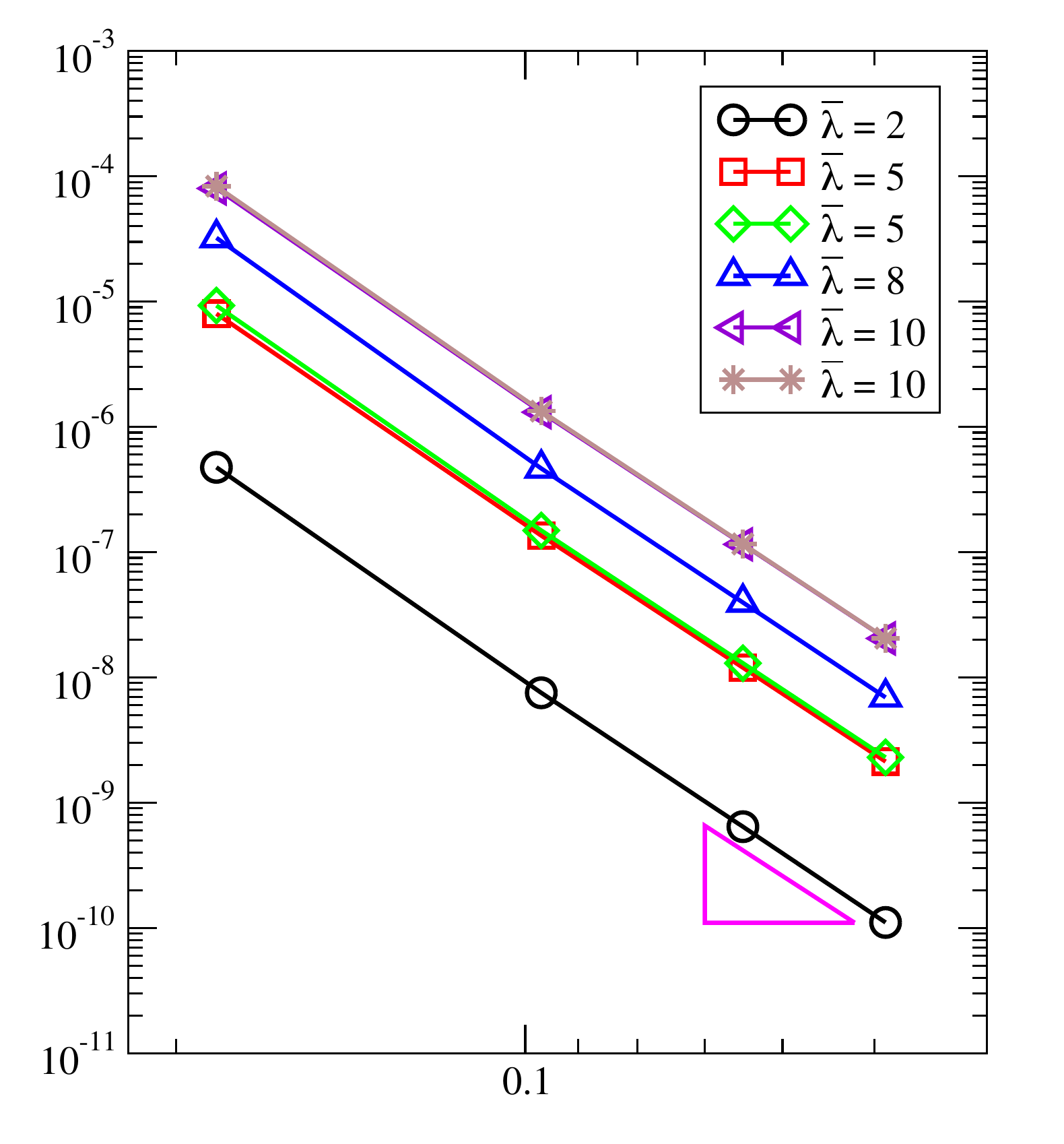}
      \put(32,-5) {\textbf{Mesh size $\mathbf{h}$}}
      \put(56,22){\textbf{6}}
    \end{overpic} 
  \end{tabular}
  \vspace{0.5cm}
  \caption{Test Case~1: Convergence plots for the approximation of the
	first six eigenvalues $\lambda=\pi^2\overline{\lambda}$ using the mainly hexagonal mesh and the
    nonconforming spaces: $\VS{\hh}_{1}$ (left panel); $\VS{\hh}_{2}$
    (mid panel); $\VS{\hh}_{3}$ (right panel). The generalized
    eigenvalue problem uses the nonstabilized bilinear form $\bsh(\cdot,\cdot)$. }
  \label{fig:hexa:rates}
\end{figure}
\begin{figure}
  \centering
  \begin{tabular}{ccc}
    \begin{overpic}[scale=0.325]{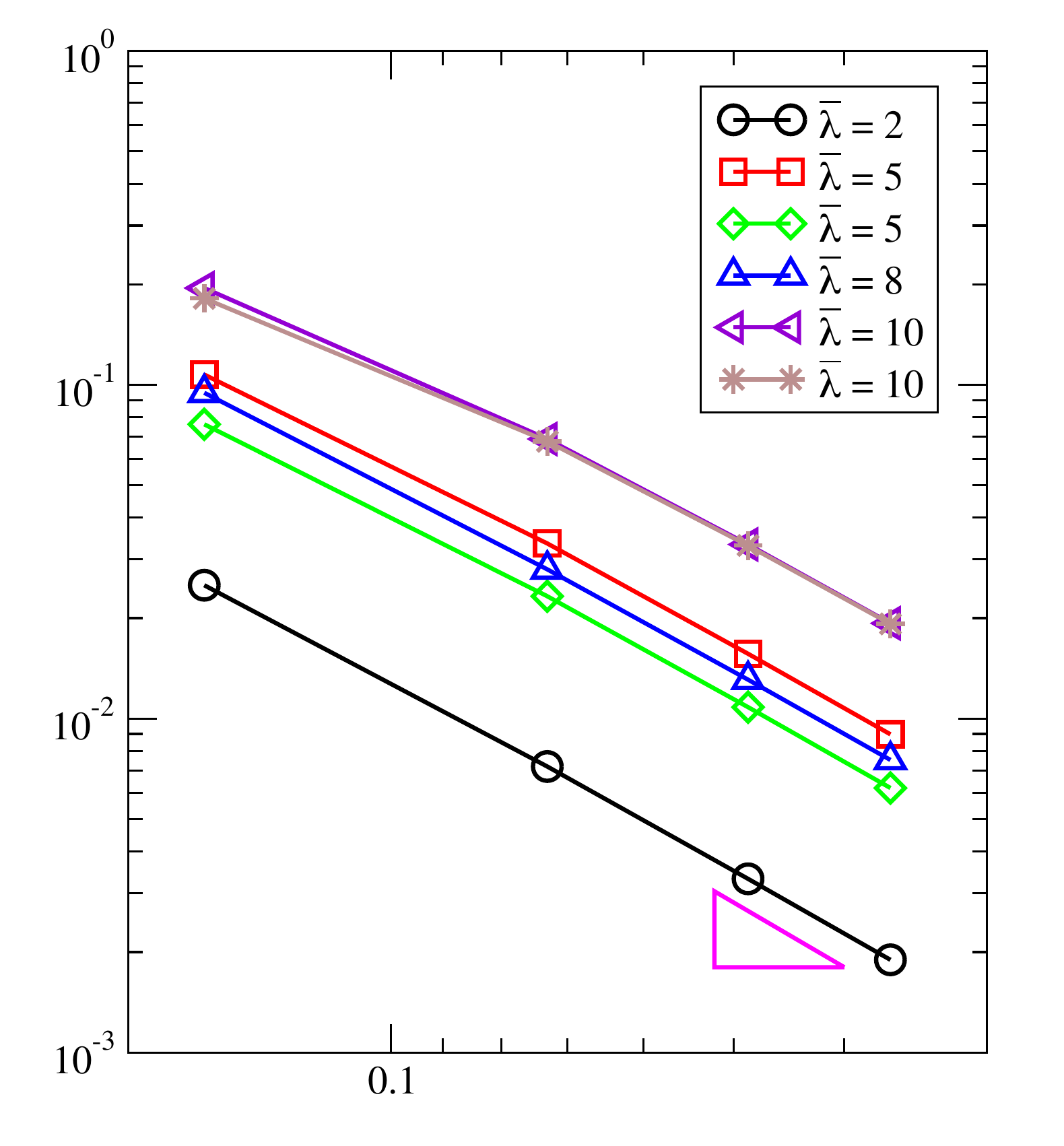}
      \put(-5,14){\begin{sideways}\textbf{Relative approximation error}\end{sideways}}
      \put(32,-5) {\textbf{Mesh size $\mathbf{h}$}}
      \put(57,18){\textbf{2}}
    \end{overpic} 
    &
    \begin{overpic}[scale=0.325]{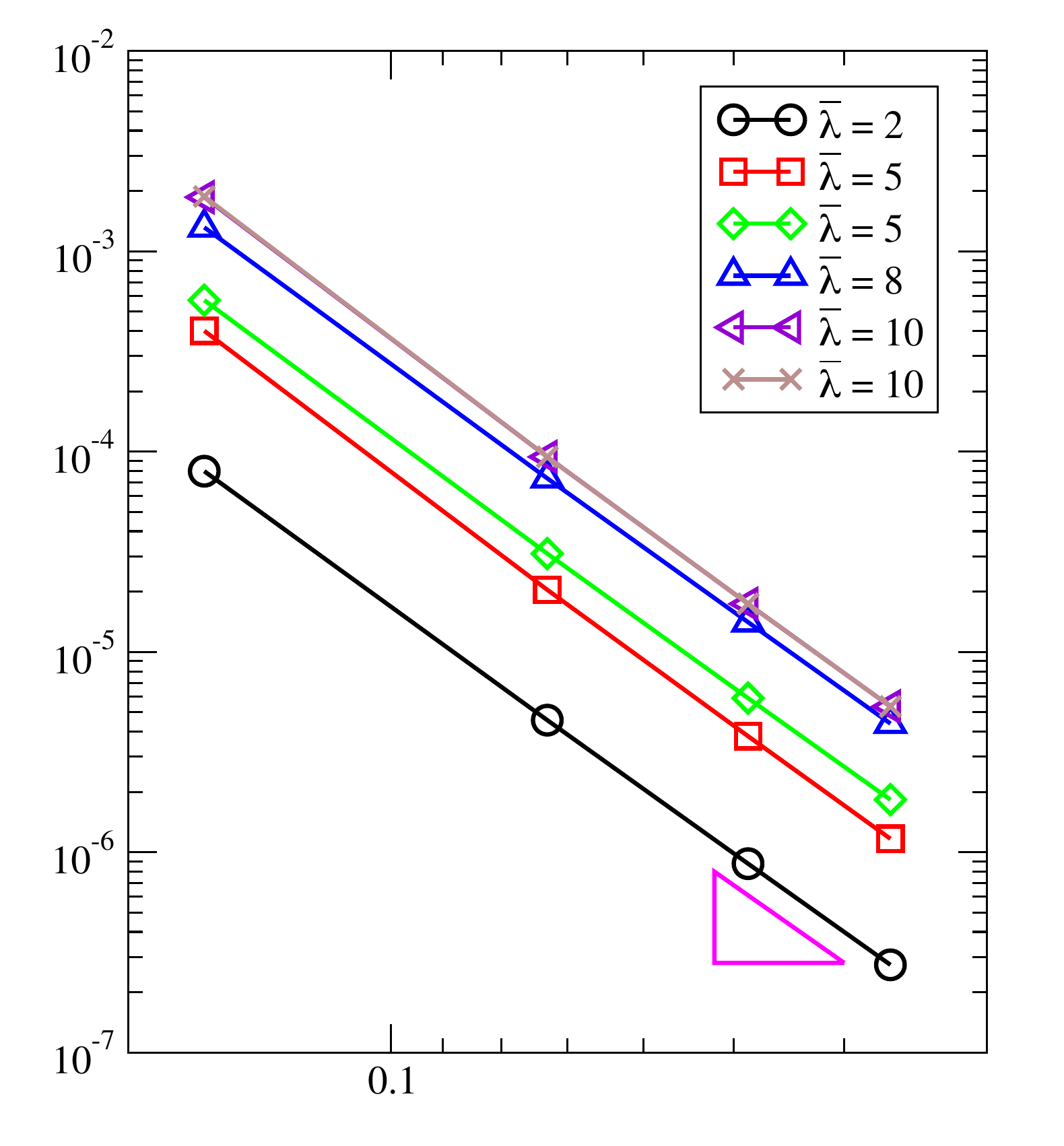}
      \put(32,-5) {\textbf{Mesh size $\mathbf{h}$}}
      \put(57,18){\textbf{4}}
    \end{overpic}
    &
    \begin{overpic}[scale=0.325]{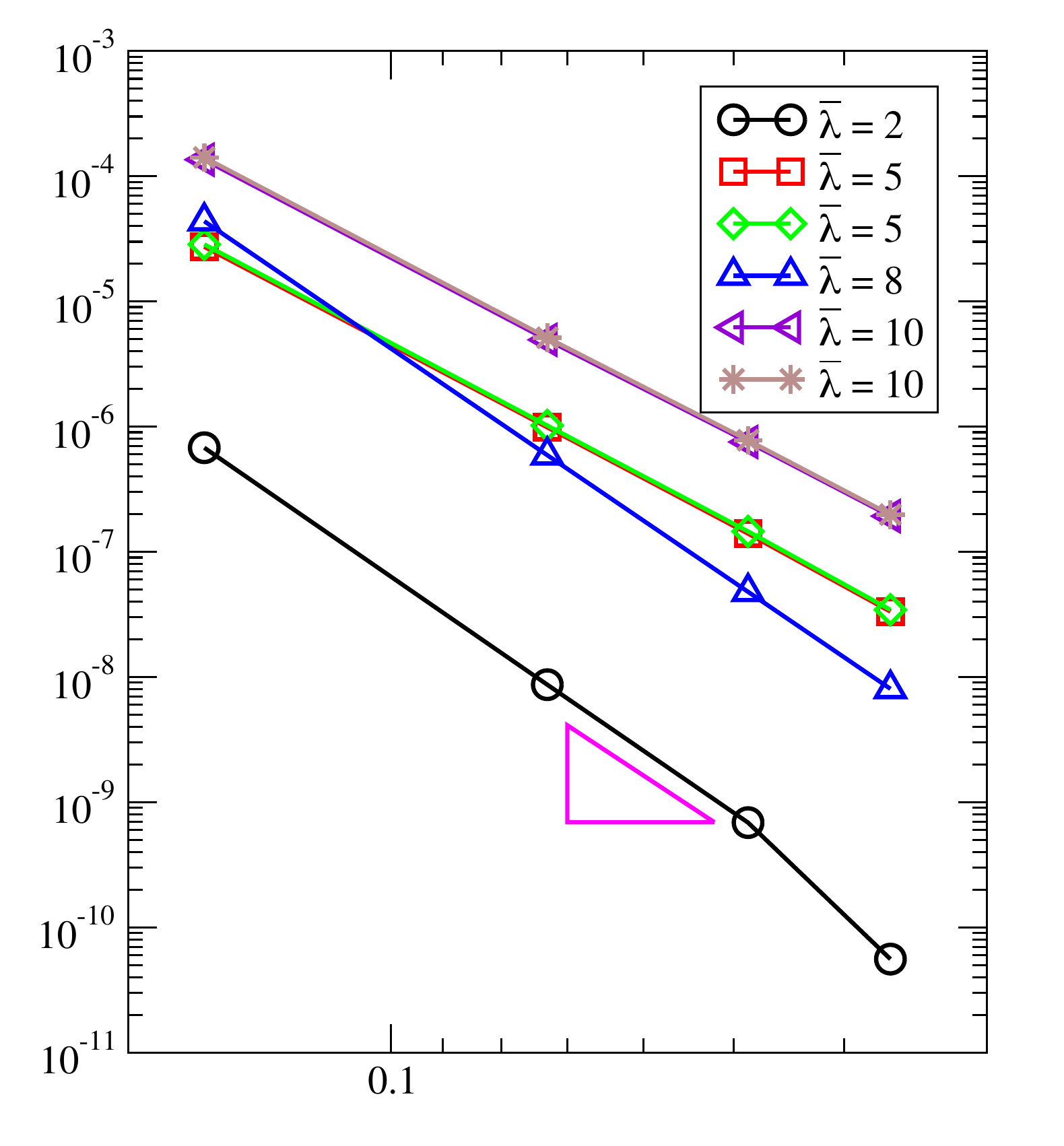}
      \put(32,-5) {\textbf{Mesh size $\mathbf{h}$}}
      \put(44,31){\textbf{6}}
    \end{overpic} 
  \end{tabular}
  \caption{Test Case~1: Convergence plots for the approximation of the
    first six eigenvalues $\lambda=\pi^2\overline{\lambda}$ using the nonconvex octagon mesh and the
    nonconforming spaces: $\VS{\hh}_{1}$ (left panel); $\VS{\hh}_{2}$
    (mid panel); $\VS{\hh}_{3}$ (right panel). The generalized
    eigenvalue problem uses the nonstabilized bilinear form $\bsh(\cdot,\cdot)$. }
  \label{fig:octa:rates}
\end{figure}
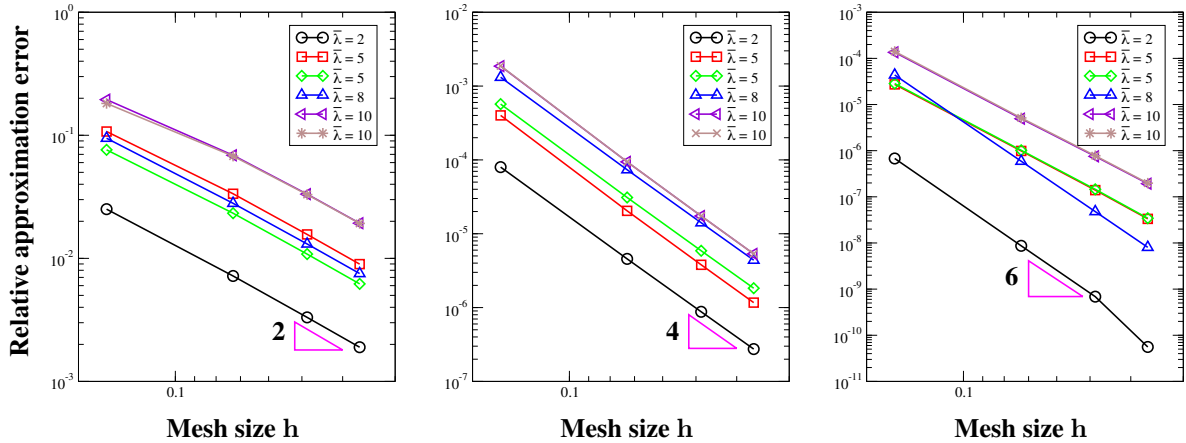

In this test case, we numerically solve the standard eigenvalue
problem with homogeneous Dirichlet boundary conditions on the square
domain $\Omega=]0,1[\times]0,1[$. 
In this case, the eigenvalues of the problem are known and are given by:
$$
\lambda = \pi^2 (n^2 + m^2)\quad n,m\in\mathbbm{N},\text{ with } n,m\neq 0.
$$
On this domain, we consider four different mesh partitionings, denoted by:
\begin{itemize}
\item \textit{Mesh~1}, mainly hexagonal mesh with continuously
  distorted cells;
\item \textit{Mesh~2}, nonconvex octagonal mesh;
\item \textit{Mesh~3}, randomized quadrilateral mesh;
\item \textit{Mesh~4}, central Voronoi tessellation.
\end{itemize}
The base mesh and the first refined mesh of each mesh sequence is
shown in Figure~\ref{fig:Meshes}.
These mesh sequences have been widely used in the mimetic finite
difference and virtual element literature, and a detailed description
of their construction can be found, for example,
in~\cite{BeiraodaVeiga-Lipnikov-Manzini:2011}.
We just mention that the last mesh sequence of central Voronoi
tessellation is generated by the code
PolyMesher~\cite{PolyMesher:2012}.
The convergence curves for the four mesh sequences above are reported
in Figures~\ref{fig:hexa:rates}, \ref{fig:octa:rates},
\ref{fig:quads:rates}, and \ref{fig:voro:rates}. 
\begin{figure}
  \centering
  \begin{tabular}{ccc}
    \begin{overpic}[scale=0.325]{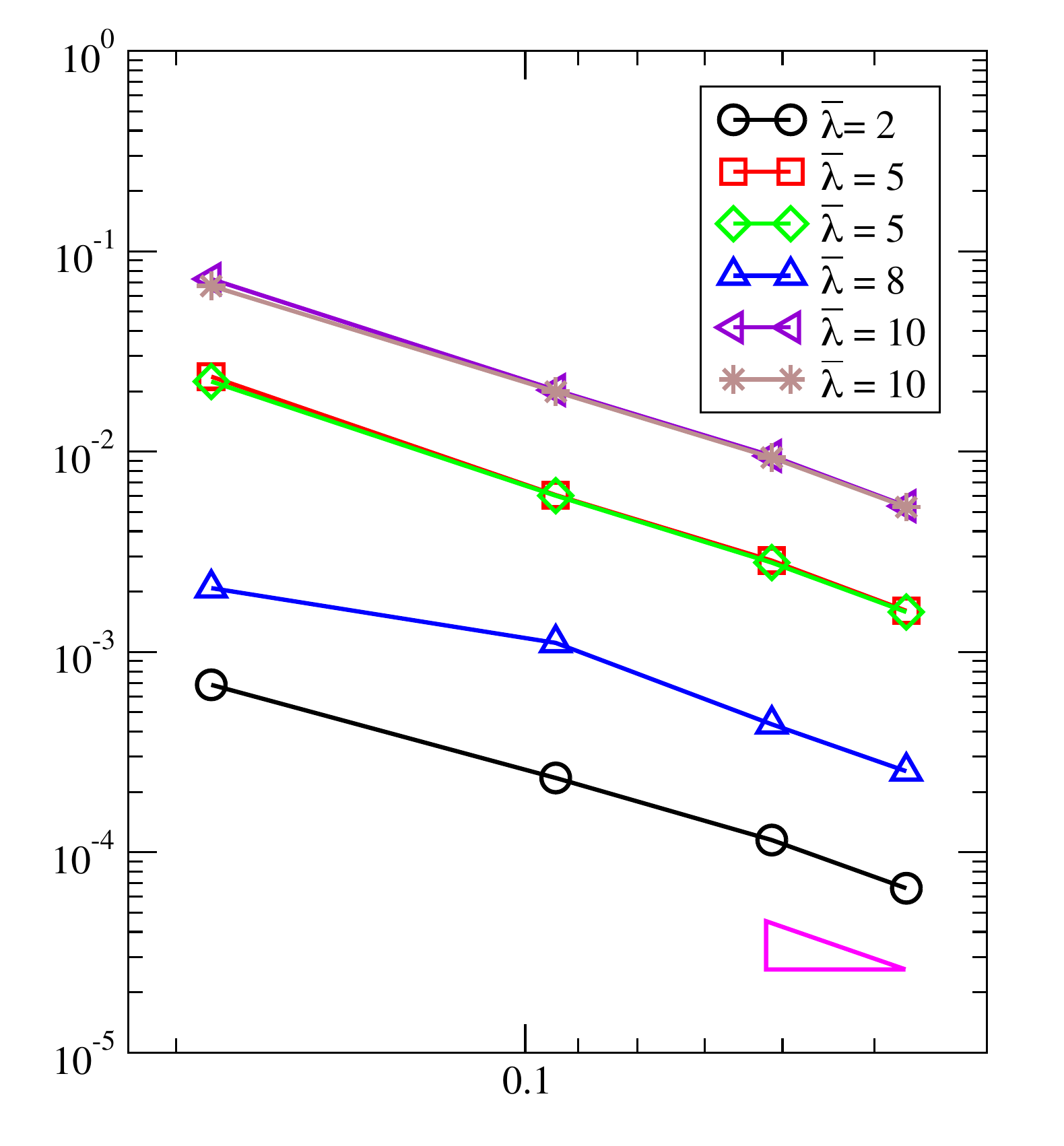}
      \put(-5,14){\begin{sideways}\textbf{Relative approximation error}\end{sideways}}
      \put(32,-5) {\textbf{Mesh size $\mathbf{h}$}}
      \put(62,18){\textbf{2}}
    \end{overpic} 
    &
    \begin{overpic}[scale=0.325]{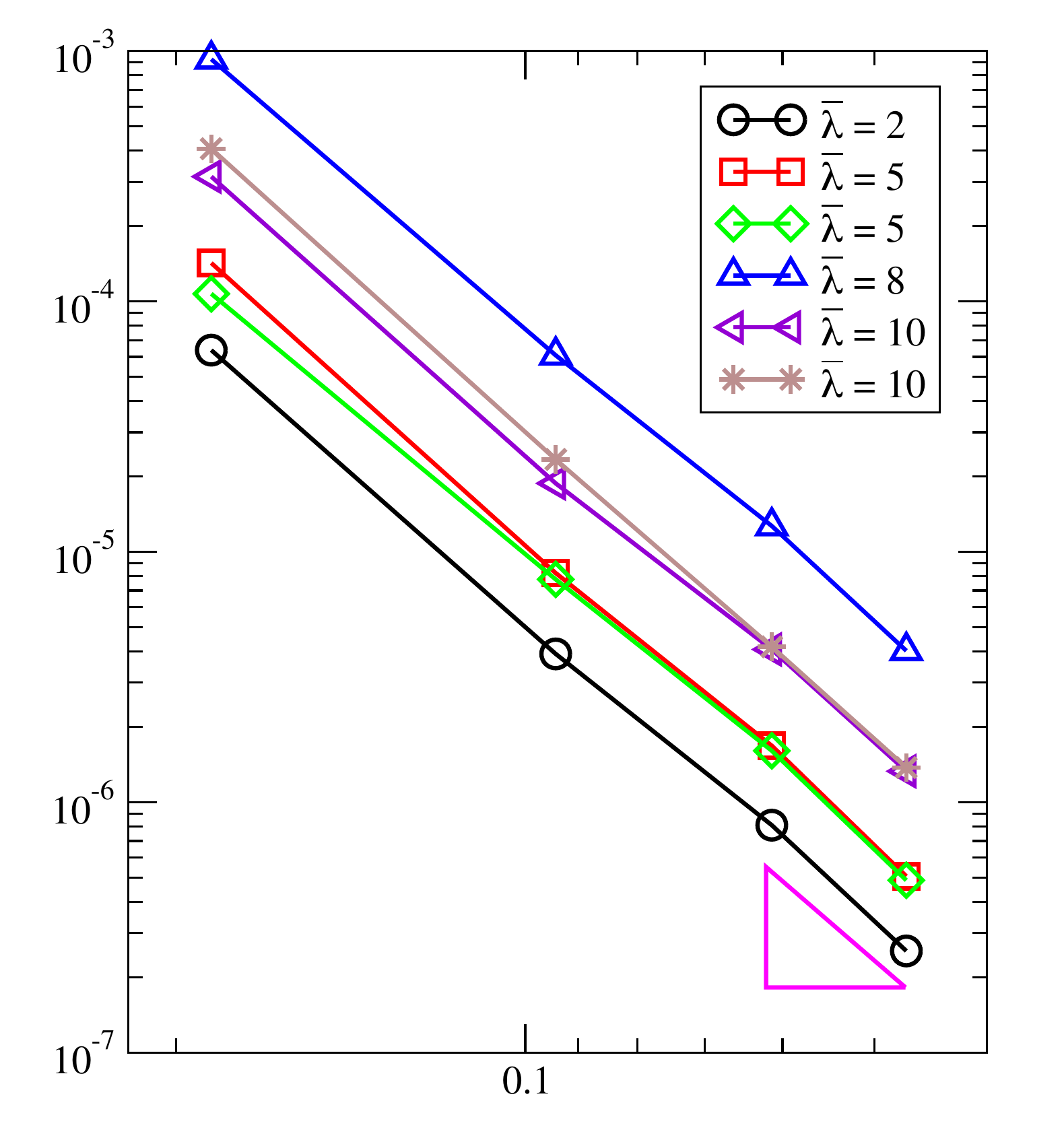}
      \put(32,-5) {\textbf{Mesh size $\mathbf{h}$}}
      \put(62,15){\textbf{4}}
    \end{overpic}
    &
    \begin{overpic}[scale=0.325]{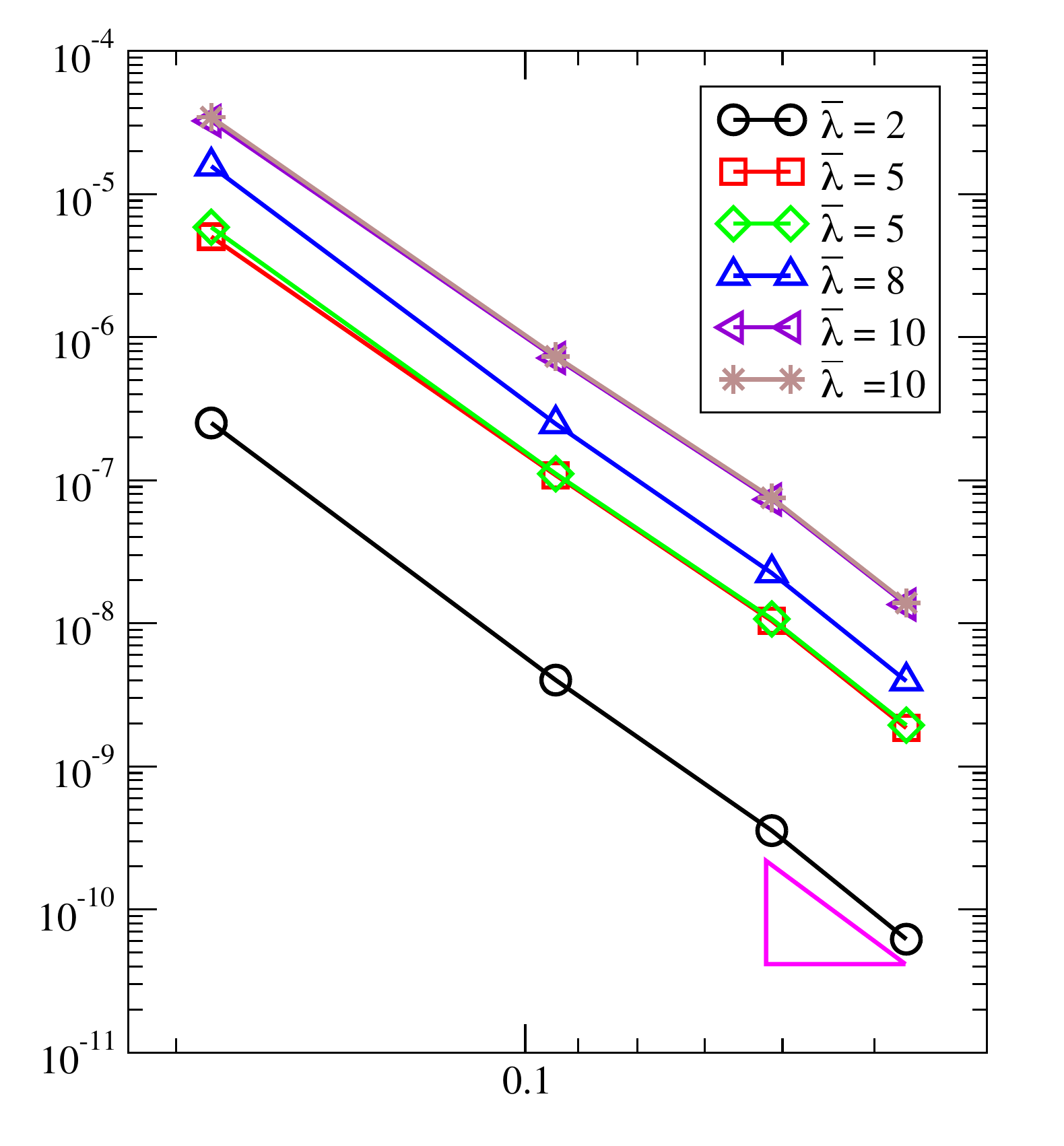}
      \put(32,-5) {\textbf{Mesh size $\mathbf{h}$}}
      \put(62,19){\textbf{6}}
    \end{overpic} 
  \end{tabular}
  \caption{Test Case~1: Convergence plots for the approximation of the
    first six eigenvalues $\lambda=\pi^2\overline{\lambda}$ using the randomized quadrilateral mesh and
    the nonconforming spaces: $\VS{\hh}_{1}$ (left panel);
    $\VS{\hh}_{2}$ (mid panel); $\VS{\hh}_{3}$ (right panel). The
    generalized eigenvalue problem uses the nonstabilized bilinear
    form $\bsh(\cdot,\cdot)$. }
  \label{fig:quads:rates}
\end{figure}

\begin{figure}
  \centering
  \begin{tabular}{ccc}
    \begin{overpic}[scale=0.35]{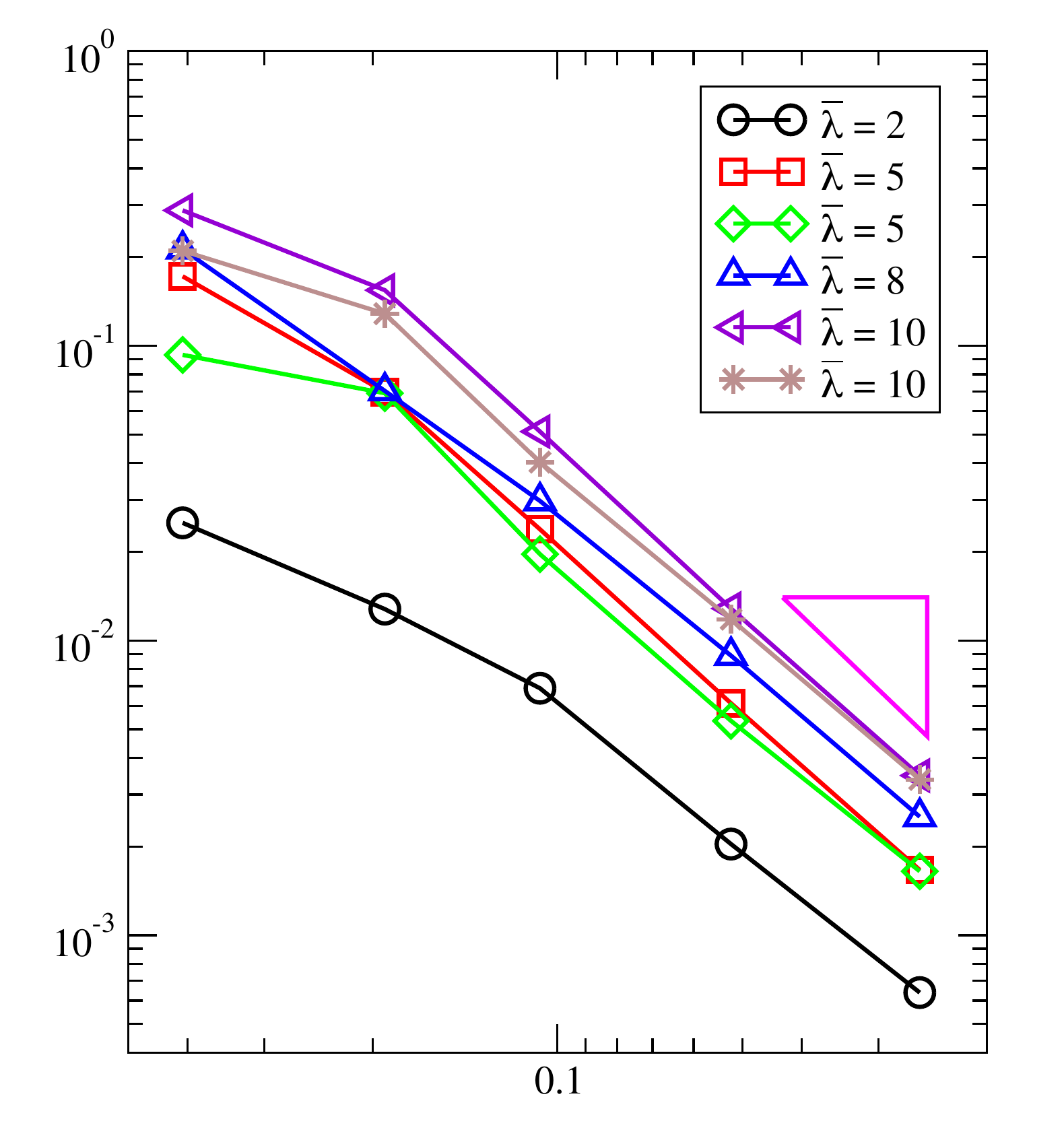}
      \put(-5,14){\begin{sideways}\textbf{Relative approximation error}\end{sideways}}
      \put(32,-5) {\textbf{Mesh size $\mathbf{h}$}}
      \put(72,49){\textbf{2}}
    \end{overpic} 
    &
    \begin{overpic}[scale=0.35]{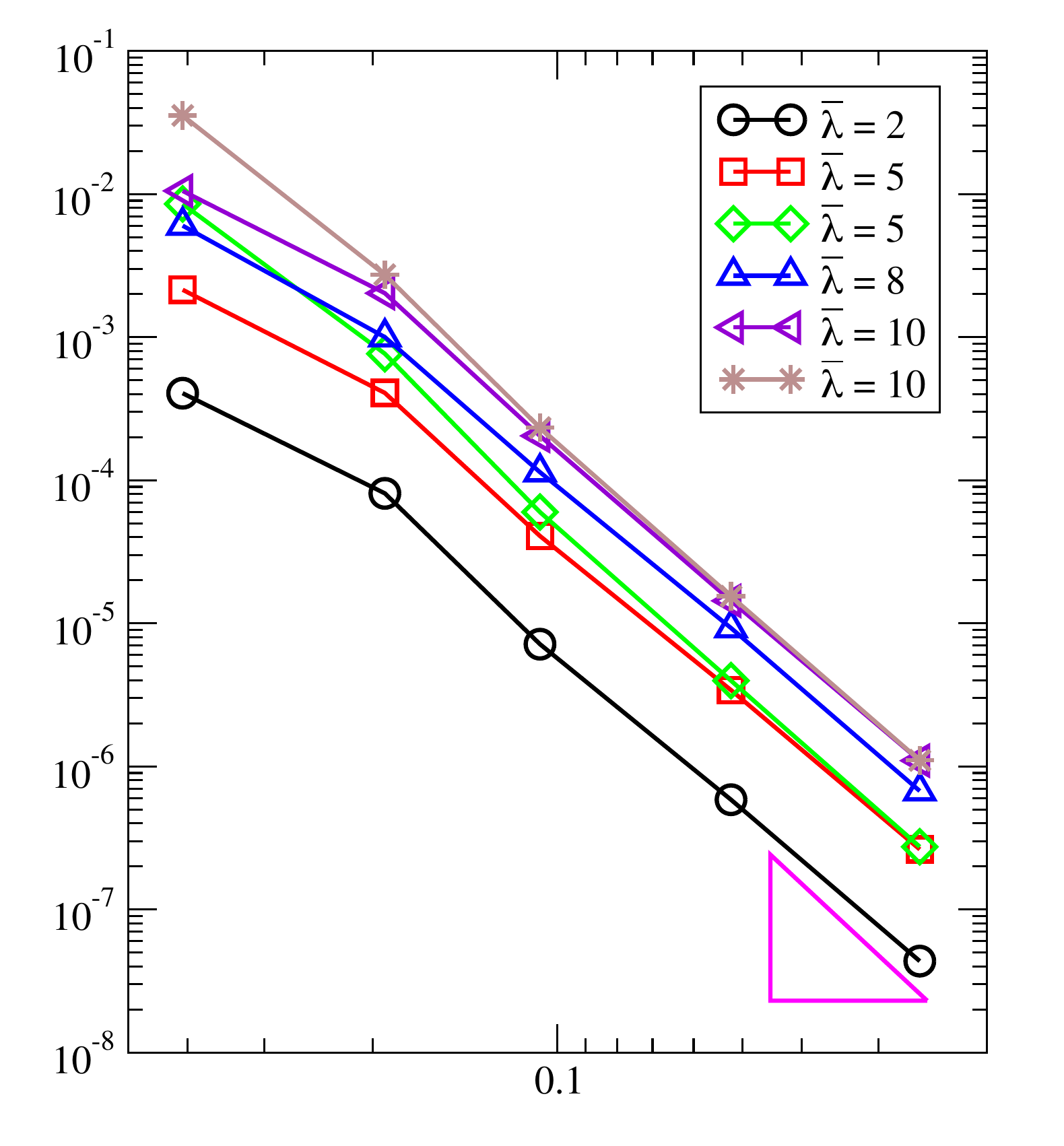}
      \put(32,-5) {\textbf{Mesh size $\mathbf{h}$}}
      \put(62,18){\textbf{4}}
    \end{overpic}
    &
    \begin{overpic}[scale=0.35]{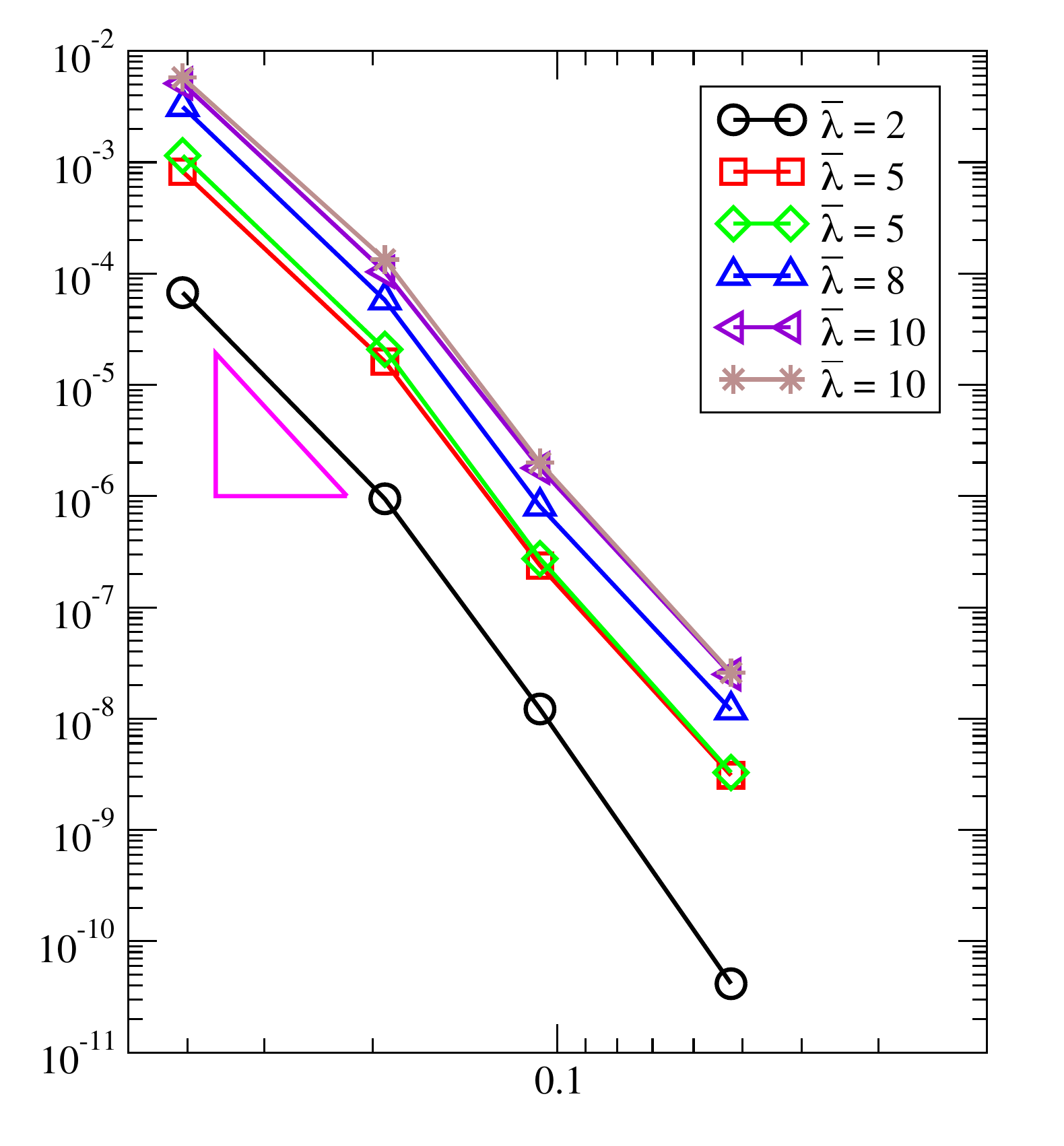}
      \put(32,-5) {\textbf{Mesh size $\mathbf{h}$}}
      \put(23,51){\textbf{6}}
    \end{overpic}
  \end{tabular}
  \caption{Test Case~1: Convergence plots for the approximation of the
    first six eigenvalues $\lambda=\pi^2\overline{\lambda}$ using the Voronoi mesh and the nonconforming
    spaces: $\VS{\hh}_{1}$ (left panel); $\VS{\hh}_{2}$ (mid panel);
    $\VS{\hh}_{3}$ (right panel). The generalized eigenvalue problem
    uses the nonstabilized bilinear form $\bsh(\cdot,\cdot)$. }
  \label{fig:voro:rates}
\end{figure}

The expected rate of convergence is shown in each panel by the
triangle closed to the error curve and indicated by an explicit label.
For these calculations, we used the VEM approximation based on the
nonconforming space $\Vhk$,
$k=1,2,3$, and the VEM formulation~\eqref{eq:discreteEigPbm} using the
nonstabilized bilinear form $\bsh(\cdot,\cdot)$.
As already observed in~\cite{Gardini-Vacca:2017} for the conforming VEM approximation, 
the same computations using formulation~\eqref{eq:discreteEigPbm2} and
the stabilized bilinear form $\bsht(\cdot,\cdot)$ produce almost identical results,
which, for this reason, are not shown here.
These plots confirm that the nonconforming VEM formulations proposed
in this work provide a numerical approximation with optimal
convergence rate on a set of representative mesh sequences, including
deformed and nonconvex cells.

\medskip
In the four plots of Figure~\ref{fig:stability-tau} we show the
dependence on the stabilization parameter $\tau$ of the value of the
first four eigenvalues $\lambda_1\slash{\pi^2}=2$,
$\lambda_2\slash{\pi^2}=5$, $\lambda_3\slash{\pi^2}=5$,
$\lambda_4\slash{\pi^2}=8$.
\begin{figure}
  \centering
  \begin{tabular}{cc}
    \begin{overpic}[scale=0.3]{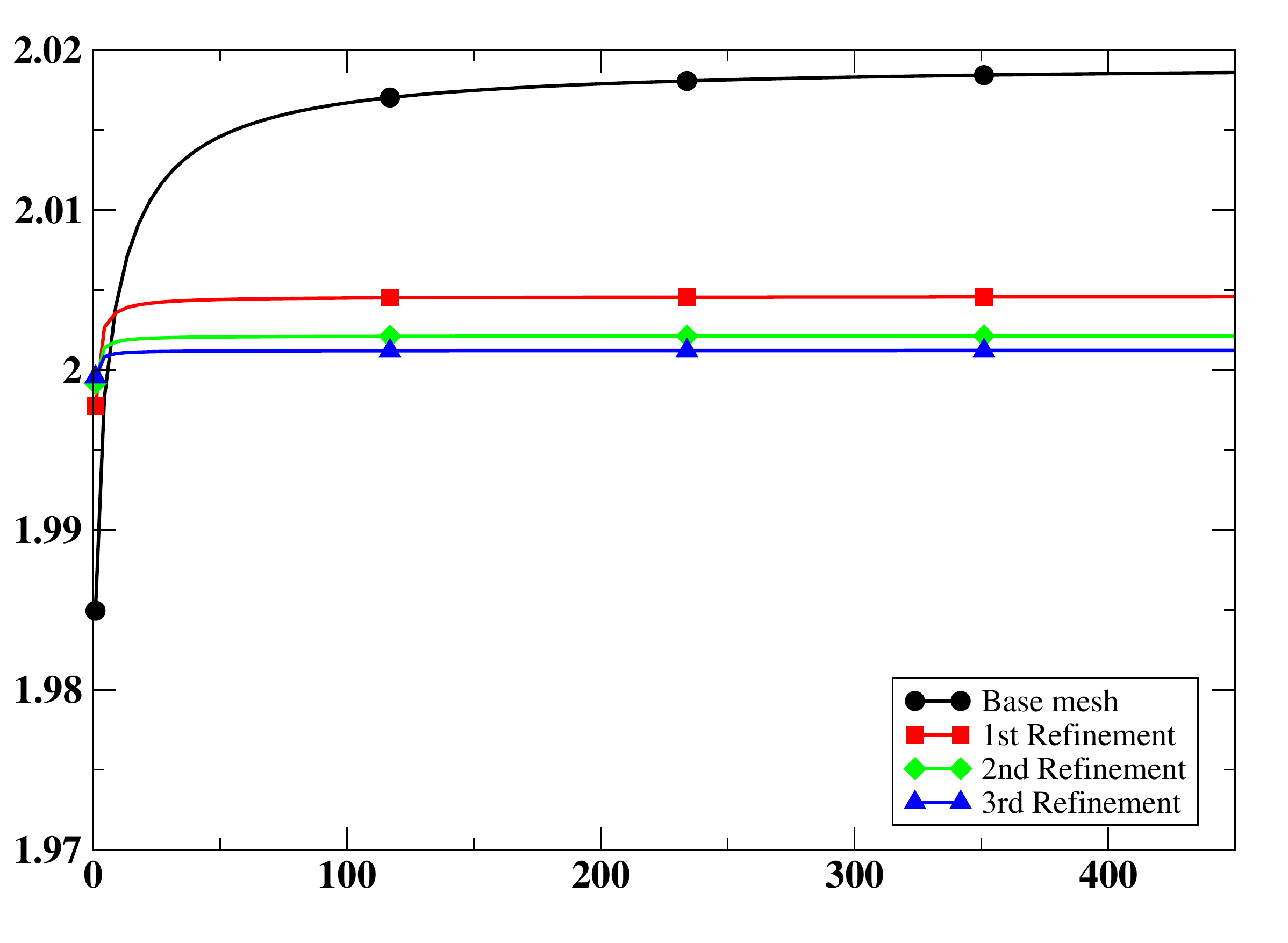}
    \put(32,-2){\textbf{Stability parameter $\tau$}}
    \end{overpic} 
    &\quad
    \begin{overpic}[scale=0.3]{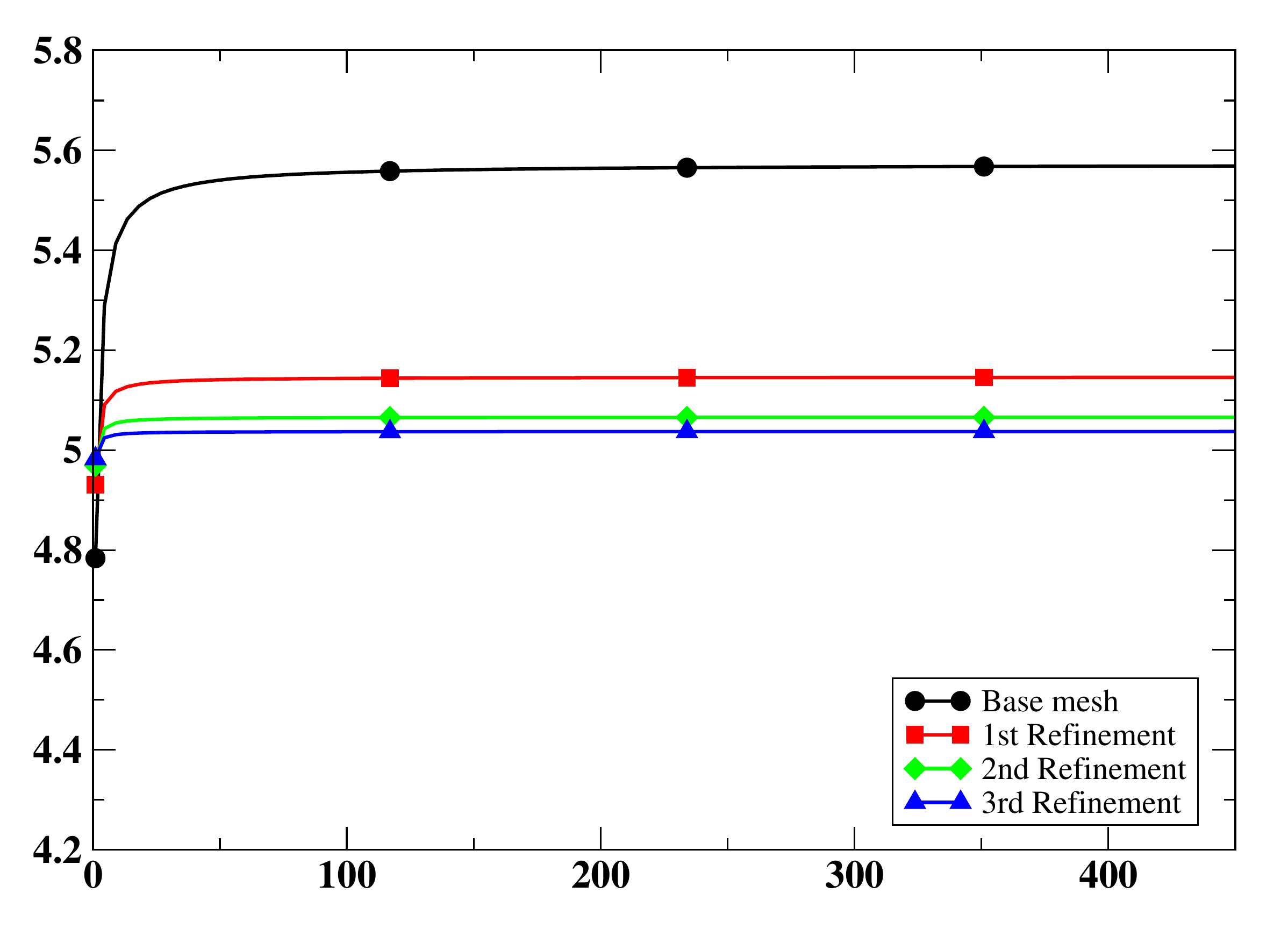}
    \put(32,-2){\textbf{Stability parameter $\tau$}}
    \end{overpic}
    \\[0.25em]
    {(a) $\lambda_1/\pi^2$} & \qquad{(b) $\lambda_2/\pi^2$}
    \\[1em]
    \begin{overpic}[scale=0.3]{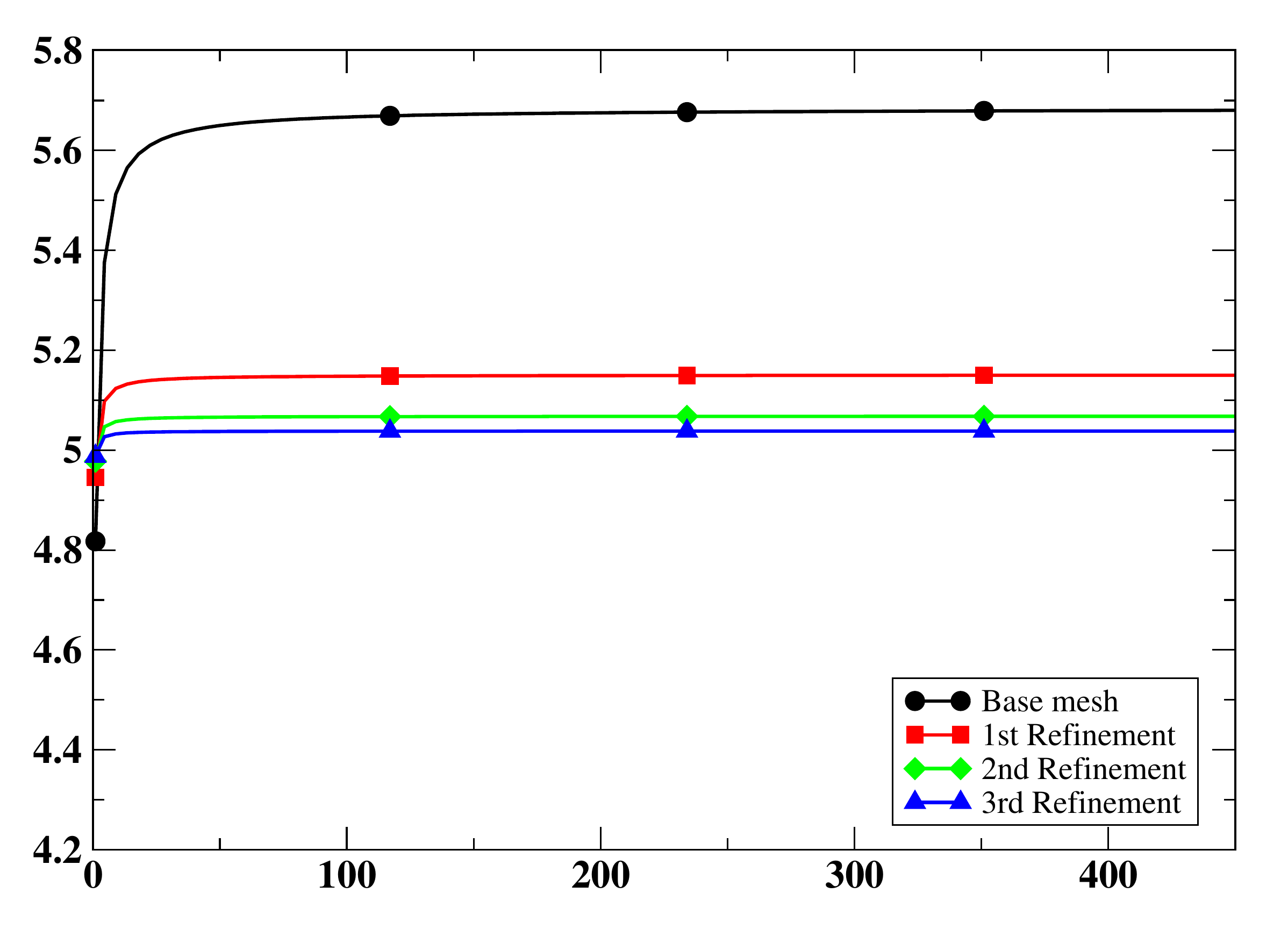}
    \put(32,-2){\textbf{Stability parameter $\tau$}}
    \end{overpic} 
    &\qquad
    \begin{overpic}[scale=0.3]{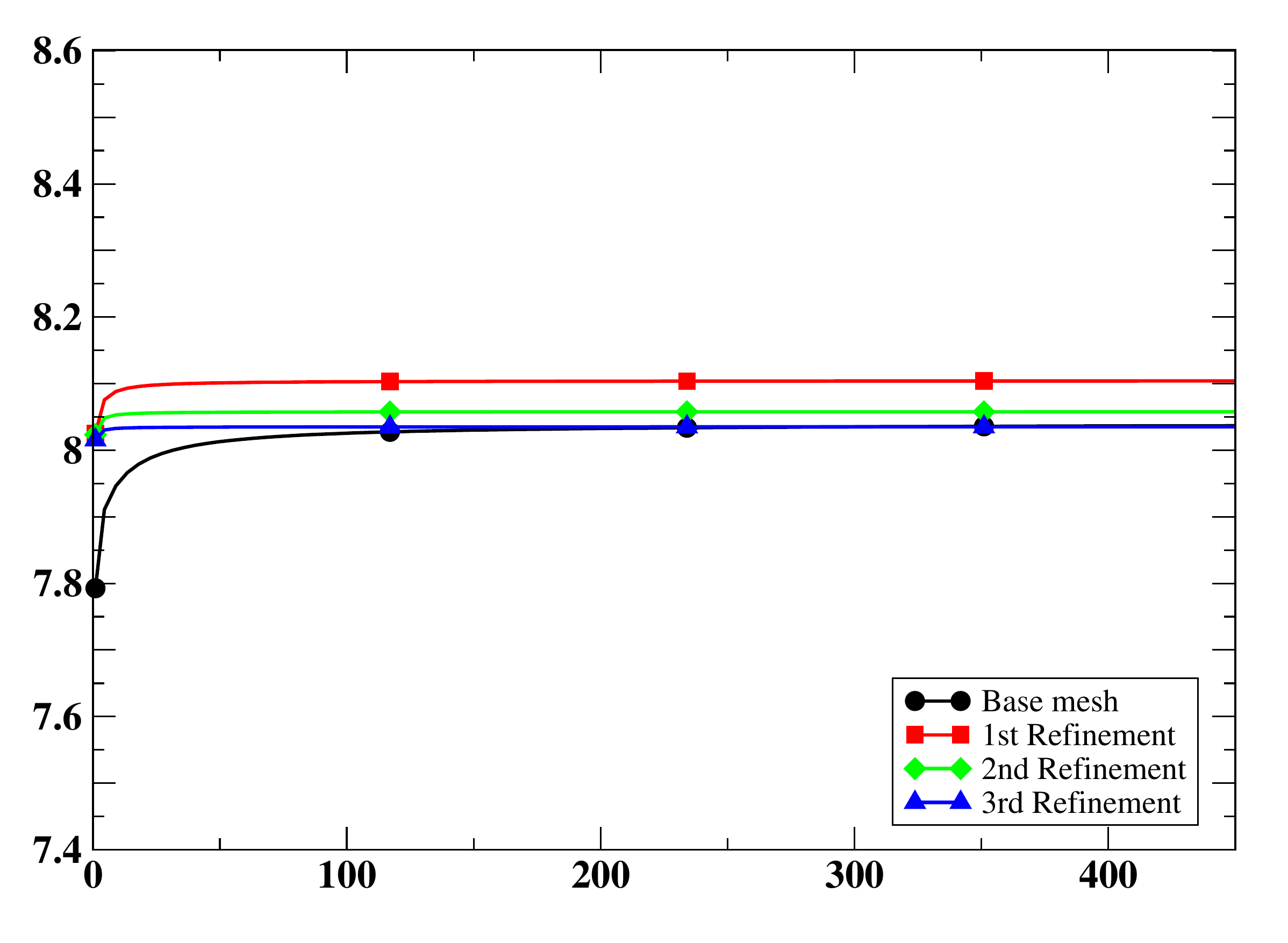}
    \put(32,-2){\textbf{Stability parameter $\tau$}}
    \end{overpic}
    \\[0.25em]
    {(c) $\lambda_3/\pi^2$} & \qquad{(d) $\lambda_4/\pi^2$}
  \end{tabular}
  \caption{Test Case~1: Eigenvalue curves versus the stability
    parameter $\tau$ using the virtual element space $\VS{\hh}_{1}$ on
    the first four meshes of the mainly hexagonal mesh family.}
  \label{fig:stability-tau}
\end{figure}
From these plots, it is clear that the eigenvalue approximation is
stable in a reasonable range of values of the parameter $\tau_{\P}$,
and that the curves of the numerical eigenvalue converge to the
corresponding eigenvalue.

\medskip
Finally, in Figure~\ref{fig:comparison} we compare the approximation
of the first eigenvalue using the conforming and nonconforming VEM
on the four mesh sequences \textit{Mesh~1}-\textit{Mesh~4}.
\begin{figure}
   \centering
   \begin{tabular}{cc}
     \begin{overpic}[scale=0.35]{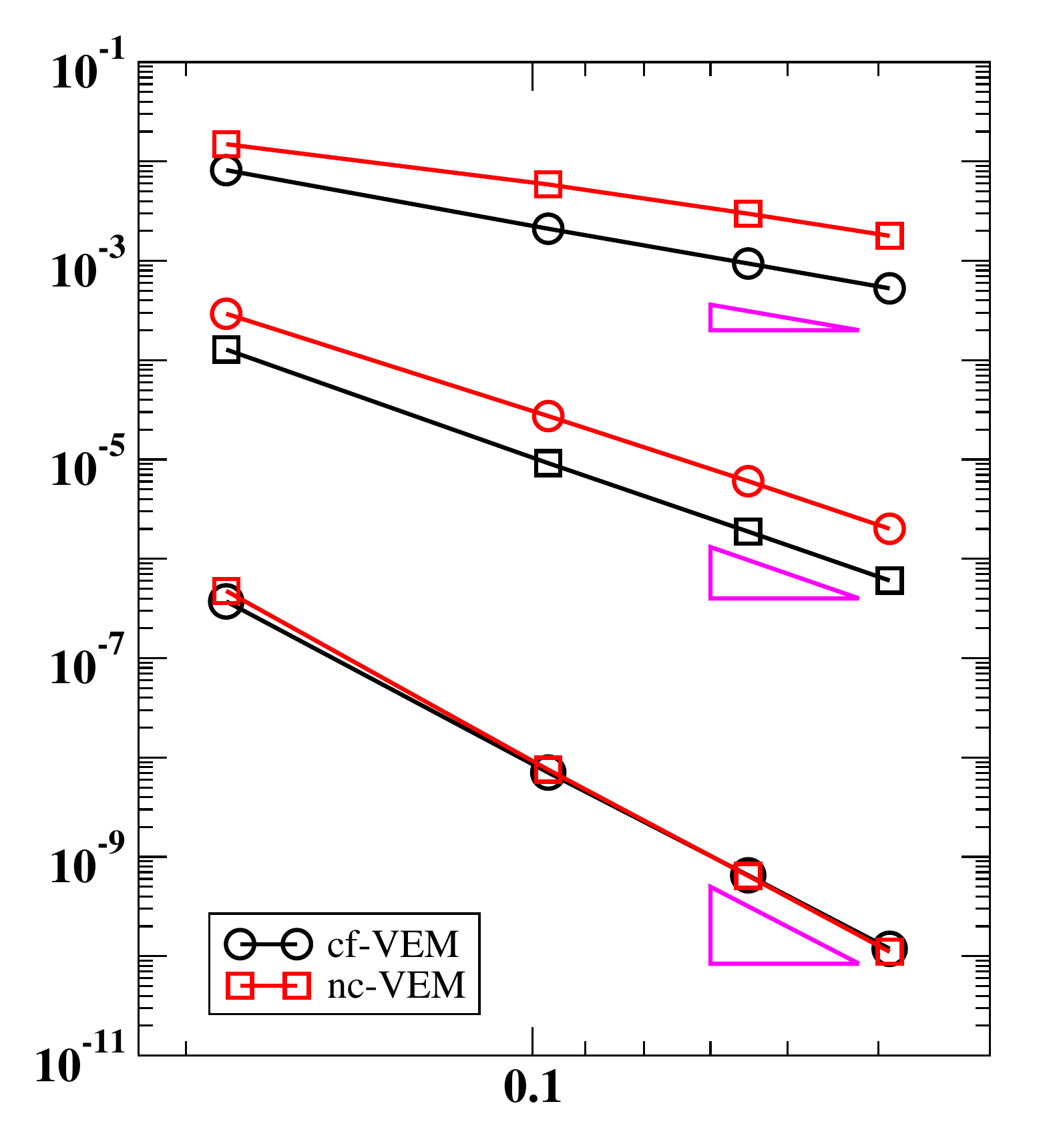}
       \put(-5,15){\begin{sideways}\textbf{Relative approximation error}\end{sideways}}
       \put(32,-5) {\textbf{Mesh size $\mathbf{\hh}$}}
       \put(56,70.5){\textbf{2}}
       \put(56,48.5){\textbf{4}}
       \put(56,18){\textbf{6}}
     \end{overpic} 
     &\qquad
     \begin{overpic}[scale=0.35]{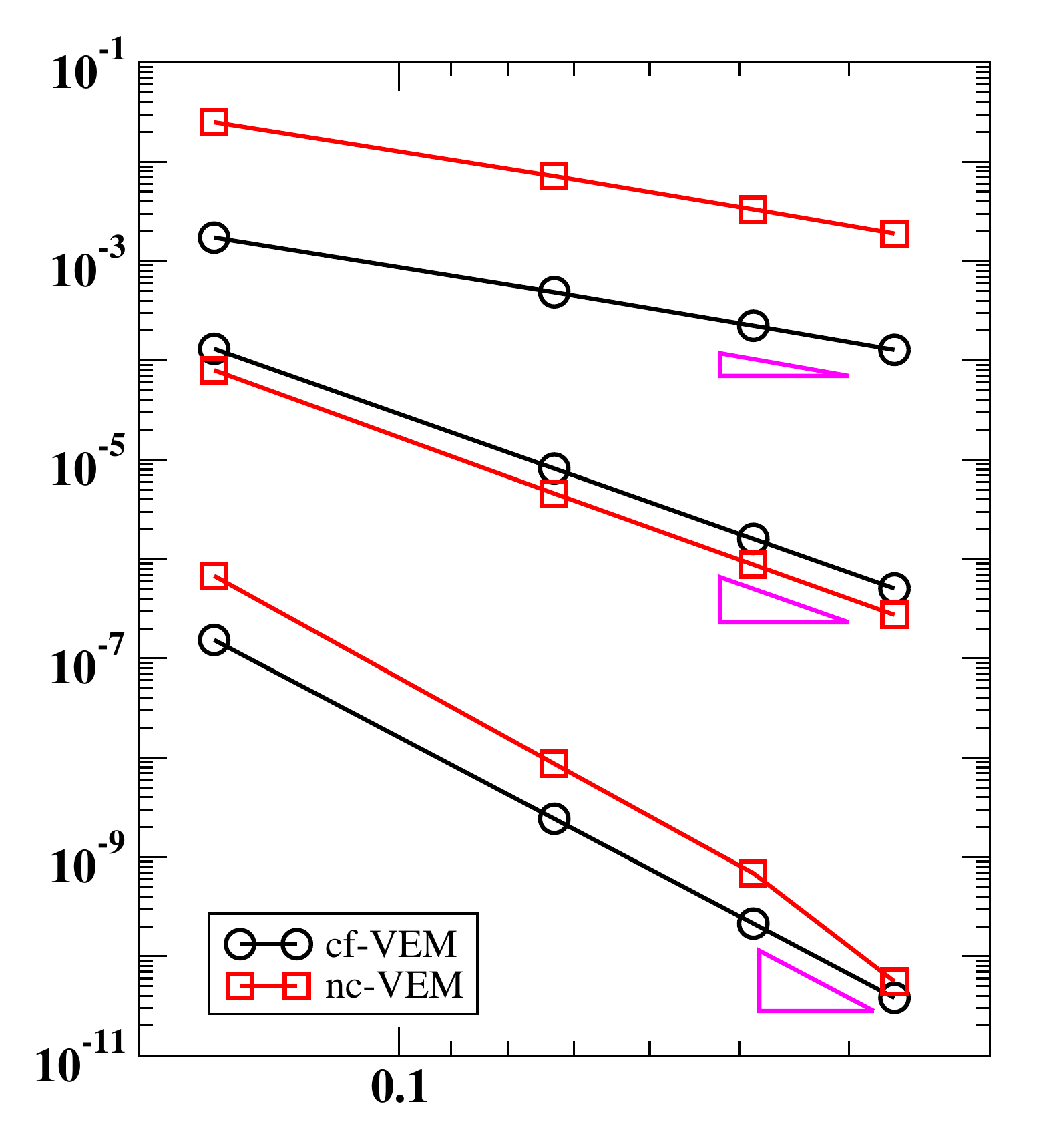}
       \put(-5,15){\begin{sideways}\textbf{Relative approximation error}\end{sideways}}
       \put(32,-5) {\textbf{Mesh size $\mathbf{h}$}}
       \put(57,66){\textbf{2}}
       \put(57,46){\textbf{4}}
       \put(61,13){\textbf{6}}
     \end{overpic}
     \\[1em]
     \begin{overpic}[scale=0.35]{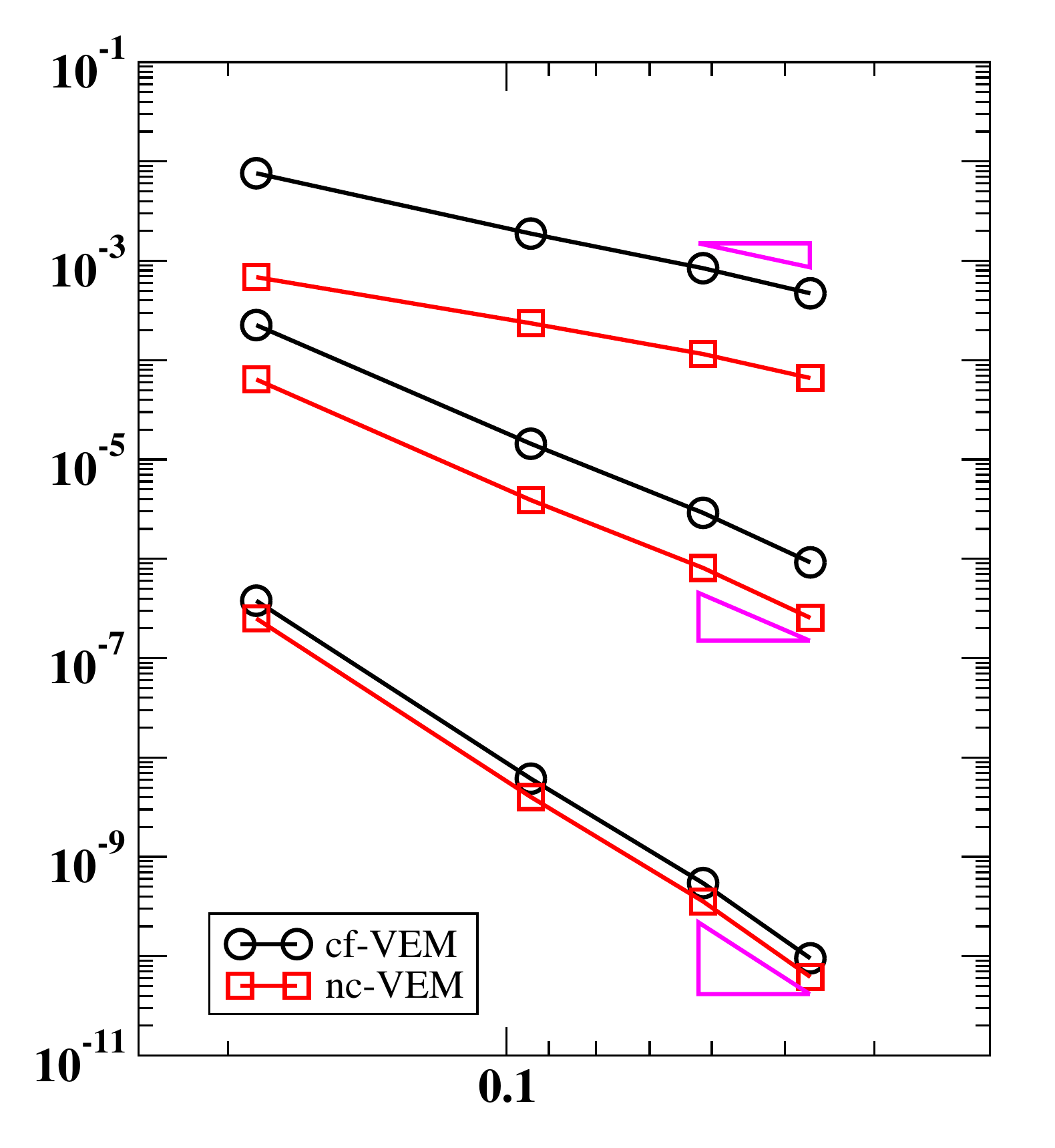}
       \put(-5,15){\begin{sideways}\textbf{Relative approximation error}\end{sideways}}
       \put(32,-5) {\textbf{Mesh size $\mathbf{\hh}$}}
       \put(64,81){\textbf{2}}
       \put(55,45){\textbf{4}}
       \put(55,15){\textbf{6}}
     \end{overpic}
     &\qquad
     \begin{overpic}[scale=0.35]{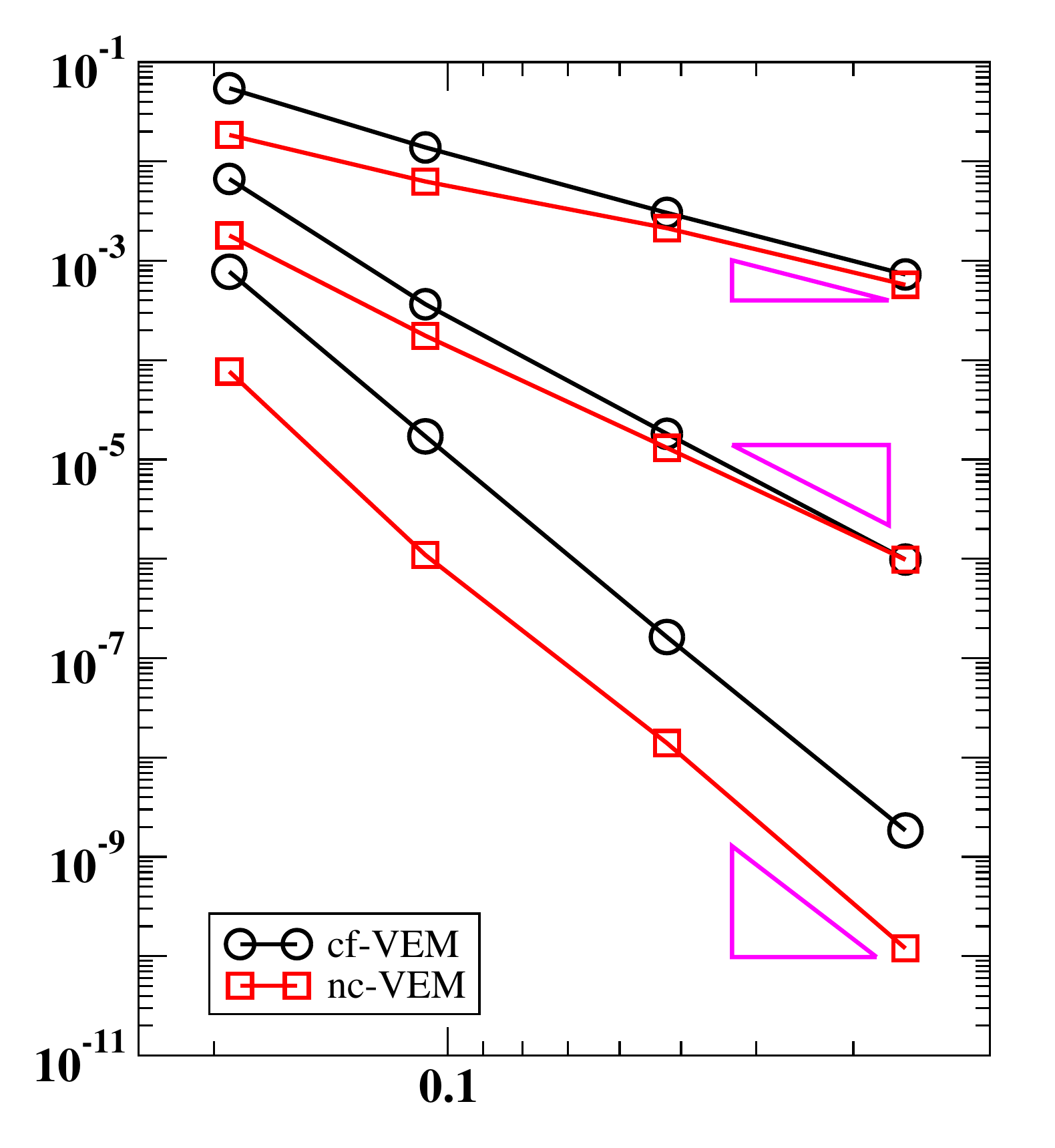}
       \put(-5,15){\begin{sideways}\textbf{Relative approximation error}\end{sideways}}
       \put(32,-5) {\textbf{Mesh size $\mathbf{h}$}}
       \put(58,74){\textbf{2}}
       \put(69,63){\textbf{4}}
       \put(59,20){\textbf{6}}
     \end{overpic}
   \end{tabular}
   \caption{Test Case~1: Comparison between the approximation of the
     first eigenvalue using the conforming and nonconforming VEM
     spaces $\VS{\hh}_{k}$, $k=1,2,3$, and the mainly hexagonal mesh
     (top-left panel); the nonconvex octagonal mesh (top-right panel);
     the randomized quadrilateral mesh (bottom-left panel); the
     Voronoi mesh (bottom-right panel).}
   \label{fig:comparison}
 \end{figure}
For all these meshes, we see that the two approximations are very
closed.

\subsection{Test~2.}

\begin{figure}
  \centering
  \begin{tabular}{cccc}
    \begin{overpic}[scale=0.2]{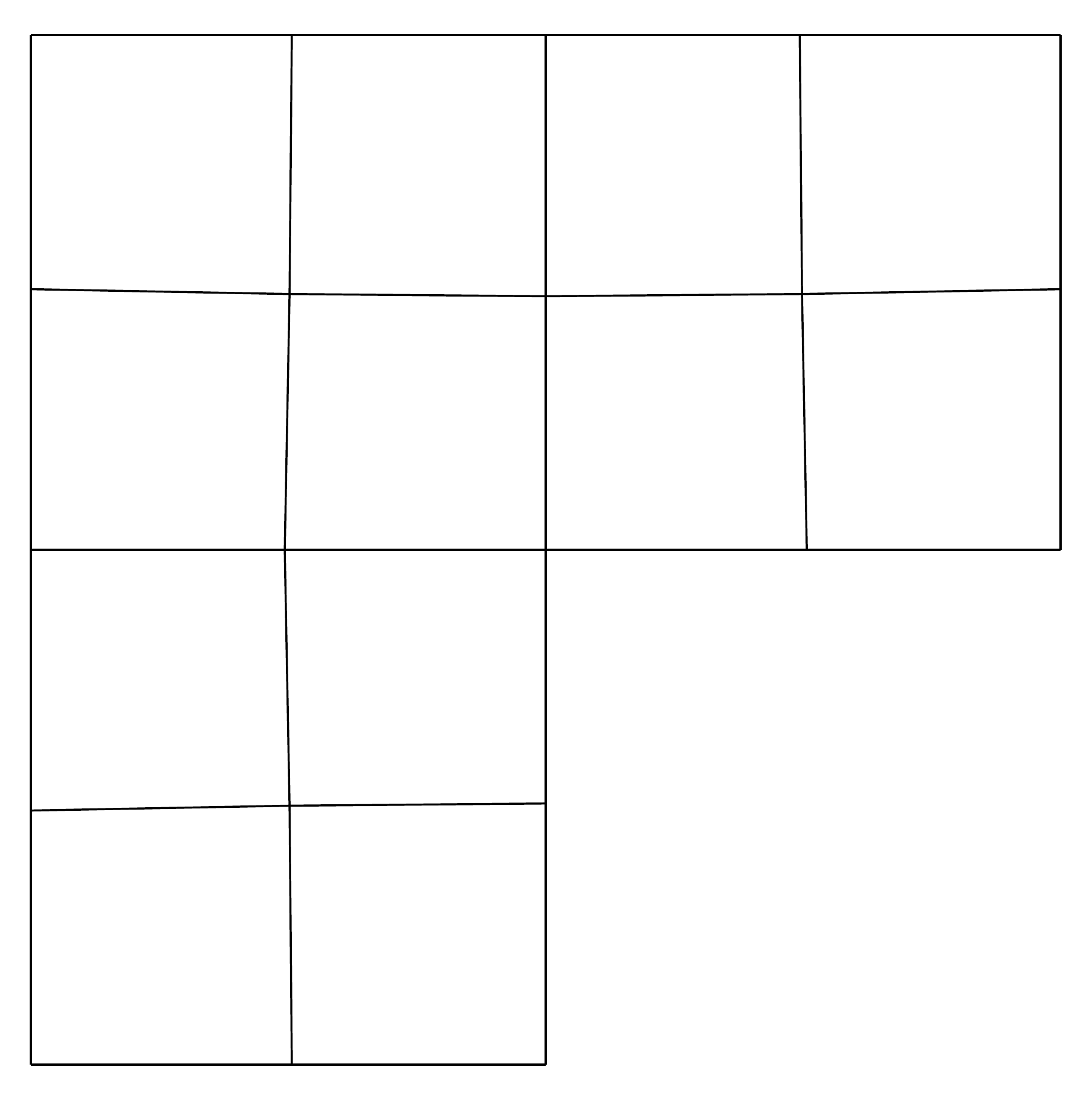}
    \end{overpic} 
    &
    \begin{overpic}[scale=0.2]{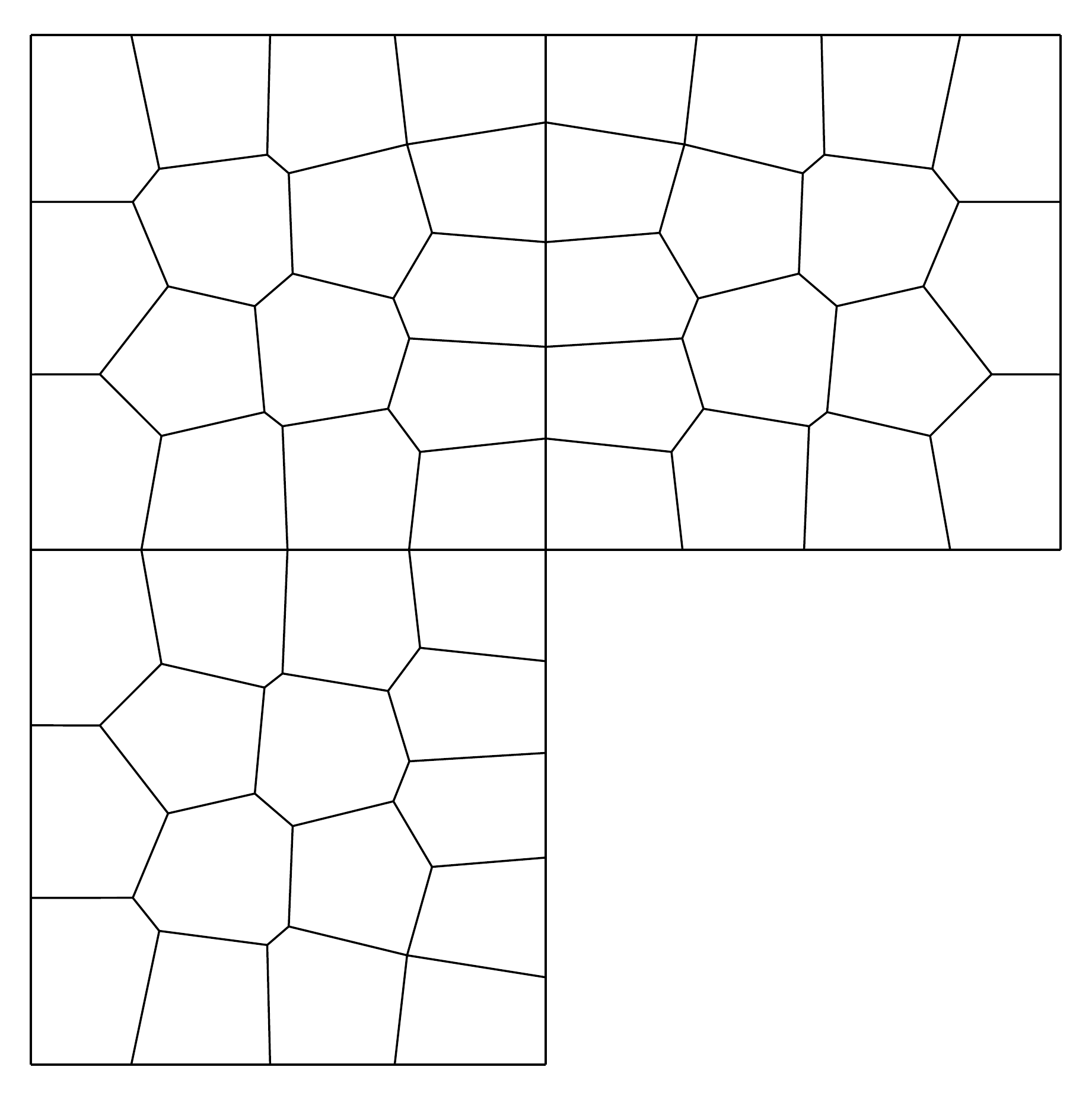}
    \end{overpic}
    &
    \begin{overpic}[scale=0.2]{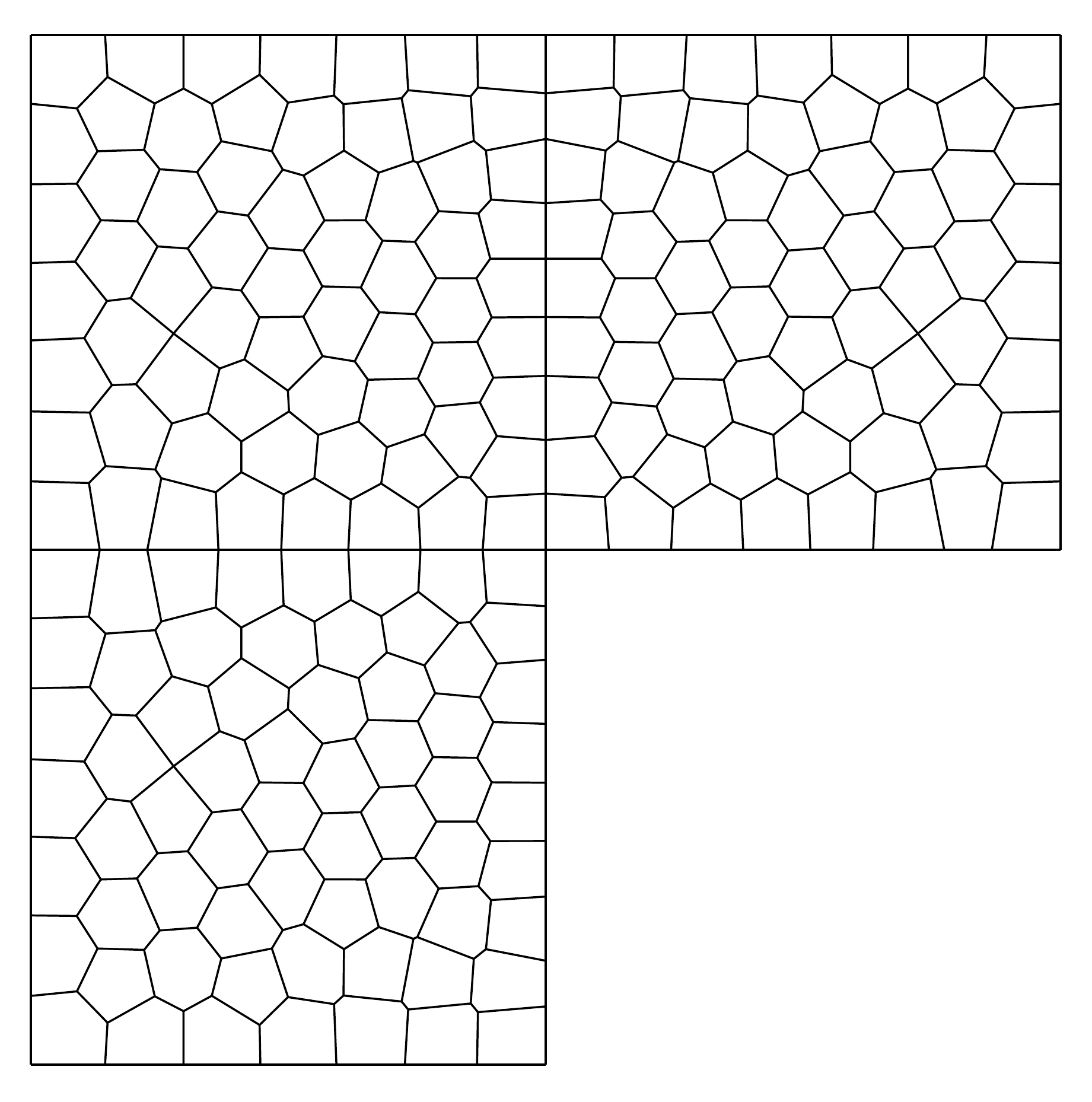}
    \end{overpic} 
    &
    \begin{overpic}[scale=0.2]{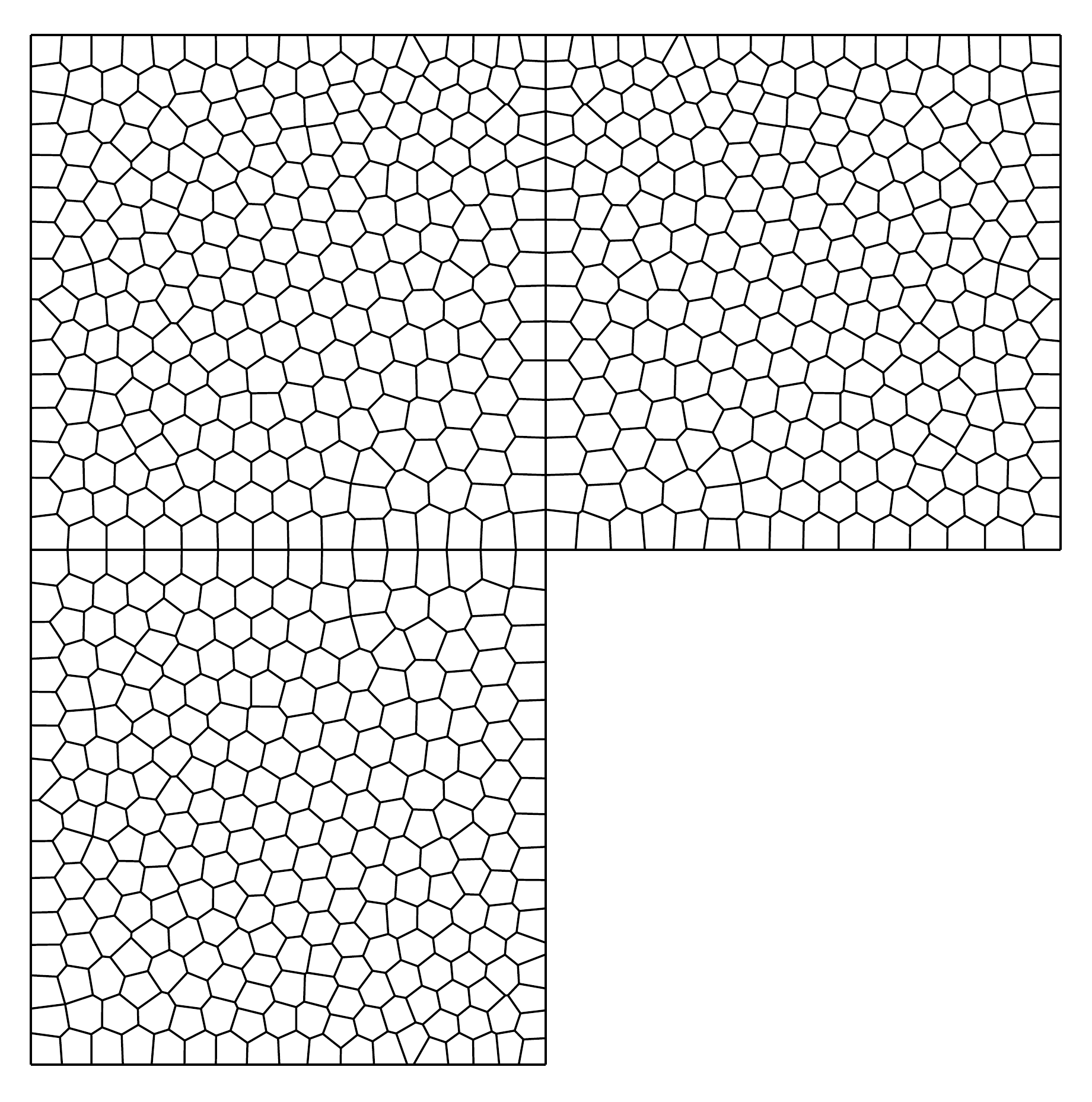}
    \end{overpic}
  \end{tabular}
  \caption{Test Case~2: Mesh sequence used in the L-shaped domain
    test.}
  \label{fig:LShape-Meshes}
\end{figure}
In this test case, we solve the eigenvalue problem with Neumann
boundary conditions on the nonconvex L-shaped domain
$\Omega=\Omega_1\backslash\Omega_0$, where
$\Omega_1=]-1,1[\times]-1,1[$, and $\Omega_0=]0,1[\times]-1,0[$.
This test problem is taken from the benchmark suite of
Reference~\cite{dauge}.
For these calculations we use the Voronoi decompositions of
Figure~\ref{fig:LShape-Meshes}.
The convergence results relative to the first and third eigenvalue are
shown in Figure~\ref{fig:lshape:rates}.
For the first eigenvalue, we observe a lower rate of convergence that
is related to the fact the corresponding eigenfunction belongs to
$\HS{1+r}(\Omega)$, with $r=2/3-\epsilon$ for any $\epsilon>0$
(see~\cite{dauge}).
\begin{figure}
  \centering
  \begin{tabular}{cc}
    \begin{overpic}[scale=0.35]{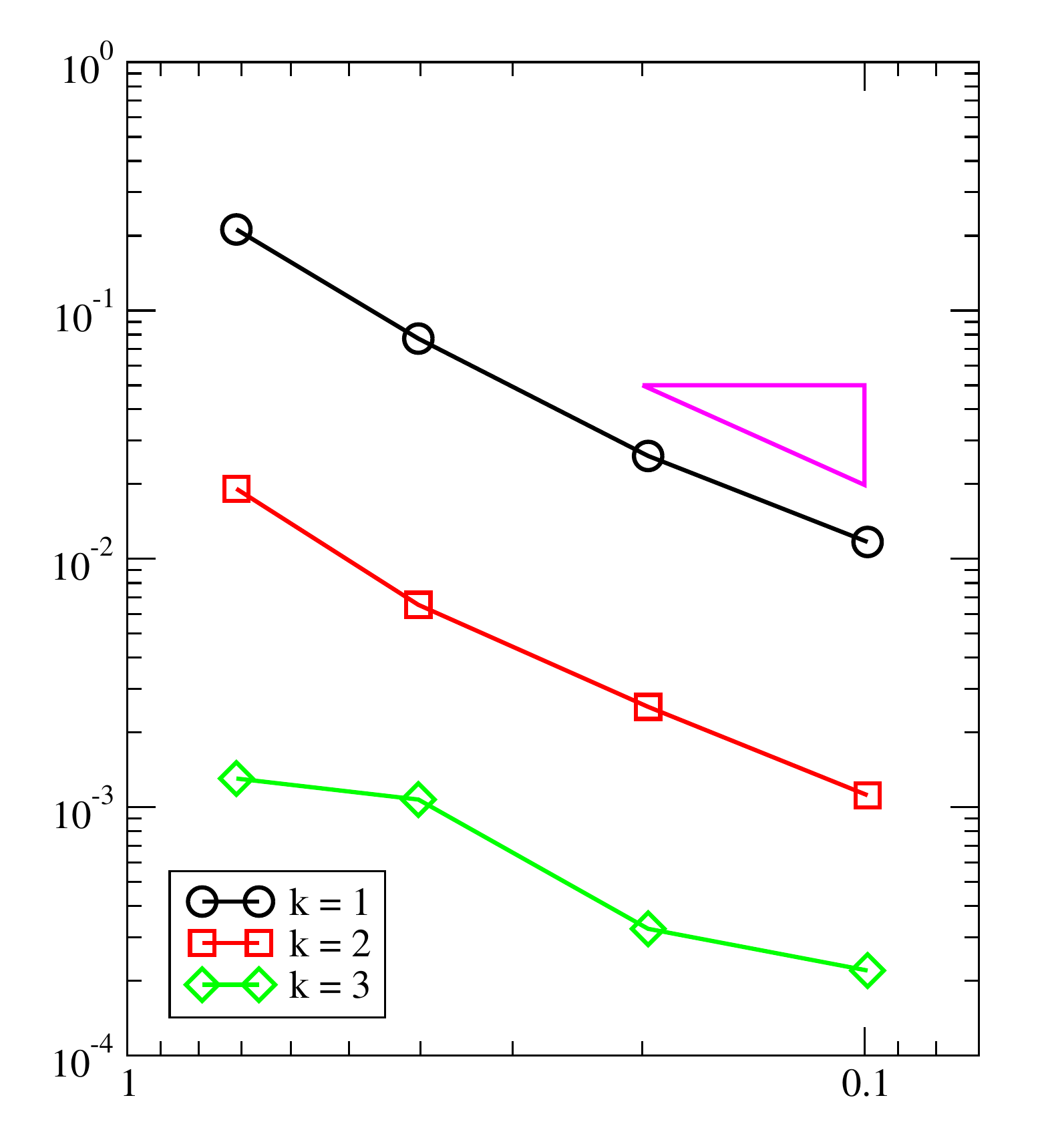}
      \put(-5,15){\begin{sideways}\textbf{Relative approximation error}\end{sideways}}
      \put(32,-5) {\textbf{Mesh size $\mathbf{\hh}$}}
      \put(62,69){\textbf{4/3}}
    \end{overpic} 
    &\qquad
    \begin{overpic}[scale=0.35]{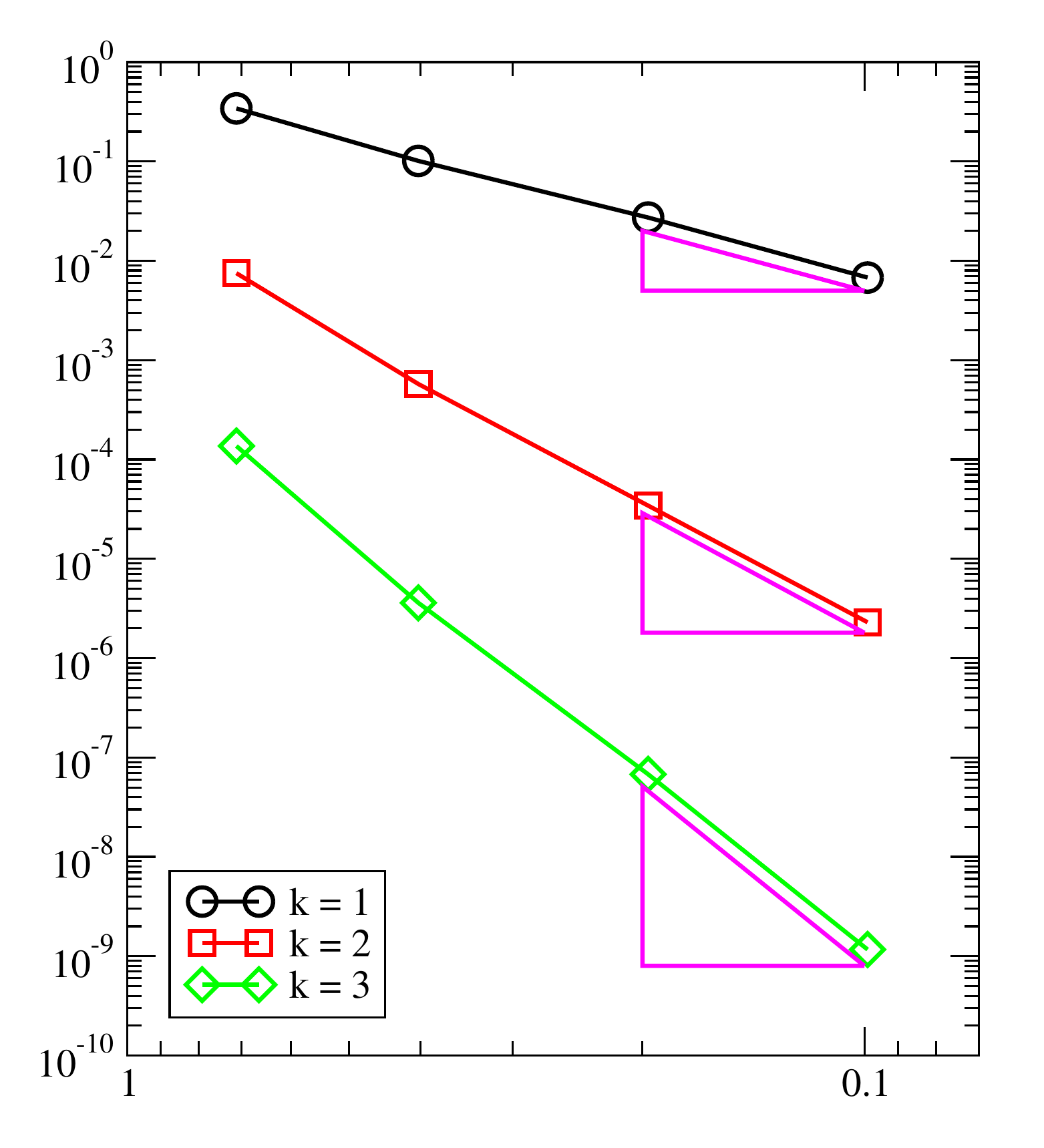}
      \put(-5,15){\begin{sideways}\textbf{Relative approximation error}\end{sideways}}
      \put(32,-5) {\textbf{Mesh size $\mathbf{h}$}}
      \put(51,76){\textbf{2}}
      \put(51,49){\textbf{4}}
      \put(51,22){\textbf{6}}
    \end{overpic}
  \end{tabular}
  \medskip
  \caption{Test Case~2: convergence curves of the first eigenvalue
    (left panel) and the third eigenvalue (right panel) using the
    nonconforming VEM space $\VS{\hh}_{k}$, $k=1,2,3$,}
  \label{fig:lshape:rates}
\end{figure}
Instead, the third eigenvalue is analytical and the optimal order of
convergence is obtained, which can be seen by comparing the slopes of
the error curves and the corresponding theoretical slopes reported on
the plot.
These results confirm the convergence analysis of the previous section
and the optimality of the method also on nonconvex domains using
polygonal meshes.

\section{Conclusions}
\label{sc:conclusions}

We analysed the nonconforming VEM for the approximation 
of elliptic eigenvalue problems. The nonconforming scheme, contrary to the 
conforming one, allows to use the same formulation both for the 
two- and the three-dimensional case. 
We proposed two different discrete formulations, which differ for the discrete 
form approximating the $L^2$-inner product. In particular, we considered 
both a nonstabilized form and a stabilized one. 
We showed that both formulations provide a correct approximation of the spectrum 
and we proved optimal \emph{a priori} error estimates for the approximations of 
eigenfunctions both in the $L^2$-norm and the discrete energy norm,  
and the usual double order of 
convergence of the eigenvalues. Eventually, we presented a 
wide set of numerical tests which confirm the theoretical results. 

\section*{Acknowledgements}
The work of the second author was partially supported by the
Laboratory Directed Research and Development Program (LDRD),
U.S. Department of Energy Office of Science, Office of Fusion Energy
Sciences, and the DOE Office of Science Advanced Scientific Computing
Research (ASCR) Program in Applied Mathematics Research, under the
auspices of the National Nuclear Security Administration of the
U.S. Department of Energy by Los Alamos National Laboratory, operated
by Los Alamos National Security LLC under contract DE-AC52-06NA25396.
The third author was partially supported by the European Research Council through the H2020
Consolidator Grant (grant no. 681162) CAVE - Challenges and Advancements in Virtual Elements. This support
is gratefully acknowledged.

\bibliographystyle{abbrv}
\bibliography{vem-nc-eigv}

\end{document}